\documentclass[10pt,a4paper,reqno]{amsart}
\makeatletter

\newcommand{\Rmnum}[1]{\expandafter\@slowromancap\romannumeral #1@}
\makeatother
\usepackage{jiuya}

\begin{document}
\title{The Sharp Erd\H{o}s-Tur\'an Inequality}
\author{Ruiwen Shu and Jiuya Wang}
\newcommand{\Addresses}{{
	    \footnotesize		
		Ruiwen Shu, \textsc{Mathematical Institute, University of Oxford, Oxford OX2 6GG, UK
		}\par\nopagebreak
		\textit{E-mail address}: \texttt{shu@maths.ox.ac.uk}	
		
		\bigskip
		\footnotesize		
		Jiuya Wang, \textsc{Department of Mathematics, University of Georgia, Boyd Graduate Studies Research Center, Athens, GA 30601, USA
		}\par\nopagebreak
		\textit{E-mail address}: \texttt{jiuya.wang@uga.edu}	
	}}
\maketitle	
	\begin{abstract}
Erd\H{o}s and Tur\'an proved a classical inequality on the distribution of roots for a complex polynomial in 1950, depicting the fundamental interplay between the size of the coefficients of a polynomial and the distribution of its roots on the complex plane.  Various results have been dedicated to improving the constant in this inequality, while the optimal constant remains open. In this paper, we give the optimal constant, i.e., prove the sharp Erd\H{o}s-Tur\'an inequality. To achieve this goal, we reformulate the inequality into an optimization problem, whose equilibriums coincide with a class of energy minimizers with the logarithmic interaction and external potentials. This allows us to study their properties by taking advantage of the recent development of energy minimization and potential theory, and to give explicit constructions via complex analysis. Finally the sharp Erd\H{o}s-Tur\'an inequality is obtained based on a thorough understanding of these equilibrium distributions.
	\end{abstract}
\pagenumbering{arabic}
\vspace{-0.5 cm}
\tableofcontents
\section{Introduction}
\subsection{Main Results and Consequences}
Erd\H{o}s and Tur\'an prove a beautiful result on the distribution of roots for polynomials $f(z) \in \C[z]$ in their 1950 work \cite{ET}, that is, roughly speaking, if $|f(z)|$ attains small value on the unit circle $|z|=1$, then firstly the magnitudes of all roots of $f(z)$ are close to $1$, secondly the angular distribution of the roots is close to equidistribution on $\bT = \mathbb{R}/\mathbb{Z}$. See \cite[Figure 1]{Sound} for a pictorial illustration. The statement on the magnitude of roots is a consequence of Jensen's formula in complex analysis, see e.g.\cite[Theorem 1]{Sound} for a compact treatment. The statement on the angular distribution requires much more work. 

The main goal of this paper is to give a sharp inequality characterizing this phenomenon of equidistribution in angles. Let $f(z) = \sum_{k=0}^{n} a_k z^k$ be a degree $n$ polynomial with complex coefficients and $a_0\neq 0$. We denote its roots by $r_j e^{2\pi i\theta_j}$ for $1\le j\le n$ with $\theta_j \in \bT =\mathbb{R}/\mathbb{Z}$. For $ \alpha\le \beta < \alpha+1$, we write $N_f(\alpha, \beta)$ to be the number of roots with $\theta_j \in [\alpha, \beta]$ when considered as a subset in $\bT$. The theorem relates two quantities of $f(z)$. We first define the \emph{discrepancy} of a polynomial $f$ to be
\begin{equation}
	\mathcal{D}[f]:= \max_{ \alpha\le \beta< \alpha+1} \frac{N_f(\alpha, \beta)}{n}-(\beta-\alpha),
\end{equation}
measuring the deviation of the angle distribution of roots away from the uniform distribution on the unit circle, then we define the \emph{height} of a polynomial by
\begin{equation}
	\mathcal{H}[f]:=\frac{1}{n} \log \frac{\max_{|z|=1}|f(z)|}{\sqrt{|a_0a_n|}},
\end{equation}
measuring how large $f$ is on the unit circle up to a normalizing factor. Our main theorem is the following,
\begin{theorem}[Sharp Erd\H{o}s-Tur\'an Inequality]\label{thm:main}
	For any polynomial $f(z)\in \C[z]$ with $f(0)\neq 0$, we have
	\begin{equation}\label{eqn:main}
		\mathcal{D}[f] \le \sqrt{2} \cdot \sqrt{\mathcal{H}[f]}.
	\end{equation}
	Moreover, this inequality is sharp. 
\end{theorem}
Theorem \ref{thm:main} gives a sharp improvement upon the original inequality of Erd\H{o}s and Tur\'an, where the constant $\sqrt{2}$ is replaced with $16$. By a family of polynomials constructed in \cite{AmMig}, it is clear that $\sqrt{\cH[f]}$ is the optimal scaling. Therefore the last unknown component is the constant, which is solved by Theorem \ref{thm:main}.

Historically, one of the motivations for Erd\H{o}s and Tur\'an to study this type of inequality is from number theory. The question of bounding the number of real solutions of a polynomial is of great interest. One perspective for the Erd\H{o}s-Tur\'an inequality is to give an error term for the angular distribution when the uniform distribution is considered as the main term in $n$
\begin{equation}
	N_f(\alpha, \beta) = (\beta-\alpha)\cdot n +O_f(\sqrt{n}).
\end{equation}
The dependence of $f$ in the error term can be bounded by $\sqrt{2}\cdot \log^{1/2} (\| f\|_{L^{\infty}(|z|=1)}/\sqrt{|a_0a_n|})$ by Theorem \ref{thm:main}. Bloch and P\'olya \cite{bloch} first investigated how to bound the number of real zeros of a polynomial $f(z)$ by the size of its coefficients. Schmidt \cite{schmidt} and Schur \cite{schur} refined the result by giving a bound in terms of a certain height just involving the degree $n$ and sum of coefficients. 
By applying the inequality to small intervals near the $x$-axis, Erd\H{o}s and Tur\'an recover the result of Schmidt and Schur as an immediate consequence. Moreover due to a different choice of height function, the upper bound of Erd\H{o}s and Tur\'an gives a slight improvement up to the constant. 
%
%
%
%
We denote $N_{\theta}(f)$ (respectively $N_{+}(f)$ and $N_{-}(f)$) to be the number of roots with angle $\theta$ (respectively positive and negative) of $f(z)\in \C[z]$. As a by-product of Theorem \ref{thm:main}, we also obtain sharp estimates for the number of real roots with a given sign in terms of $\mathcal{H}[f]$. 
\begin{theorem}[Sharp Estimates for Signed Real Roots]\label{thm:real-root}
	For any polynomial $f(z)\in \C[z]$ with $f(0)\neq 0$ and any $\theta \in \bT$, we have the sharp inequality
	\begin{equation}
		N_{\theta}(f) \le \sqrt{2} \cdot \sqrt{\cH[f]} \cdot n.
	\end{equation}
	In particular, we have the sharp estimate 
	\begin{equation}
	N_+(f)\le \sqrt{2} \cdot \sqrt{\mathcal{H}[f]}\cdot n, \quad \quad N_-(f)\le \sqrt{2}\cdot \sqrt{\mathcal{H}[f]}\cdot n,
	\end{equation}
	for the number of positive real roots of $f(z)$ and the number of negative real roots of $f(z)$.
\end{theorem}
Although the sharp estimate for the total number of real roots (including both positive and negative) does not follow directly from Theorem \ref{thm:main}, we will prove in a separate forthcoming note the sharp estimate for the number of real roots in terms of $\cH[f]$ from a similar approach we prove Theorem \ref{thm:main}. 

This inequality also has close connections to the theory of complex analysis and harmonic functions. In  \cite{Ganelius}, Ganelius first made the connection of this inequality to harmonic functions. Let $u$ and $v$ be a pair of conjugate harmonic functions inside the unit circle. Ganelius proved a theorem bounding the variation of $v$ by the maximal value of $u$ and $\partial v/\partial \theta$ in the unit disk. He then applies it  to deduce the inequality of Erd\H{o}s and Tur\'an with a better constant. Mignotte \cite{Mignotte} slightly refined the result by replacing the upper bound of $u$ with the integral of $u_+$ (the positive part of $u$) on the unit circle. We manage to show that Ganelius' conjugate function problem turns out to be equivalent to the Erd\H{o}s and Tur\'an inequality, thus Theorem \ref{thm:main}, in turn, implies a sharp estimate on harmonic functions. To our knowledge, such an equivalence has been neither noticed nor utilized before.
\begin{theorem}[Sharp Estimates for Harmonic Functions]\label{thm:conjugate-functions}
	Let $f(z)= u(z)+i v(z)$ be an analytic function in $|z|<1$ with $f(0)=0$. Suppose $u\le H$ and $\partial v /\partial \theta \le K$ in $|z|<1$ where $H, K>0$. Then we prove the following sharp estimate 
	\begin{equation}\label{thm_gane_1}
		|v(z_1)-v(z_2)|\le \sqrt{2\pi} \cdot \sqrt{HK},  \quad \quad \text{ for } |z_1|, |z_2| <1.
	\end{equation}
\end{theorem}

\subsection{Questions and History}
In an earlier work \cite{ETReal}, Erd\H{o}s and Tur\'an studied the relation between the discrepancy of a distribution $\{ z_i \} \subseteq [-1,1]$ and the values of $f(z) = \prod_j (z-z_j)$ on $[-1,1]$. This might be the origin of the inequality we are considering in this paper. Actually we can consider Theorem \ref{thm:main} as an analogue by replacing the interval $[-1,1]$ with the unit circle. Indeed 
%
it has been observed by \cite{schur} that in order to prove Theorem \ref{thm:main} it suffices to consider polynomials with unimodular roots. Therefore the question at the core is fundamentally the following: given $f(z) = \prod_j (z- z_j)$ with $z_j= e^{2\pi i \theta_j}$, what is the minimal value of $\max_{|z|=1} |f(z)|$?

With such a simple and clean form of this question, the inequality of Erd\H{o}s and Tur\'an inevitably has connections to many questions in different areas. Besides the application to bounding the number of real roots of polynomials, in \cite{ET} this inequality is also applied to recover the result of Jentzsch \cite{jent} and Szeg\H{o} \cite{szego} in complex analysis that every point of the circle of convergence for a power series is a limit point of zeros of its partial sums. It is noted by \cite{ETReal} that such an equidistribution theorem has a potential theoretic characterization that dates back to Hilbert in 1897. We can imagine that the roots of a polynomial are negatively charged particles and they repel each other with a force, and equidistribution corresponds to a state where the energy is minimized. Indeed, one of our main tools in this paper will be to apply recent development in modern potential theory and energy minimization to review this old question. Potential theoretic approach has also appeared in studies of closely related questions, for example in Bilu's result on equidistribution of roots for $f(z)\in \mathbb{Z}[x]$ \cite{Bilu}, see \cite{GranExpo} for a nice exposition. This result also has applications towards small points on abelian varieties \cite{zhang1,zhang}.  In \cite{Fekete}, this inequality is applied to study distribution of roots of Fekete polynomials, which has applications on the distribution of $L$-function values. 

After the fundamental works of Erd\H{o}s and Tur\'an \cite{ETReal,Erdosearly,ET} in the 1940s, questions on equidistribution and discrepancy have been generalized in various directions. We refer interested audiences to \cite{Dbook} for a collection of results on generalizations. We give some examples, far from being exhaustive, in the following. Instead of using this notion of height $\cH[f]$, one can bound the discrepancy $\cD[f]$ using other quantities. Mignotte in \cite{Mignotte} gives a bound of $\cD[f]$ in terms of $h[f]:=\frac{1}{n}\int_{\mathbb{T}}(\log\frac{|f(z)|}{\sqrt{|a_0a_n|}})_+ \rd{\theta} \le \mathcal{H}[f]$, which is equivalent to take $L^1(\bT)$ norm of $\log |f(z)|/\sqrt{|a_0a_n|}$ whereas $\cH[f]$ is close to a $L^{\infty}$ norm of $\log |f(z)|/\sqrt{|a_0a_n|}$. See also \cite{Sound,aim} for inequalities involving $h[f]$. In \cite{Bilu}, Bilu uses the Mahler measure for an integral polynomial $M[f] =\exp( \int_{\bT} \log |f(z)| \rd{\theta})$ to bound the discrepancy. This is also generalized to higher dimensions. H\"uesing \cite{huesing} has used energy to bound discrepancy for signed measures. Instead of working over $\bT$, \cite{curve,and97,andrievskii1999erdos} studies the discrepancy of $\mu$ on a general quasiconformal curve and bound the discrepancy in terms of the value $\log |f(z)|$ on the curve.  On the segment $[-1,1]$, \cite{blatt1992distribution,totik1993distribution} studies the equidistribution distribution of simple roots of a polynomial. Results in higher dimensions are also studied in \cite{sjogrenhighdim,gotz}. More recently, in \cite{steinerwass}, Steinerberger studies the bound for Wasserstein distance from the uniform distribution, which generalizes the notion of discrepancy on $\bT$. We also mention works on roots distribution of polynomials related to Erd\H{o}s-Tur\'an inequality. For more refined discussion on the modulus of roots for polynomials, see \cite{Tamas}. Erd\H{o}s-Tur\'an inequality is also applied to study roots of polynomials with small coefficients \cite{Poonen}, Fekete polynomials \cite{Fekete}, and random polynomials \cite{Hugh,Eofdisc}.

Despite the huge amount of generalization, the original form of Erd\H{o}s and Tur\'an inequality hasn't been improved much ever since the early 1950s. As commented in \cite{Dbook}, the constant remains the only factor that is not sharp in this inequality. Early work of Ganelius \cite{Ganelius} improved this constant to $\sqrt{2\pi/k} \approx 2.5619$ where $k$ is the Catalan constant. In a very recent work \cite{Sound}, Soundararajan uses $h[f]$ to bound $\mathcal{D}[f]$ by an elegant Fourier argument. As a consequence it sharpens the constant to $8/\pi \approx 2.5464$. In the AIM workshop in 2021, a group involving the second author sharpens the tools of Soundararajan. In \cite{aim}, the group improves the constant to $4/\sqrt{\pi} \approx 2.2567$ by giving a complete solution of a Fourier optimization problem. This seems to be the best one can extract from this Fourier argument. 
%
 We remark that optimizing the inequality with respect to $h[f]$ will be a different problem, as the optimal polynomial will probably be different. It is observed by Mignotte \cite{Mignotte}, an unpublished note of Soundararajan and the recent AIM work \cite{aim}, that the optimal constant for bounding $\mathcal{D}[f]$ by $h[f]$ is at least about $1.7593$, which is achieved when $f(z) = (z-1)^n$. We mention that Mignotte and Amoroso \cite{AmMig} constructed a sequence of polynomials with $\mathcal{H}[f]/\mathcal{D}[f]^2$ approaching to $1/2$, indicating that 
the optimal constant in the Erd\H{o}s-Tur\'an inequality cannot be smaller than $\sqrt{2}$, which is exactly the constant we prove in Theorem \ref{thm:main}.

\subsection{Methods}
\subsubsection{Schur's Observation}
A crucial simplification for Theorem \ref{thm:main}, Theorem \ref{thm:real-root} and Erd\H{o}s-Tur\'an's original inequality is one observation due to Schur. In \cite{schur} Schur showed that to state Theorem \ref{thm:main} for all polynomials, it suffices to prove it for $f(z)$ where all roots are unimodular. Indeed, given $f(z) = a_n\prod_j(z-r_j e^{2\pi i \theta_j})$ and $\tilde{f}(z) = \prod_j (z- e^{2\pi i\theta_j})$, one can observe that
\begin{equation}\label{eqn:schur}
 \frac{|f(z)|}{\sqrt{|a_0a_n|}}\le |\tilde{f}(z)|,  \text{ for } |z| = 1,
\end{equation}
%
whereas the discrepancy remains the same, see for example \cite[Section 3]{Sound}. Therefore to study the optimal constant, it suffices to focus on $f$ with all roots on the unit circle $\mathbb{T}$. After normalization to monic polynomials, we can assume $|a_0|=a_n=1$, so $\mathcal{H}[f]$ is simply the maximal value of $ \log(|f(z)|^{1/n})$ on unit circle $\mathbb{T}$.

To demonstrate the phenomenon quantified by this inequality, we look at two contrasting polynomials $(z-1)^n$ and $z^n-1$ with all roots on the unit circle. The polynomial $f(z) = (z-1)^n$ is the most discrepant polynomial with degree $n$ since $z=1$ is a root of multiplicity $n$ and the angular distribution is supported on a single point $\theta = 0$, which is far away from being equidistributed. Meanwhile $|f(z)|$ attains the maximal value $2^n$ at $z = -1$, which is exponentially large in the degree, and $\mathcal{D}[f]$ is achieved at $\alpha=\beta =0$ with the value $\mathcal{D}[f]=1$. Theorem \ref{thm:main} predicts $1\le \sqrt{2\cdot \log 2}\approx 1.18$. The other polynomial $f(z) = z^n-1$ has all roots that are $n$-th roots of unity, and is the most evenly distributed polynomial with degree $n$. The discrepancy is $\mathcal{D}[f] = 1/n$ and $\max_{|z|=1} |f(z)|$ is as small as $2$. In this case Theorem \ref{thm:main} says $1/n \le \sqrt{2} \sqrt{\log 2/n}$. This also implies that $\cH/\cD^2$ cannot be bounded from above.

%

\subsubsection{From Discrete to Continuous}
Although this question is originally stated as one about roots distribution of polynomials, it can be considered as a question just about distributions of $n$ points on the unit circle, thanks to Schur's observation and the convenient translation between polynomials and their roots. Since the degree of $f$ can be arbitrarily large, the extreme case will clearly take place when $n$ is large enough. This leads to one of the main ideas in proving Theorem \ref{thm:main}, that is, to solve the discrete problem using continuous methods. Instead of considering discrete measures supported on $n$ points, we will enlarge the class of distributions of consideration to all probability measures $\cM$ on $\bT$.

Under the language of probability measures, we can describe the root distribution of $f(z)$ by an \emph{empirical measure} of $n$ points 
\begin{equation}
	\rho_f = \frac{1}{n} \sum_{j} \delta_{\theta_j},
\end{equation}
where $\delta_{\theta_j}$ is the Dirac function at $\theta_j$.
The concepts of discrepancy $\mathcal{D}[f]$ and height $\mathcal{H}[f]$ of a polynomial can now be naturally extended to an arbitrary probability measure $\rho$ by defining
\begin{equation}\label{eqn:define-D-H-continuous}
	\mathcal{D}[\rho]:= \sup_I \int_I (\rho-1) \rd\theta, \quad \quad \mathcal{H}[\rho]:= -\underset{\mathbb{T}}{\ess\inf} (W*\rho),
\end{equation}
where $I\subset \mathbb{T}$ is any closed interval on the unit circle and $W(x) = -\log|1-e^{2\pi i x}| =- \log|2\sin(\pi x)|$. 

We then extend the question to an optimization problem of $\mathcal{H}[\rho]/\mathcal{D}[\rho]^2$ for general $\rho \in\cM$. Notice that we can also conveniently approximate a probability distribution by empirical measures with large $n$, see Section \ref{sec:approximation} for the details. Therefore it suffices for us to prove the following.  
\begin{theorem}\label{thm:probability-distribution}
	For all probability measure $\rho \in \mathcal{M}$ on $\bT$, we have the sharp inequality
	\begin{equation}\label{eqn:probability-distribution}
		\mathcal{D}[\rho] \le \sqrt{2} \cdot \sqrt{\mathcal{H}[\rho]}.
	\end{equation}
\end{theorem}

\subsubsection{Potential Theory}
As an optimization problem, we see that the optimizer is not among the trivially extremal ones, in view of the two previous examples. So one of the main challenges is to construct distributions with the optimal value. The main tool for us to characterize the shape of the optimal distribution is potential theory and energy minimization. 

We start by giving a potential theoretic description of Theorem \ref{thm:probability-distribution}. The reader can also find basic potential theoretic facts in \cite{Saffbook}. As we noted before $\mathcal{H}[\rho] =-\ess \inf (W*\rho)$, where $W(x) = -\log|2\sin(\pi x)|$. We can view $W$ as a pairwise interaction potential among particles. Its negative sign stands for the repulsion when two particles are close. For a probability distribution $\rho\in \cM$, we denote $V[\rho]:=W*\rho$ to be the potential field generated by $\rho$ with the logarithmic potential $W$, then the total potential energy of $\rho$ is
\begin{equation}
	\mathcal{E}[\rho]:=\frac{1}{2}\int_{\mathbb{T}} (W*\rho)\cdot \rho \rd \theta.
\end{equation}
Since $W(x)$ is mean zero over $\mathbb{T}$, we see
$V[1_{\mathbb{T}}] = W*1_{\mathbb{T}}\equiv 0$ and it is easy to show that $1_{\mathbb{T}}$ is the unique probability measure with $\mathcal{E}[1_{\mathbb{T}}] \le 0$. Such a probability measure $\rho$ that minimizes the total energy is called the \emph{energy minimizer} or \emph{equilibrium}. On the other hand, it was shown in \cite{Wass} that $\mathcal{D}[\rho]:= 2\cdot d_{\infty}(\rho, 1_{\mathbb{T}})$, which is the Wasserstein-infinity distance between $\rho$ and the uniform distribution $1_{\mathbb{T}}$. We mention that Wasserstein distances are the most natural distance in studies like optimal transport. 

Then (\ref{eqn:probability-distribution}) as
\begin{equation}
	 d_{\infty}(1_{\mathbb{T}}, \rho)\le \frac{1}{\sqrt{2}} \cdot \| (V[1_{\mathbb{T}}]-V[\rho])_+ \|_{\infty}^{1/2} =\frac{1}{\sqrt{2}} \cdot \sqrt{-\ess \inf V[\rho]}.\\	  	 
\end{equation}
Now we can interpret Theorem \ref{thm:probability-distribution} as saying if the potential generated by $\rho$ is close to that of the equilibrium, then $\rho$ is close to the equilibrium under the $d_{\infty}$ metric. In a forthcoming work, the authors will exactly use this formulation and give an application of Erd\H{o}s-Tur\'an type inequality towards stability of energy minimizers. 

\subsubsection{Energy Minimization}
In this part, we characterize the property for the optimal distributions. We first restrict this optimization problem of $\cH/\cD^2$ from $\rho \in \mathcal{M}$ to $\rho \in \cM_{\cD\ge d}$, where $\mathcal{D}[\rho]\ge d$. Then the optimal value for $\cH/\cD^2$ can be approximated as $d$ approaches $0$,
%
%
%
\begin{equation}
	\inf_{\rho \in \cM} \cH[\rho]/\cD[\rho]^2 = \lim_{d\to 0+} \inf_{\rho \in \cM_{\cD\ge d}} \cH[\rho]/\cD[\rho]^2.
\end{equation}

Then in Section \ref{sec:minimizer-characterization}, our main goal is to translate the problem of optimizing $\mathcal{H}[\rho]/\mathcal{D}[\rho]^2$ with fixed $d$ into an energy minimization problem with certain fixed external potentials $U$. Given an external potential $U$, we define 
\begin{equation}
V_U[\rho]:= U+ W*\rho, \quad\quad \mathcal{E}_U[\rho]:= \frac{1}{2}\int_{\mathbb{T}} (W*\rho)(x)\cdot \rho(x) \rd x + \int_{\mathbb{T}} U(x)\rho(x) \rd x,
\end{equation}
for arbitrary $\rho\in \cM_{m}$, which is the set of measures over $\bT$ with total mass $m$. We study the properties of energy minimizers in Section \ref{ssec:energy-min}. The theory of energy minimization guarantees that for fixed $m$ and $U$, if $W$ and $U$ are both nice, then there exists a unique $\rho \in \cM$ with minimal $\cE_U$. Moreover, such an energy minimizer $\rho$ always satisfies the condition
\begin{equation}\label{eqn:sediment}
	V_U[\rho](x) \le \ess \inf (V_U[\rho]),  \quad \quad \forall x \in \supp \rho.
\end{equation}
We say that a measure $\rho$ satisfying \ref{eqn:sediment}  is \emph{sediment} with respect to $U$. Intuitively this condition is saying that all of the mass of $\rho$ is resting at the bottom of the total potential $V_U[\rho]$ generated by $\rho$ and $U$. It is shown that such a sediment distribution $\rho$ is unique. Moreover among all $\rho \in \cM_m$, the unique sediment $\rho$ sees the highest bottom, i.e. has the largest $\ess\inf V_U[\rho]$. See Proposition \ref{prop:energy-min} for detailed treatment.

On the other hand, we develop a transport plan called \emph{microscopic diffusion} in Section \ref{ssec:micro-diff} to study the minimizer of $\cH/\cD^2$ in $\cM_{\cD\ge d}$. It shows that if a certain $\rho\in \cM_{\cD\ge d}$ has $(W*\rho)(x_0)> \ess \inf (W*\rho)$ for some $x_0\in \supp \rho$, then we can always construct a local diffusion such that $W*\rho$ increases for $x$ away from $x_0$. Therefore as long as the local operation stays away from the endpoints of $I$ witnessing $\cD$, we will be able to decrease $\cH/\cD^2$. The key of this argument lies in the convexity of our potential function $W(x) = -\log|2\sin(\pi x)|$. What it implies is an ultimate surprise and stunning connection, the minimizer of $\cH/\cD^2$ in $\cM_{\cD\ge d}$ is the unique sediment distribution with respect to the external potential  $U$ generated by Dirac mass(es) at endpoints of $I$. This leads to the characterization of minimizer(s) for $\cH/\cD^2$ in Theorem \ref{thm:main-characterization}. For this part, our method will also work in general for $\cH/\cD^{\alpha}$ for any $\alpha\ge 0$ and for a more general class of interaction potentials $W$.

\subsubsection{Construction of Minimizers}
Next, our goal is to construct explicit distributions satisfying the characterization in Theorem \ref{thm:main-characterization}. For each Dirac mass(es) configuration $\rho_d$ (position $M$ and mass $m$), there exists a unique sediment distribution with respect to $U = W*\rho_d$. Therefore it suffices to construct sediment distributions for each $M$ and $m$ in order to find a pool of candidates for the minimizer. In Section \ref{sec:construction-stationary}, we start with constructing a larger family of distributions, called \emph{stationary} distribution with respect to $U$, which satisfies
\begin{equation}
	V_U[\rho]'(x) = 0, \quad \quad x \in \supp \rho.
\end{equation}
Intuitively these are the distributions where all mass of $\rho$ are stationary and has total force $0$. The construction heavily depends on the fact that the kernel of the Hilbert transform is derivative of logarithm. Due to this elegant connection between logarithmic potential and Hilbert transform, we are able to construct these distributions $\rho$ as the real part $\Re(g)$ of some nice analytic functions $g(z)$, while $\Im(g) =-\frac{1}{\pi} (W* \Re(g))'=-\frac{1}{\pi} V_U[\rho]'$ is the derivative of the generated potential. By constructing analytic functions $g(z)$ which are either totally real or totally imaginary, we thus obtain $\rho$ where $(W* \Re(g))'=V_U[\rho]'$ is $0$ on $\supp \rho$. In order to evaluate the minimal value of $\cH/\cD^2$ conveniently, we also give the analogous construction for a similar class of distributions over $\R$, using the Hilbert transform over $\R$. 

From this construction, we obtain a family of distributions over $\bT$ parametrized by $M$, $m$ and another parameter $m_1 = \int_{(-M, M)} \rho(x) \rd{x}$. By imposing the condition that $\rho$ is sediment with respect to its Dirac masses configuration, we then get rid of the parameter and obtain a family just parametrized by $M$ and $m$. We can discuss distributions over $\R$ similarly, and obtain a class of distributions called \emph{admissible distribution}, also parametrized by two parameters. 

We expect this construction to be useful for various types of optimization problems.
%
\subsubsection{Estimation of Min Value}
Now given this two-parameter family of energy minimizers ${\rho}_{M,m}$ where the external potentials $U = W*\rho_{d}$ with $\rho_{d} = m(\delta_M + \delta_{-M})$. Our goal is to show that $1/2$ is a lower bound for   $\cH[\rho]/\cD[\rho]^2$ for every $\rho = \rho_{m, M}$. 

In order to reduce the complexity of the problem over $\mathbb{T}$, we make a comparison between sediment distributions over $\bT$ and admissible distributions over $\R$. A natural operation to get a distribution over $\bT$ from a distribution $\mu$ over $\R$ is to take the periodization $\rho_{\circ}$, see \eqref{rhocirc0},\eqref{rhocirc}. Notice that both families are parametrized by their Dirac mass(es) configurations. Therefore we can associate each sediment distribution $\rho$ with a periodization $\rho_{\circ}$ of a unique admissible distribution $\mu$ over $\R$. Then in Section \ref{sec:R-to-T}, we firstly relate two functionals $\tilde{\cH}$ and $\tilde{\cD}$ for distributions over $\R$ (see Section \ref{sec:min-value-R} for the definitions of $\tilde{\cH}$ and $\tilde{\cD}$ for distributions over $\R$ ) to $\cH$ and $\cD$ over $\bT$, and prove the first comparison theorem
\begin{equation}
	\frac{	\tilde{\cH}[\mu]}{\tilde{\cD}[\mu]^2} \le  \frac{\cH[\rho_{\circ}]}{\cD[\rho_{\circ}]^2},
\end{equation}
for periodizations in Section \ref{ssec:periodization}, see Theorem \ref{prop:rho-circ-H-D}. This reduces the optimization problem over $\bT$ to an optimization problem over $\R$, where a great benefit is that we can get rid of one more parameter due to scaling invariant, see Section \ref{ssec:dimension-reduction}. We solve this new optimization problem for admissible distributions over $\R$ in Section \ref{sec:min-value-R} and show that 
\begin{equation}
\frac{1}{2}\le 	\frac{\tilde{\cH}[\mu]}{\tilde{\cD}[\mu]^2},
\end{equation}
for admissible $\mu$. Finally in Section \ref{sec:R-to-T}, via a technical convexity argument, see Section \ref{ssec:compare1}, we prove the second comparison theorem 
\begin{equation} 
\frac{\cH[\rho_{\circ}]}{\cD[\rho_{\circ}]^2} \le \frac{\cH[\rho]}{\cD[\rho]^2},
\end{equation}
where $\rho_{\circ}$ and $\rho$ are both distributions over $\bT$ and share the Dirac masses configuration, see Theorem \ref{thm_Dcomp}. This concludes the chain of comparisons and proves the inequality in Theorem \ref{thm:probability-distribution}. 

\subsection{Notations and Preliminaries}\label{sec:notation}
We will denote the torus $\mathbb{T} = \mathbb{R}/\mathbb{Z}$, and use a real number $x$ to indicate a point in $\mathbb{T}$. Similarly, we will use an interval $[a,b]\subseteq \mathbb{R}$ with length less than 1 to indicate an interval $I$ in $\mathbb{T}$. For an interval $I$ in $\bT$, we use $I^c$ to denote its complement. 

By measures over $\bT$, we always mean Borel measures on $\mathbb{T}$. We denote $\cM_{m}$ to be the set of all measures $\rho$ with total mass $\int_{\mathbb{T}} \rho = m$, thus $\cM = \cM_{1}$ is the set of all probability measures on $\mathbb{T}$. For a point $x\in \bT$, we will denote a Dirac function located at $x$ by $\delta_x$. A sequence $\{\rho_n\}$ of probability measures \emph{weakly converges to} $\rho$ if for all $f\in C(\mathbb{T})$, we have $\lim_{n\to\infty} \int_{\mathbb{T}} f\rd \rho_n = \int_{\mathbb{T}} f \rd \rho$, we will write $\rho_n \wc \rho$. Since $\bT$ is compact, it is a well known fact that every sequence of probability measures has a weakly convergent subsequence. 

We will always use $W$ to denote a potential function that describes the pairwise interaction between particles. In some occasions, there is also an \emph{external potential} which we denote by $U$. For $\rho \in \cM$, we denote $V[\rho]:= W*\rho$ to be the potential generated by $\rho$ and $V_U[\rho]:= W*\rho + U$ to be the total potential. 

For $\rho \in \cM$, for  $W$ a potential function, several important functionals we use are
$$\cD[\rho]:= \sup_I \int_I (\rho-1) \rd{x}, \quad \cH_W[\rho]:= -\ess \inf_{\bT} (W*\rho), \quad \cG[\rho]:= \frac{\cH_W[\rho]}{\cD[\rho]^2},$$
where $\ess \inf (f)  = \sup \{ a \in \R: |\{ x\in \bT: f(x)<a \}|= 0 \}$ is the essential infimum of $f$. We will suppress the dependence of $W$ in $\cH_W[\rho]$ when there is no confusion. We denote
\begin{equation}
	\cM_{\cD\ge d}:=\{\rho\in\cM:\cD[\rho]\ge d\},
\end{equation}
for some positive real number $0<d\le 1$.

For us, the Fourier transform (respectively Fourier coefficients) of $u$ is defined by
\begin{equation}
	\hat{u}(\xi) = \int_{\mathbb{R}}u(x)e^{-2\pi i x\xi}\rd{x},\quad \xi\in\mathbb{R}, \quad\quad 	\hat{u}(k) = \int_{\mathbb{T}}u(x)e^{-2\pi i xk}\rd{x},\quad k\in\mathbb{Z}
\end{equation}
respectively when $u$ is a function over $\R$ and $\bT$.

As an auxillary mollifier, we define $\psi$ as
\begin{equation}\label{psialpha}
	\psi(x) = \max\{1-|x|,0\},\quad \psi_a(x) = \frac{1}{a}\psi(\frac{x}{a}),\,a>0
\end{equation}
on $\mathbb{R}$. It is even and supported on $[-a, a]$. For $a\le 1/2$, $\psi_a$ can be interpreted as a function on $\mathbb{T}$ via identifying $\mathbb{T}$ with $[-1/2,1/2)$, and then $\hat{\psi}_a(k) = \hat{\psi}(ak)$ for $k\in \mathbb{Z}$ where $\hat{\psi}$ is the Fourier transform of $\psi$ on $\mathbb{R}$. 

\subsection{Organization of Contents}
In Section \ref{sec:minimizer-characterization}, we characterize minimizer(s) of $\cH/\cD^2$ in the class of probability measures $\cM_{\cD\ge d}$ for any $d>0$. We prove our main result Theorem \ref{thm:main-characterization} by relating minimizer(s) of $\cH/\cD^2$ and minimizer(s) of potential energy with logarithmic interaction and external potentials. Such a connection is established by showing that both minimizers are sediment distributions in Section \ref{ssec:micro-diff} and Section \ref{ssec:energy-min} respectively.

In Section \ref{sec:construction-stationary}, we construct sediment distributions, see in Propositions \ref{prop_rhoi}, \ref{prop_rhoii} and \ref{prop:2-parameter}. We achieve so by constructing a larger class of distributions by using the property of Hilbert transform. Meanwhile we also construct analogue distributions over $\R$ in Section \ref{ssec:stationary-R} and Section \ref{ssec:admissible-R}.

In Section \ref{sec:min-value-R}, we formulate an optimization problem over $\R$ analogous to minimizing $\cH/\cD^2$ over $\bT$, and then solve the minimal value for distributions constructed in Section \ref{ssec:admissible-R}. 

In Section \ref{sec:R-to-T}, we solve the optimization problem over $\bT$ by relating it to the optimization problem over $\R$ via two comparison arguments, Theorem \ref{prop:rho-circ-H-D} in Section \ref{ssec:periodization} and Theorem \ref{thm_Dcomp} in Section \ref{ssec:comparison2}.

Finally, we conclude the proofs of Theorems \ref{thm:main} and \ref{thm:probability-distribution} in Section \ref{ssec:final}, Theorem \ref{thm:real-root} in Section \ref{ssec:real-final}, and Theorem \ref{thm:conjugate-functions} in Section \ref{sec:conjugate-function}.

\vspace{0.5 cm}
\begin{center}
	\textbf{ \large Acknowledgement}
\end{center}

The first author was supported in part by NSF and ONR grants DMS1613911 and N00014-1812465. The first author was supported by the Advanced Grant Nonlocal-CPD (Nonlocal PDEs for Complex Particle Dynamics: Phase Transitions, Patterns and Synchronization) of the European Research Council Executive Agency (ERC) under the European Union's Horizon 2020 research and innovation programme (grant agreement No. 883363).

The second author would like to thank Theresa Anderson, Frank Thorne and Trevor Wooley for organizing the American Institute of Mathematics (AIM) workshop \emph{Arithmetic Statistics, Discrete Restriction, and Fourier Analysis} in 2021, and  Micah B. Milinovich and Emanuel Carneiro for introducing this problem to our knowledge. The second author would like to thank Emanuel Carneiro, Mithun Das, Alexandra Florea, Angel V. Kumchev, Amita Malik, Micah B. Milinovich and Caroline Turnage-Butterbaugh for helpful conversations. The authors would like to thank Jos\'e A. Carrillo and Dino J. Laurenzini for suggestions on earlier drafts. 

\section{Minimizer in $\cM_{\cD\ge d}$}\label{sec:minimizer-characterization}
In this section, we will study minimizer(s) of $\cG$ among probability measures over $\mathbb{T}$ with $\cD[\rho]\ge d$. In particular, for fixed $0<d<1$, we are interested in what the value $\inf_{\cM_{\cD\ge d} } \cG$ is, and whether it can be achieved and whether it is unique, and if it can be achieved or approximated how the minimizer(s) look like. 

We will focus on the last two questions above in this section, and save the discussion on the minimal value in Section \ref{sec:min-value-R} and Section \ref{sec:R-to-T}. Instead of just working with the functional $\cG[\rho] = \cH[\rho]/\cD[\rho]^2$ where $\cH[\rho] = -\ess \inf_{\mathbb{T}} (W*\rho)$ with $W = -\log |2\sin(\pi x)|$, in this section we will work in general with functional in the form
\begin{equation}
	\cG_{\alpha}[\rho]:= \cH[\rho]/\cD[\rho]^{\alpha},
\end{equation}
where $\alpha\ge 0$ is an arbitrary real number and $W: \mathbb{T}\to (-\infty, \infty]$ is any potential function satisfying the following assumptions:
\begin{itemize}
	\item {\bf (H1)}: $W\in L^1(\mathbb{T})\cap C^2(\mathbb{T}\backslash\{0\})$ with $\int_\mathbb{T}W\rd{x} = 0$.
	\item {\bf (H2)}: $W$ is an even function: $W(x)=W(-x)$.
	\item {\bf (H3)}: $\lim_{x\rightarrow0}W(x)=W(0)=\infty$.
	\item {\bf (H4)}: there exists a constant $C_1>0$ such that $W''(x)\ge C_1$ for $x\in \mathbb{T}\backslash\{0\}$.
	\item {\bf (H5)}: there exists a constant $C_2>0$ such that for all $0<r<1/2$ and $x\in \mathbb{T}$,  
\begin{equation}
	\frac{1}{2r}\int_{x-r}^{x+r}W(y)\rd{y}-2\inf W \le C_1 (W(x)-2\inf W).
\end{equation}
\end{itemize}

It is straightforward to verify that $W(x) = -\log|2\sin (\pi x)|$ satisfies {\bf (H1)} - {\bf (H5)}. We compute that $W'' = 2\pi/ (1- \cos(2\pi x)) =\pi /\sin ^2\pi x \ge \pi$. 
It follows from Jensen's formula that $W$ has mean value $0$. For {\bf (H5)}, it suffices to notice that $- c\log|x| \le W(x)-2 \inf W \le -C \log |x|$ for some positive constants $c$ and $C$ for $x\in [-1/2, 1/2]$ and $\log |x|$ satisfies the condition via a computation. 

Our main theorem for this section is the following:
\begin{theorem}\label{thm:main-characterization}
	Assume $W$ satisfies {\bf (H1)}-{\bf (H4)}. Let $\alpha\ge 0$, $0<d\le 1$. Then 
	\begin{enumerate}
		\item[\textnormal{(i)}] There exists a minimizer $\rho \in \cM_{\cD\ge d}$ for $\cG_{\alpha}$ that is even and $\cD[\rho] = \int_{I} (\rho-1) $ with $I=[-M,M]$ for some $0\le M<1/2$.
		\item[\textnormal{(ii)}] Let $\rho$ be a minimizer as in $1$. Then $\rho$ is in the form
		\begin{equation}
			\rho=m(\delta_M+\delta_{-M}) + \rho_1 
		\end{equation}
	   for some $0<m\le 1/2$, and $\rho_1 \in \cM_{1-2m}$ is supported on two (possibly empty) closed intervals $J\subsetneq I^c$ and $K\subsetneq I$.
      \item[\textnormal{(iii)}]
       Let $\rho$ be a minimizer as in $1$. Then $\rho$ is a sediment distribution with respect to $W$, i.e.,
		\begin{equation}\label{eqn:thm_chap1_1}
			V[\rho](x) \le  \ess \inf V[\rho],\,x\in \supp\rho\backslash \{\pm M\}.
		\end{equation}
	\item[\textnormal{(iv)}] Further assume that $W$ satisfies {\bf (H5)}. Let $\rho$ be a minimizer. Then $\rho$ is the unique probability measure satisfying (\ref{eqn:thm_chap1_1}) in the class of probability measures with the same Dirac mass configuration.
	\end{enumerate}
\end{theorem}
Theorem \ref{thm:main-characterization} shows the existence of the minimizer(s) for $\cG_{\alpha}$ within $\cM_{\cD\ge d}$ and describes the shape of minimizer(s) to a great extent. Moreover, it shows that minimizer(s) have a characterizing property (\ref{eqn:thm_chap1_1}) and is the unique measure to satisfy this property with a fixed Dirac mass configuration. This translates the problem of constructing minimizers of $\cG_{\alpha}$ to a problem of constructing sediment distributions, which we will do in Section \ref{sec:construction-stationary}. 

We organize this section as the following. We first prove Theorem \ref{thm:main-characterization}(i)in Section \ref{ssec:D-H-property}, see Corollary \ref{coro:existence-even}. In Section \ref{ssec:micro-diff}, we prove (ii) and (iii) in Theorem \ref{thm:minimizer-shape} and Corollary \ref{coro:sediment-minimizer}. Finally we prove (iv) in Section \ref{ssec:energy-min} in Corollary \ref{cor_rhoform}. 

\subsection{Properties of $\cD$ and $\cH$}\label{ssec:D-H-property}
In this subsection, we will give some basic results about the functionals $\cD$ and $\cH$ in Lemma \ref{lem:D-H-property}. Then we will show that it implies Theorem \ref{thm:main-characterization}(i) as a corollary. 
\begin{lemma}\label{lem:D-H-property} Assume $W$ satisfies 
	{\bf (H1)}-{\bf (H4)} and $\alpha>0$. Let $\rho \in \cM$. 
	\begin{enumerate}
	\item[\textnormal{(i)}]
		There exists a closed interval $I$ such that $\cD[\rho] = \int_I (\rho-1) \rd {x}$.
		\item[\textnormal{(ii)}]
		If $\rho_n \wc \rho$ in $\cM$, then $\cD[\rho] = \lim_{n\to \infty} \cD[\rho_n] $.
		\item[\textnormal{(iii)}]
		If $\rho_n \wc \rho$ in $\cM$, then $\cH[\rho] \le \liminf_{n\to \infty} \cH[\rho_n]$.
		\item[\textnormal{(iv)}]
		If $\rho_n \wc \rho$ in $\cM$, then $\cG_{\alpha}[\rho] \le \liminf_{n\to \infty} \cG_{\alpha}[\rho_n]$.
	\end{enumerate}
\end{lemma}
Before we prove this lemma, we first apply it to show Theorem \ref{thm:main-characterization} (i).

\begin{corollary}\label{coro:existence-even}
	Assume $W$ satisfies 
	{\bf (H1)}-{\bf (H4)}, $\alpha\ge 0$ and $0<d\le 1$. There exists a minimizer $\rho$ for $\cG_{\alpha}$ in $\cM_{\cD\ge d}$ that is even and $\cD[\rho] = \int_I (\rho-1)$ with $I = [-M,M]$ for some $0\le M <1/2$.
\end{corollary}
\begin{proof}
	Lemma \ref{lem:D-H-property}(ii) shows that if $\rho_n \wc \rho$ and $\rho_n\in \cM_{\cD\ge d}$, then $\rho\in \cM_{\cD\ge d}$. If $\{\rho_n\}$ is a minimizing sequence of $\cG = \cH/\cD^{\alpha}$ in the class $ \cM_{\cD\ge d}$, then we can take a weakly convergent subsequence, still denoted as $\{\rho_n\}$, and $\rho_n \wc \rho \in \cM_{\cD\ge d}$. Combined with Lemma \ref{lem:D-H-property}(iv), this shows the existence of a minimizer $\rho \in \cM_{\cD\ge d}$. 
%
	
	 We further show that $\rho$ can be made even. Without loss of generality, we assume that $I=[-M,M]$ is an interval witnessing $\cD[\rho]$. We define $\bar{\rho}(x) = (\rho(x)+\rho(-x))/2$. Then $\cD[\bar{\rho}] \ge \cD[\rho]$ since 
     \begin{equation}
     	\int_I (\bar{\rho}-1)\rd{x} = \int_I (\rho-1)\rd{x},
     \end{equation}
	and $\cH[\bar{\rho}]\le \cH[\rho]$ since
	\begin{equation}
		\ess \inf W*\bar{\rho} \ge \frac{1}{2} \Big( \ess \inf(W*\rho)+\ess \inf(W*\rho(-\cdot)) \Big)= \ess\inf(W*\rho).
	\end{equation}
Therefore $\cD[\bar{\rho}] = \cD[\rho]$ and $\cH[\bar{\rho}] = \cH[\rho]$ and $\bar{\rho}$ is the desired minimizer.
\end{proof}
Now we focus on proving Lemma \ref{lem:D-H-property}.
\begin{proof}[Proof of Lemma \ref{lem:D-H-property}]
	\textbf{Proof of (i)}:\\
	If $\cD[\rho]=0$, then $\rho=1$ which implies that $\cD[\rho]$ is achieved at any closed interval $I$. So we assume $D[\rho]>0$. Let $\{I_n\}$ be a maximizing sequence of $\cD[\rho]$ where each $I_n$ is a closed interval on $\mathbb{T}$. We may represent $I_n$ by $[a_n,b_n]$ for some $a_n\in [0,1),b_n\in [a_n,a_n+1)$. By compactness of $\bT$, one can form a subsequence, still denoted as $\{I_n=[a_n,b_n]\}$, such that $\lim_{n\rightarrow\infty}a_n=a$ and $\lim_{n\rightarrow\infty}b_n=b$. By dominated convergence theorem
	\begin{equation}
		\cD[\rho]=\lim_{n\to \infty}	\int_{I_{n}}(\rho-1)\rd{x}\le \lim_{n\to \infty}	\int_{I_{n}\cup \{a, b\}}(\rho-1)\rd{x} = \int_{I} (\rho-1) \rd{x},
	\end{equation}
	where $I = [a,b]$ is a closed interval witnessing $\cD[\rho]$. 

\noindent\textbf{Proof of (ii)}:\\
	Let $I=[a,b]$ be a closed interval witnessing $\cD[\rho]$ and $\rho_n \wc \rho \in \cM$. We first prove $\liminf_{n\rightarrow\infty} \cD[\rho_n] \ge \cD[\rho]$. For $\epsilon>0$ with $b-a<1-2\epsilon$ we define the continuous function
	\begin{equation}\label{psiab}
		\phi_{\epsilon}(x) = \left\{\begin{split}
			& 1,\quad a\le x \le b \\
			& 1-\frac{a-x}{\epsilon},\quad a-\epsilon\le x \le a\\
			& 1-\frac{x-b}{\epsilon},\quad b\le x \le b+\epsilon\\
			& 0\quad \text{otherwise.}
		\end{split}\right.
	\end{equation}
Then for any $n$ and $\epsilon>0$, by definition of $\cD$ and $\phi_{\epsilon}$ we have
\begin{equation}
\cD[\rho_n]\ge \int_{[a-\epsilon,b+\epsilon]} (\rho_n-1)\rd{x} \ge \int_{\bT} (\rho_n-1)\phi_{\epsilon} \rd{x} -2\epsilon.
\end{equation}
Therefore for any $\epsilon>0$, by weak convergence of the sequence
\begin{equation}
\lim \inf_n	\cD[\rho_n]\ge \int_{\bT} (\rho-1)\phi_{\epsilon} \rd{x} -2\epsilon,
\end{equation}
which implies $\lim\inf_n \cD[\rho_n] \ge \cD[\rho]$ as $\epsilon \to 0$. 
	
	Next we prove $\limsup_{n\rightarrow\infty} \cD[\rho_n] \le \cD[\rho]$. Let $I_n=[a_n,b_n]$ be the intervals witnessing $\cD[\rho_n]$. We may take a subsequence of $\{\rho_n\}$, still denoted as $\{\rho_n\}$, such that $\lim_{n\rightarrow\infty} \cD[\rho_n]=\limsup_{n\rightarrow\infty} \cD[\rho_n]$, and $\lim_{n\rightarrow\infty}a_n=a_0$, $\lim_{n\rightarrow\infty}b_n=b_0$ exist. For any $\epsilon>0$, let $\tilde{\phi}_{\epsilon}(x)$ be given as in \eqref{psiab} with $a,b$ replaced by $a_0-\epsilon,b_0+\epsilon$ respectively, then for sufficiently large $n$, we have
	\begin{equation}\label{abD}
		|a_n-a_0|<\epsilon,\quad |b_n-b_0|<\epsilon,\quad \cD[\rho_n] =\int_{[a_n,b_n]}(\rho_n-1)\rd{x} > \lim_{n\rightarrow\infty} \cD[\rho_n]-\epsilon.
	\end{equation}
and 
\begin{equation}
	\int_{\mathbb{T}}\tilde{\phi}_{\epsilon}(\rho-1)\rd{x} \ge \int_{\mathbb{T}}\tilde{\phi}_{\epsilon}(\rho_n-1)\rd{x}-\epsilon.
\end{equation}
since $\rho_n\wc \rho$. Therefore for any $\epsilon>0$ and $I_{\epsilon} = [a_0-\epsilon, b_0+\epsilon]$, we have
	\begin{equation}
		\cD[\rho] \ge \int_{I_{\epsilon}} (\rho-1)\rd{x} \ge \int_{\bT} (\rho-1)\tilde{\phi}_{\epsilon}\rd{x} -2\epsilon  \ge \int_{\mathbb{T}}\tilde{\phi}_{\epsilon}(\rho_n-1)\rd{x}-3\epsilon\ge \int_{[a_n,b_n]}(\rho_n(x)-1)\rd{x}-7\epsilon 
	\end{equation}
	when $n$ is sufficiently large. Since $\epsilon$ is arbitrary, we obtain $\limsup_{n\rightarrow\infty} \cD[\rho_n] \le \cD[\rho]$.
	
\noindent\textbf{Proof of (iii)}:\\
	Let $\rho_n \wc \rho \in \cM$. It suffices to prove
	\begin{equation}\label{claim1}
		(W*\rho)(x) \ge \limsup_{n\rightarrow\infty} \Big( \ess\inf (W*\rho_n) \Big),\quad a.e.\, x
	\end{equation}
    Assume on the contrary that \eqref{claim1} is false. Since $W*\rho \in L^1$, almost every $x\in \bT$ is a Lebesgue point for $W*\rho$. Without loss of generality, assume $x=0$ is a Lebesgue point and there exists $\epsilon>0$ and a subsequence $\{\rho_{n_k}\}$ such that
	\begin{equation}\label{alphag}
		(W*\rho)(0) <  \ess\inf (W*\rho_{n_k}) - 2\epsilon,\quad \forall k.
	\end{equation}
	By definition of Lebesgue points,
	\begin{equation}
		(W*\rho)(0) = \lim_{a\to 0} \int (W*\rho)(x) \psi_a(x) \rd{x},
	\end{equation}
	where $\psi_a$ is as in \eqref{psialpha}. Therefore, one can choose $a$ such that
	\begin{equation}\label{alphag1}
		\int (W*\rho)(x) \psi_a(x) \rd{x} <  \ess\inf (W*\rho_{n_k}) - \epsilon,\quad \forall k.
	\end{equation}
	Notice that
	\begin{equation}
		\int (W*\rho)(x) \psi_a(x) \rd{x} = (W*\rho*\psi_a)(0)  = \int (W*\psi_a)(x)\rho(x)\rd{x}
	\end{equation}
	and $W*\psi_a$ is a continuous function. Therefore, by the weak convergence of $\{\rho_{n_k}\}$, we have
	\begin{equation}
		\int (W*\rho)(x) \psi_a(x) \rd{x} = \lim_{k\rightarrow\infty}\int (W*\rho_{n_k})(x) \psi_a(x) \rd{x}
	\end{equation}
	Notice that $\int (W*\rho_{n_k})(x) \psi_a(x) \rd{x} \ge \ess\inf (W*\rho_{n_k})$. Therefore we get a contradiction with \eqref{alphag1}.

    Finally, (iv) on lower semi-continuity of $\cG_{\alpha}$ follows directly from (ii)  and (iii).
\end{proof}

\subsection{Microscopic Diffusion}\label{ssec:micro-diff}
In this section, we prove Theorem \ref{thm:main-characterization}(ii) and (iii) in Theorem \ref{thm:minimizer-shape} and Corollary \ref{coro:sediment-minimizer}. Let $\rho$ be an even minimizer of $\cG_{\alpha}$ in $\cM_{\cD\ge d}$. The core tool is what we call \emph{microscopic diffusion}. It is a local transport plan for probability measures that can decrease $\cH$ while maintaining or increasing $\cD$, thus decreasing $\cG_{\alpha}$ overall. 

Recall the setup and notations $U$, $V[\rho]$ and $V_U[\rho]$ in Section \ref{sec:notation}.
\begin{lemma}[Microscopic Diffusion]\label{lem:micro-diffusion}
	Assume $W$ satisfies {\bf (H1)}-{\bf (H4)} and $U$ is a function on $\mathbb{T}$ bounded from below. Given $\rho \in \cM$. For a given $x_0\in \mathbb{T}$ and $0<\epsilon<\frac{1}{2}$, if $V_U[\rho](x) \ge \ess\inf V_U[\rho]+c$ (possibly $\infty$) for $x\in [x_0-\epsilon,x_0+\epsilon]$ with some $c>0$ and $\supp\rho\cap (x_0-\epsilon/2,x_0+\epsilon/2)\ne \emptyset$, then we define 
	\begin{equation}
		\rho_\lambda(x) = \rho(x) + \lambda \left(m_1\delta_{x_0-\epsilon}+m_2\delta_{x_0+\epsilon} -\rho(x) \chi_{(x_0-\epsilon,x_0+\epsilon)}\right)
	\end{equation}
where $m_1,m_2> 0$ are determined by the moment conditions 
	\begin{equation}\label{lem_micro_1}
		\int_{(x_0-\epsilon,x_0+\epsilon)}\Big(m_1\delta_{x_0-\epsilon}(y)+m_2\delta_{x_0+\epsilon}(y)-\rho(y)\Big) \cdot y^k \rd{y} = 0, \quad k =0,1.
	\end{equation}
Then for $\lambda>0$ sufficiently small, $\rho_{\lambda}$ is a probability measure with
	\begin{equation}
		-\ess\inf V_U[\rho_\lambda] <- \ess\inf V_U[\rho].
	\end{equation}
\end{lemma}
See Figure \ref{fig:Micro-Diff} for illustration. 
\begin{proof}
	We start by showing the sign of a finite difference using the convexity of $W$. For any $x\notin [-\epsilon,\epsilon]$ and $u\in (-\epsilon,\epsilon)$, if $a_1(u),a_2(u)$ are determined by the moment conditions
\begin{equation}\label{eqa1a2}
	\int_{(-\epsilon,\epsilon)}\Big((a_1\delta_{-\epsilon}(y)+a_2\delta_{\epsilon}(y))-\delta_u(y)\Big)\cdot  y^k \rd{y} = 0, \quad k = 0,1
\end{equation}
that is,
\begin{equation}
	a_1(u)=\frac{1}{2}(1-\frac{u}{\epsilon}),\quad a_2(u)=\frac{1}{2}(1+\frac{u}{\epsilon}),
\end{equation}
then the difference of potential from splitting the Dirac mass at $u$ is
\begin{equation}\label{eqn:a1a2}\begin{split}
		& \big(W*(a_1\delta_{-\epsilon}+a_2\delta_{\epsilon} - \delta_u)\big)(x) \\=&  a_1W(x+\epsilon)+a_2W(x-\epsilon)-W(x-u) \\
		= & a_1\int_{x-u}^{x+\epsilon}\int_{x-u}^y W''(z)\rd{z}\rd{y} + a_2\int_{x-\epsilon}^{x-u} \int_y^{x-u} W''(z)\rd{z}\rd{y} \\
		\ge &C_1 \left(a_1\int_{x-u}^{x+\epsilon}\int_{x-u}^y \rd{z}\rd{y} + a_2\int_{x-\epsilon}^{x-u} \int_y^{x-u} \rd{z}\rd{y}\right)>0,\\
\end{split}\end{equation}
where $c(\epsilon,u):=a_1\int_{x-u}^{x+\epsilon}\int_{x-u}^y \rd{z}\rd{y} + a_2\int_{x-\epsilon}^{x-u} \int_y^{x-u} \rd{z}\rd{y}>0$ and $C_1>0$ is as given in {\bf (H4)}. 

Integrating \eqref{eqn:a1a2} in $u\in (-\epsilon,\epsilon)$ with weight $\rho(x_0+u)$, we obtain \eqref{lem_micro_1} with 
\begin{equation}\label{eqn:m1m2}
	m_1 = \int_{(-\epsilon,\epsilon)}\frac{1}{2}(1-\frac{u}{\epsilon})\rho(x_0+u)\rd{u},\quad m_2 = \int_{(-\epsilon,\epsilon)}\frac{1}{2}(1+\frac{u}{\epsilon})\rho(x_0+u)\rd{u},
\end{equation}
which are both positive since $\supp\rho\cap (x_0-\epsilon/2,x_0+\epsilon/2)\ne \emptyset$. Here $m_1$ and $m_2$ are uniquely determined because the coefficient matrix in \eqref{lem_micro_1} is invertible. Now for any $x\notin [x_0-\epsilon,x_0+\epsilon]$
\begin{equation}\label{eqn:micro-diffusion-away}\begin{split}
\left(W*\left(m_1\delta_{x_0-\epsilon}+m_2\delta_{x_0+\epsilon}-\rho(x) \chi_{(x_0-\epsilon,x_0+\epsilon)}\right)\right)(x) 
		\ge  C_1\int_{(-\epsilon,\epsilon)} c(\epsilon,u)\rho(x_0+u)\rd{u} > 0.
\end{split}\end{equation}

For $x\in [x_0-\epsilon,x_0+\epsilon]$, notice that $\rho_\lambda = (1-\lambda)\rho + \lambda \rho_1$ we have
\begin{equation}
	V_U[\rho_\lambda](x) = (1-\lambda)V_U[\rho](x) + \lambda V_U[\rho_1](x) \ge (1-\lambda)(\ess\inf V_U[\rho] + c) + \lambda (\inf W + \inf U),
\end{equation}
which is strictly larger than $\ess\inf V_U[\rho]$ by taking $\lambda>0$ small enough. Therefore we obtain the conclusion.
\end{proof}

\begin{figure}
	\includegraphics[width=7cm, height = 5cm]{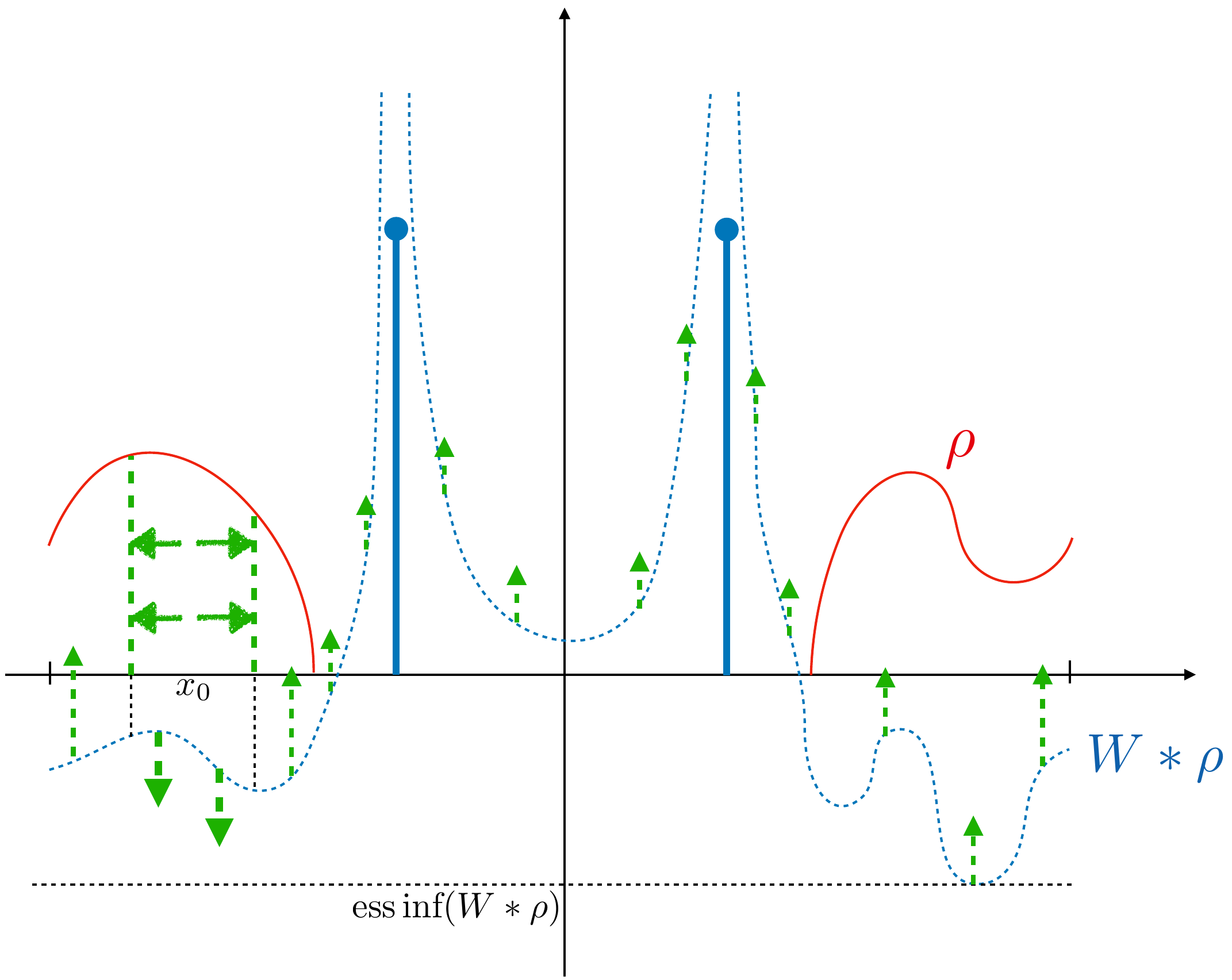}
	\caption{ Microscopic Diffusion: horizontal arrows indicate the diffusion of mass near $x_0$, vertical arrows indicate the change of the generated potential $W*\rho$, which is increasing away from $x_0$ and decreasing near $x_0$. }
	\label{fig:Micro-Diff}
\end{figure}
Microscopic diffusion operation basically implies that if $\rho$ is a minimizer of $\cG_{\alpha}$ with $I = [-M, M]$ witnessing $\cD[\rho]$, then $V[\rho](x)\le \ess\inf V[\rho]$ as long as $x\in \supp \rho \backslash \{ \pm M\}$. Suppose not, by Proposition \ref{prop:semi}, we can always find a small neighborhood $(x_0-\epsilon, x_0+\epsilon)$ for some $x_0 \in \supp \rho \backslash  \{\pm M \}$ to apply microscopic diffusion and strictly decrease $\cH$. Since the diffusion operation is local and away from $\pm M$, it won't decrease $\cD$, we then get a contradiction with $\rho$ being a minimizer. This proves Theorem \ref{thm:main-characterization}(iii).

\begin{corollary}\label{coro:sediment-minimizer}
	Assume $W$ satisfies {\bf (H1)}-{\bf (H4)}. Let $\alpha\ge 0$, $0<d\le 1$, $0\le M <1/2$. Let $\rho$ be an even minimizer of $\cG_{\alpha}$ in $\cM_{\cD\ge d}$ with $I = [-M, M]$ witnessing $\cD[\rho]$. Then $V[\rho](x) \le \ess\inf V[\rho]$ for $x \in \supp \rho \backslash\{ \pm M \}$.
\end{corollary}

This in turn imposes a really strong restriction on the shape of $\supp \rho$ for a minimizer $\rho$. Suppose $(x_1, x_2)\subset (\supp \rho)^c$ and $x_1, x_2\in \supp \rho$ and $x_i$ are not $\pm M$. Then by Corollary \ref{coro:sediment-minimizer} $V[\rho](x_1) = V[\rho](x_2) \le \ess \inf V[\rho]$. The generated potential $V[\rho]$ is continuous and convex at $x \notin \supp \rho$, $V[\rho]$ is right continuous at $x_1$ and left continous at $x_2$ by Proposition \ref{prop:semi}. Therefore by convexity $V[\rho](x)< \ess \inf V[\rho]$ for $x \in (x_1, x_2)$, which is a contradiction. Therefore there exists no $(x_1, x_2) \subset (\supp \rho)^c$ with $x_i \in \supp(\rho) \backslash \{ \pm M \}$. This helps us to further characterize the shape of the minimizer in the following theorem.  
\begin{theorem}\label{thm:minimizer-shape}
	Assume $W$ satisfies {\bf (H1)}-{\bf (H4)}. Let $\alpha\ge 0$, $0<d\le 1$, $0\le M <1/2$.  Let $\rho$ be an even minimizer of $\cG_{\alpha}$ in $\cM_{\cD\ge d}$ with $I = [-M, M]$ witnessing $\cD[\rho]$. Then 
	\begin{equation}
		\rho=m (\delta_M+\delta_{-M}) + \rho_1 
	\end{equation}
	for some $0<m\le 1/2$, and $\rho_1 \in \cM_{1-2m}$ is supported and non-zero on two (possibly empty) closed intervals $J\subsetneq I^c$ and $K\subsetneq I$.
\end{theorem}
Before we give the proof of Theorem \ref{thm:minimizer-shape}, we first give a lemma on $W$.
\begin{lemma}\label{lem:W-hat-positive}
	Assume $W$ satisfies {\bf (H1)}, {\bf (H2)}, {\bf (H4)}. Then $\hat{W}(k)>0$ for any $k\in\mathbb{Z},\,k\ne 0$.
\end{lemma}
\begin{proof}
By {\bf (H2)} $W$ is even, therefore $W'(x)$ is odd for $x \neq 0$ with $W'(1/2)=W'(-1/2)=0$. By {\bf (H4)} $W''(x)>0$ for any $x\ne 0$, therefore $W'$ is a strictly increasing function which maps $[-1/2,0)$ onto $[0,A)$ with $A=\lim_{x\rightarrow0-}W'(x) \in (0,\infty]$, and $\inf_{\bT} W = W(-1/2)$. Let $u$ be the inverse function of $W'|_{[-1/2,0)}$, which is strictly increasing on $(0,A)$.
	
For $x\in [-1/2,0)\cup (0,1/2]$, we have $W(x) = W(-|x|)$ and
	\begin{equation}
			W(x) - W(-1/2) = \int_{-1/2}^{-|x|} W'(y)\rd{y}  = \int_0^A \max\{-|x|-u(h),0\}\rd{h}.
   \end{equation}
Therefore for $k\neq 0$
\begin{equation}
	\hat{W}(k) = \int_{\bT} (W(x) - W(-1/2))e^{-2\pi i kx} \rd{x} =\int_{0}^{A} \rd{h} \int_{\bT} \max\{-|x|-u(h),0\} e^{-2\pi ikx} \rd {x}. 
\end{equation}
For each fixed $h$, we get $\beta = -u(h)\in (0,1/2)$. It suffices to notice that $\max\{\beta-|x|, 0\}$ has positive Fourier coefficients at $k\neq 0$ for almost every $\beta$.
\end{proof}
Now we are ready to prove Theorem \ref{thm:minimizer-shape}.
\begin{proof}[Proof of Theorem \ref{thm:minimizer-shape}]
	Firstly, since $I$ is maximizing $\int_I (\rho-1)$, its endpoints $\pm M \in \supp \rho$. By the reasoning above Theorem \ref{thm:minimizer-shape}, there is no open interval $(x_1, x_2)$ in $(\supp \rho)^c$ with $x_i \in \supp \rho \in \backslash \{ \pm M\}$, there is at most one closed interval $J \subset \bar{I^c}$ (closure of the open set $I^c$) and at most one closed interval $K\subset I$ (empty when $M=0$), i.e., $\supp \rho = \{ \pm M \} \cup J \cup K$. 
	
	We will show that $J\neq \bar{I^c}$ and $K\neq I$. If $\{  \pm M \} \subset J\cup K$ then $V[\rho](\pm M) = \ess \inf V[\rho]$, by $1$ in Proposition \ref{prop:semi} and by Corollary \ref{coro:sediment-minimizer}. Moreover, combining $3$ in Proposition \ref{prop:semi}, we see that $V[\rho]$ is now continuous with $V[\rho](x) \le \ess \inf V[\rho]$ for $x\in \supp \rho = J\cup K$. Since $\bT\backslash (J\cup K)$ is a disjoint union of two open intervals or a single open interval, by convexity of $V[\rho]$ for all $x\in \bT$ we have $V[\rho](x)\le \ess \inf V[\rho]$. This implies that $V[\rho]=W*\rho= 0$. Then $\widehat{W*\rho} = \hat{W}\cdot  \hat{\rho} = 0$, therefore $\hat{\rho}(k)=0$ for every $k\neq 0$ by Lemma \ref{lem:W-hat-positive} and $\rho$ must be the uniform distribution. This contradicts with $d>0$. 
	
	Finally, since $J\subsetneq \bar{I^c}$ and $K\subsetneq I$ and $I$ witnesses $\cD[\rho]$, it follows that $\rho$ must have Dirac masses at $\pm M$. 
\end{proof}

\subsection{Energy Minimization}\label{ssec:energy-min}
Our main goal in this section is to prove the uniqueness in Theorem \ref{thm:main-characterization}(iv). We will do so by using the idea of energy minimization in potential theory. 

To set up the question, we will consider an external potential $U\in L^1(\mathbb{T})$ of the form
\begin{equation}\label{eqn:U}
	U =  W*(\rho_{+} - \rho_{-} + \rho_{d})
\end{equation}
where $\rho_{+}$  and $\rho_{-}$ are nonnegative continuous functions, $\rho_{d}$ is the sum of finitely many positive Dirac masses and $W$ satisfies {\bf (H1)}-{\bf (H4)}. As a consequence, $U$ is bounded from below and at every point $x \in \bT \backslash \supp \rho_d$ we have $U(x)<\infty$ and continuous. Define the total energy functional
\begin{equation}\label{eqn:E}
	\cE_U[\rho] = \frac{1}{2}\int ( W*\rho)(x)\rho(x) \rd{x} + \int U(x)\rho(x)\rd{x}
\end{equation}
for $\rho \in \cM_{m}$ with $m\ge 0$. The dependence of $\cE_U$ on $U$ will be omitted when it is clear from the context. Here the two terms in \eqref{eqn:E} physically represent the pairwise interaction energy and potential energy respectively. For any $\rho\in\cM_{m}$, both terms in $\cE_U[\rho]$ take value in $(-\infty,\infty]$ since $U$ and $W$ (implied by {\bf (H1)} and {\bf (H3)}) are bounded from below, therefore finiteness of $\cE_U[\rho]$ implies finiteness of both terms. 

\begin{proposition}\label{prop:energy-min}
	Assume $W$ satisfies {\bf (H1)}-{\bf (H5)}, $U$ has the form \eqref{eqn:U} and $m\ge 0$. Then there exists a unique minimizer $\rho$ of $\cE$ in $\cM_{m}$. It is the only element in $\cM_{m}$ satisfying
	\begin{equation}\label{eqn:Emin-sediment}
		V_U[\rho](x) \le  \ess\inf V_U[\rho ],\,x\in \supp\rho.
	\end{equation}
	Furthermore, $\rho$ is also the unique maximizer of $\ess\inf V_U[\rho]$ for $\rho\in \cM_{m}$.
\end{proposition}
\begin{proof}
	{\bf Existence of energy minimizer:}\\
	Let $\{\rho_n\}$ be a minimizing sequence of $\cE$, and take a weakly convergent subsequence, still denoted as $\{\rho_n\}$, which converges weakly to some $\rho_\infty\in \cM_{m_0}$. Since both $W$ and $U$ are lower semicontinuous, we apply \cite[Theorem 1.3.4 (i)(iv)]{WeakConv} and get
		\begin{equation}
			\iint_{\bT^2}  W(x-y)\rho_\infty(y)\rho_\infty(x)\rd{x}\rd{y} \le \liminf_{n\rightarrow\infty}\iint_{\bT^2}  W(x-y)\rho_n(y)\rho_n(x)\rd{x}\rd{y}, 
		\end{equation}
	and
		\begin{equation}
			\int U(x)\rho_\infty(x)\rd{x} \le \liminf_{n\rightarrow\infty}\int U(x)\rho_n(x)\rd{x}.
		\end{equation}
		Therefore
		\begin{equation}
			\cE[\rho_\infty] \le \liminf_{n\rightarrow\infty}\cE[\rho_n] = \inf_{\rho\in \cM_{m_0}}\cE[\rho].
		\end{equation}	
	
\noindent	{\bf Characterizing property}: \\
	If $\rho_{\infty}$ does not satisfy \eqref{eqn:Emin-sediment}, we will give a transport plan to construct $\rho$ with smaller $\cE$. Under this assumption, by lower continuity of $V_{\infty}:=V_U[\rho_{\infty}]$ in Proposition \ref{prop:semi}, there exists $x_0\in \supp \rho$, $a>0$ and $\epsilon>0$ such that $V_\infty(x)\ge \ess\inf V_\infty+2\epsilon$ for $x\in [x_0-a,x_0+a]$. Since $x_0\in \supp\rho_\infty$, we have $\int_{[x_0-a,x_0+a]}\rho_\infty\rd{x}>0$. By the definition of $\ess\inf V_\infty$, there exists a set $S$ with positive measure such that $V_\infty(x) \le \ess\inf V_\infty+\epsilon$ for $x\in S$. Now consider
	\begin{equation}\label{eqn:energy-min-transport-plan}
\varphi_1=\rho_\infty\cdot \chi_{[x_0-a,x_0+a]} ,\quad \varphi_2=  \chi_S\cdot \frac{1}{|S|} \int_{\bT}\varphi_1\rd{x}, \quad\rho_\beta = \rho_\infty + \beta (-\varphi_1+\varphi_2)
	\end{equation}
	for $0<\beta<1$. Notice that $\varphi_1,\varphi_2$ are nonnegative and $\gamma:=\int_{\mathbb{T}}\varphi_1\rd{x}=\int_{\mathbb{T}}\varphi_2\rd{x}>0$, and $\rho_\infty-\beta\varphi_1\ge 0$ by the construction of $\varphi_1$. Therefore $\rho_\beta$ is in $\cM_{m}$. Then
	\begin{equation}\begin{split}
			\cE[\rho_\beta] = &\cE[\rho_\infty] + \beta\int V_\infty\cdot  (-\varphi_1+\varphi_2)\rd{x} + \frac{\beta^2}{2}\int \Big(W*(-\varphi_1+\varphi_2)\Big)\cdot (-\varphi_1+\varphi_2)\rd{x} \\
			\le & \cE[\rho_\infty] - \beta\cdot (\ess\inf V_\infty+2\epsilon)\int \varphi_1\rd{x} + \beta\cdot (\ess\inf V_\infty+\epsilon)\int \varphi_2\rd{x} + C\beta^2 \\
			= & \cE[\rho_\infty] - \gamma\epsilon\beta + C\beta^2 \\
	\end{split}\end{equation}
	where $C=\int \big( W*(-\varphi_1+\varphi_2))\cdot (-\varphi_1+\varphi_2)\rd{x} <\infty$ from the construction of $\varphi_i$ and {\bf (H1)}. Therefore we obtain a contradiction with the minimality of $\rho_\infty$ by taking $\beta$ small enough.
	
\noindent	{\bf Uniqueness of minimizer}:\\
Assume $\rho_0$ and $\rho_1$ are two distinct minimizers of $\cE$ in $\cM_{m}$. Define $\rho_t = (1-t)\rho_0+t\rho_1\in \cM_{m}$ for $0\le t \le 1$. Then
	\begin{equation}\label{eqn:Erhot}\begin{split}
			\cE[\rho_t] = & \frac{(1-t)^2}{2}\int ( W*\rho_0)\cdot \rho_0 \rd{x} + \frac{t^2}{2}\int ( W*\rho_1)\cdot \rho_1 \rd{x} + t(1-t)\int ( W*\rho_0)\cdot \rho_1 \rd{x} \\
			& + (1-t)\int U\cdot \rho_0\rd{x} + t\int U\cdot \rho_1\rd{x},
	\end{split}\end{equation}
	is a quadratic function in $t$, with
	\begin{equation}\label{eqn:d2Erhot}
		\frac{\rd^2}{\rd{t}^2}\cE[\rho_t]  = \int ( W*(\rho_1-\rho_0))(x)(\rho_1-\rho_0)(x) \rd{x} = \sum_{k\in \mathbb{Z}}\hat{ W}(k)|\hat{\rho}_1(k)-\hat{\rho}_0(k)|^2
	\end{equation}
	by Proposition \ref{prop:energy-Fourier}. Since $\rho_0\ne \rho_1$ and $\hat{\rho}_0(0)=\hat{\rho}_1(0)=m$, there exists some $k\neq $ such that $\hat{\rho}_1(k)\ne\hat{\rho}_0(k)$. Notice $\hat{W}(k)>0$ for $k\ne 0$ by Lemma \ref{lem:W-hat-positive}. Therefore $\frac{\rd^2}{\rd{t}^2}\cE[\rho_t]>0$ and $\cE[\rho_{1/2}]<\cE[\rho_0] = \cE[\rho_1] $. Contradiction.
	
\noindent	{\bf Uniqueness of $\rho$ satisfying \eqref{eqn:Emin-sediment}}:\\
Assume $\rho_1\ne \rho_\infty$ satisfies  \eqref{eqn:Emin-sediment} with $\rho_\infty$ being the unique minimizer. Let $a>0$ be small, and $\rho_0=\rho_\infty*\psi_a$, where $\psi_a$ is as defined in \eqref{psialpha}. Then we define $\rho_t = (1-t)\rho_0+t\rho_1\in \cM_{m}$ for $0\le t \le 1$, and $\cE[\rho_t]$ is given by \eqref{eqn:Erhot} and satisfies \eqref{eqn:d2Erhot}. We also have
	\begin{equation}
		\begin{aligned}
			\frac{\rd}{\rd{t}}\cE[\rho_t] = & -(1-t)\int ( W*\rho_0)\cdot \rho_0 \rd{x} + t\int ( W*\rho_1)\cdot \rho_1 \rd{x} + (1-2t)\int ( W*\rho_1)\cdot \rho_0 \rd{x} \\
			& -\int U\cdot \rho_0\rd{x}+\int U\cdot \rho_1\rd{x}
			\end{aligned}
\end{equation}
	Evaluating at $t=1$, we obtain
	\begin{equation}\label{dEt1}
			\frac{\rd}{\rd{t}}\Big|_{t=1} \cE[\rho_t] = \int V_U[\rho_1]\cdot (\rho_1-\rho_0)\rd{x} \le  \ess\inf V_U[\rho_1]m - \ess\inf V_U[\rho_1]m = 0,
\end{equation}
	using \eqref{eqn:Emin-sediment} for $\rho_1$ and the fact that $\rho_0$ is a continuous function. 
	
	Then we integrate the inequality $\frac{\rd^2}{\rd{t}^2}\cE[\rho_t]\ge 0$ from $t$ to $1$ and get
	\begin{equation}
		\frac{\rd}{\rd{t}}\cE[\rho_t] = \frac{\rd}{\rd{t}}\Big|_{t=1}\cE[\rho_t] - \int_t^1 \frac{\rd^2}{\rd{s}^2}\cE[\rho_s]\rd{s} \le 0,\quad \forall 0\le t \le 1,
	\end{equation}
therefore $\cE[\rho_0]\ge \cE[\rho_1] $. Since $\cE[\rho_\infty]-\cE[\rho_1]< 0$, by Proposition \ref{prop:energy-Fourier} we can take $a>0$ small enough such that $\cE[\rho_0]< \cE[\rho_1]$. Contradiction.	

\noindent{\bf Maximizers of $\ess\inf V_U[\rho]$ satisfying \eqref{eqn:Emin-sediment}}: \\
We first note that the maximizer of $\ess\inf V_U[\rho]$ exists. If $U=0$, the existence of maximizer of $\ess \inf V[\rho] = -\cH[\rho]$ is equivalent to the existence of minimizer of $\cH$, which follows from lower semi-continuous of $\cH$ in Lemma \ref{lem:D-H-property} after scaling with $m$. For general $U$ in the form of \eqref{eqn:U}, the same proof in Lemma \ref{lem:D-H-property} applies when $W*\rho$ is replaced with $W*(\rho + \rho_+ - \rho_-+\rho_d)$. 

Let $\rho_\infty$ be a maximizer of $\ess\inf V_U[\rho]$. Suppose the statement is not true, then by lower semicontinuity of $V_U[\rho_{\infty}]$, there exists $x_0\in \supp \rho_{\infty}$, $a>0$ and $\epsilon>0$ such that
\begin{equation}
	V_U[\rho_\infty](x) > \ess\inf V_U[\rho_\infty]+ \epsilon,\quad x\in [x_0-a, x_0+a]
\end{equation}
Then applying Lemma \ref{lem:micro-diffusion} to $[x_0-a,x_0+a]$ with the external potential $U$, we obtain $\rho_\lambda$ with $\ess\inf V_U[\rho_\lambda] > \ess\inf V_U[\rho_\infty]$. Contradiction. It follows that the unique minimizer of $\cE_U[\rho]$ is simultaneously the unique maximizer of $\ess \inf V_U[\rho]$ and the unique element satisfying \eqref{eqn:Emin-sediment}. 
\end{proof}
Now we are ready to prove Theorem \ref{thm:main-characterization}(iv) following Proposition \ref{prop:energy-min}. Recall Theorem \ref{thm:main-characterization}(ii) and Theorem \ref{thm:minimizer-shape}, we can always write $\rho =  m(\delta_M+\delta_{-M})+\rho_1$. By letting $U = W* m(\delta_M+\delta_{-M})$ and $\rho_1$ be the unique minimizer of $\cE_U$, we obtain that $V_U[\rho_1] = V[\rho]$ satisfying the \eqref{eqn:sediment}. Since $\rho_1$ is the energy minimizer of $\cE_U$, it is clear that $\rho_1$ does not contain any Dirac mass. Therefore the Dirac mass configuration of $\rho$ is completely determined by $m$ and $M$ in $U$. 
\begin{corollary}\label{cor_rhoform}
Assume $W$ satisfies {\bf (H1)}-{\bf (H5)}. Let $\alpha\ge 0$, $0<d\le 1$, $0\le M< 1/2$. Following Theorem \ref{thm:minimizer-shape}, let $\rho= m(\delta_M+\delta_{-M})+\rho_1$ be an even minimizer of $\cG_{\alpha}$ in $\cM_{\cD\ge d}$ with $I = [-M, M]$ witnessing $\cD[\rho]$. Then $\rho_1$ is the unique minimizer of $\cE_U$ in $\cM_{1-2m}$ where $U = m(\delta_M+\delta_{-M})$. Furthermore $\rho$ is the unique probability measure satisfying \eqref{eqn:sediment} in the class of probability measures with the same Dirac mass configuration.
\end{corollary}

For the purpose of our discussion in Section \ref{sec:construction-stationary}, we also give a refined version of Proposition \ref{prop:energy-min} in the case that $U$ is the sum of two Dirac masses of the same size at $-M$ and $M$, where we restrict $\rho$ to a smaller class of measures with the prescribed total mass on the two intervals $(-M, M)$ and $(M, 1-M)$.

\begin{corollary}\label{cor_Emin}
	Let $0<m\le 1/2$, $0<M<1/2$, $\rho_d = m(\delta_M +\delta_{-M})$ and $U(x) = W* \rho_d$, and $m_1,m_2\ge 0$ with $m_1+m_2=1-2m$. Then there exists a unique minimizer $\rho_{\infty,m_1,m_2}$ of $\cE_U$ in 
	\begin{equation}\label{Mm1m2}
		\cM_{m_1,m_2}=\Big\{\rho\in\cM_{1-2m}:\int_{(-M,M)}\rho\rd{x} = m_1,\,\int_{(M,1-M)}\rho\rd{x} = m_2\Big\},
	\end{equation}
which is the unique $\rho$ satisfying
\begin{equation}\label{cor_Emin_1}\begin{split}
		& V_U[\rho](x) \le  \underset{y\in (-M,M)}{\ess\inf} V_U[\rho](y),\quad x\in \supp\rho\cap (-M,M) \\
		& V_U[\rho](x) \le  \underset{y\in (M,1-M)}{\ess\inf} V_U[\rho](y),\quad x\in \supp\rho\cap (M,1-M) \\
\end{split}\end{equation}
There exists a unique $(m_1,m_2)$ with either of the following holds:
\begin{itemize}
	\item $m_1=0$, $ \underset{y\in (-M,M)}{\ess\inf}V_U[\rho_{\infty,m_1,m_2}](y)\ge\underset{y\in (M,1-M)}{\ess\inf}V_U[\rho_{\infty,m_1,m_2}](y)$.
	\item $m_2=0$, $ \underset{y\in (-M,M)}{\ess\inf}V_U[\rho_{\infty,m_1,m_2}](y)\le\underset{y\in (M,1-M)}{\ess\inf}V_U[\rho_{\infty,m_1,m_2}](y)$.
	\item $m_1>0,\,m_2>0$, $\underset{y\in (-M,M)}{\ess\inf}V_U[\rho_{\infty,m_1,m_2}](y)=\underset{y\in (M,1-M)}{\ess\inf}V_U[\rho_{\infty,m_1,m_2}](y)$
\end{itemize}
and in this case $\rho_{\infty,m_1,m_2}=\rho_\infty$ is the minimizer of $\cE_U$ in $\cM_{1-2m}$.
\end{corollary}

Since the proof is very similar to Proposition \ref{prop:energy-min}, we will only give a sketch of the proof.
\begin{proof}
	{\bf Existence of energy minimizer}: Let $\rho_n \wc \rho \in \cM_{m_1,m_2}$ minimizing $\cE_U$. Then $\rho$ is a minimizer of $\cE$ by lower semicontinuity of $\cE$ with respect to weak convergence. 
	
	{\bf Characterizing property}: If $\rho_{\infty, m_1, m_2}$ does not satisfy the first line in \eqref{cor_Emin_1}, a similar transport plan with \eqref{eqn:energy-min-transport-plan} where $S$ is now a subset of $(-M, M)$ will construct $\rho_{\beta}$ with smaller $\cE_U$. Similarly for the second line of \eqref{cor_Emin_1}.
	
	{\bf Uniqueness of minimizer in $\cM_{m_1, m_2}$}. One can prove by contradiction using a linear interpolation $\rho_t=(1-t)\rho_0+t\rho_1\in \cM_{m_1,m_2}$ between two minimizers of $\cE$ in $\cM_{m_1,m_2}$.
	
	{\bf Relation between $\rho_{\infty;m_1,m_2}$ and $\rho_\infty$}: Let $1-2m\ge 0$, and $\rho_\infty$ be the minimizer of $\cE$ in $\cM_{1-2m}$, whose existence and uniqueness are guaranteed by Proposition \ref{prop:energy-min}. Since $\cE_U[\rho_{\infty}]<\infty$, $\rho_{\infty}$ does not contain Dirac masses, therefore the uniqueness and existence of the pair $(m_1, m_2)$ follows. The conditions on $\rho_{\infty, m_1, m_2}$ follow from \eqref{eqn:Emin-sediment} directly.	
\end{proof}

\section{Construction of Distributions}\label{sec:construction-stationary}
Our main goal for this section is to construct a class of measures that are candidates for the unique minimizer of $\cG$. In Section \ref{sec:minimizer-characterization}, we prove that the unique minimizer of $\cG_{\alpha}$ in $\cM_{\cD\ge d}$ must be a certain Dirac mass configuration in the format $\rho_d = m (\delta_M+\delta_{-M})$ together with the unique sediment distribution with respect to $U = W*\rho_d$ in $\cM_{1-2m}$ under the potential $W$. In this section, we will first give the constructions of a class of (signed) measures which we loosely refer to as \emph{stationary distribution}. These measures satisfy the condition that $V_U[\rho]'(x)=0$ for $x\in \supp \rho$. It clearly forms a larger class of measures since a differentiable $V_U[\rho]$ for a sediment distribution $\rho$ satisfies $V_U[\rho] = \ess \inf V_U[\rho]$ in $\supp \rho$, which is a constant. \footnote{This concept of the stationary distribution is closely related to the \emph{stationary state} (or \emph{steady state}) and \emph{local energy minimizer} in potential theory, see e.g. \cite{ShuCar}.} We demonstrate the inclusion of all mentioned classes as following:
\begin{center}
$\{ $stationary measures w.r.t $U \} \supset \{ $sediment measures w.r.t $ U  \} =$
$ \{ $minimizers of $\cE_U \} \supset$ \\ $ \{$ minimizer of $\cG_{\alpha}$ in $\cM_{\cD\ge d}$ subtracts its Dirac masses $ \}$.
\end{center}

Although our final goal and the above inclusions are all about measures over $\bT$, we take a detour to discuss signed measures over $\R$ where the minimal value of another functional $\tilde{\cG}$, defined in \eqref{eqn:def-tilde-G} as an analogue of $\cG$, is relatively easier to determine. We will discuss the minimal value of $\cG$ in Section \ref{sec:R-to-T} and its analogue $\tilde{\cG}$ over $\R$ in Section \ref{sec:min-value-R}. For the current section, we will focus on the construction of interesting measures. From now on, we will take $W(x) = -\log|2\sin (\pi x)|$ for $x\in \bT$ and $\tilde{W}(x) = -\log |x|$ for $x\in \R$. The construction heavily deploys the facts that we are using logarithmic potential, which is closely related to the kernel of Hilbert transform, on both $\R$ and $\bT$. We will first give the construction over $\R$ since it is relatively cleaner, and save the construction over $\bT$ to after. 
\subsection{Stationary Distributions over $\R$}\label{ssec:stationary-R}
In this section, we firstly construct a family of signed measures over $\R$. To simplify the notation, given a sequence of real numbers
\begin{equation}
	L_1 < M_1 < R_1 \le L_2 < M_2 < R_2 \le \cdots \le L_n < M_n < R_n,
\end{equation}
we denote $S\subset \R$ to be the union of intervals 
\begin{equation}
	S:=(-\infty,L_1]\cup \bigcup_{j=1}^{n-1}[R_j,L_{j+1}]\cup [R_n,\infty),
\end{equation}
and denote the product 
\begin{equation}
	T_{L}(x):=\prod_{j=1}^n |x-L_j|, \quad T_{L,k}(x):=\prod_{j\neq k} |x-L_j|.
\end{equation}
Similarly for $T_M$ and $T_R$.

\begin{lemma}\label{lem_complex}
Given a sequence of real numbers $L_1 < M_1 < R_1\le \cdots \le L_n < M_n < R_n$. We define $\mu(x) = -1+\mu_c(x) + \mu_d(x)$ where
	\begin{equation}
		\mu_c(x) = \frac{\sqrt{T_L(x)T_R(x)}}{T_M(x)}\chi_S, \quad \mu_d(x) = \sum_{j=1}^n a_j\delta_{M_j}
	\end{equation}
	with
	\begin{equation}\label{ak}
		a_k =\pi\prod_{j=1}^n\frac{\sqrt{T_L(M_k)T_R(M_k)}}{T_{M,k}(M_k)}>0.
	\end{equation}
Then $\mu(x)$ satisfies
	\begin{equation}\label{eqn:diff}
		\pv (\tilde{W}'*\mu)(x) = \left\{\begin{split}
			& -\pi\cdot\sgn(x-M_k) \frac{\sqrt{T_L(x)T_R(x)}}{T_M(x)},\quad x\in [L_k,M_k)\cup(M_k,R_k] \\
			& 0,\quad x\in S
		\end{split}\right.
	\end{equation}
If further more
\begin{equation}\label{lem_complex2_1}
	\sum_{j=1}^n (L_j+R_j-2M_j) = 0,
\end{equation}
then $|\mu(x)| \le \frac{C}{x^2}$ as $|x|\to \infty$ and $\int_{\R} |\mu(x) |\rd{x} < \infty$ and $\int_{\mathbb{R}}\mu(x)\rd{x}=0$.
\end{lemma}
\begin{proof}	
We define a function $g(z) = -1+ g_c(z)+g_d(z)$ where
	\begin{equation}\label{g}
		g_c(z) =\frac{\prod_{j=1}^n\phi(z-L_j)\phi(z-R_j)}{\prod_{j=1}^n (z-M_j)}, \quad  g_d(z) =- \sum_{j=1}^n \frac{a_j}{\pi}\frac{i}{z-M_j}
	\end{equation}
where $a_k$ is given by \eqref{ak} and $\phi(z)= \sqrt{z}$ is the analytic function on $\mathbb{C}\backslash i\mathbb{R}_{\le0}$. It is clear that $g(z)$ is holomorphic on the upper half plane $\C_{y>0}$. Moreover, $g(z)$ is continuous in $\C_{y\ge 0}$ since the pole of $g_c(z)$ at $M_k$ is exactly cancelled by that of $g_d(z)$ due to the value of $a_k$.
	
	Now we analyze $g$ on the real line. If $x\in S$ then $g_c(x) = \mu_c(x)$ is real and positive. If $x\in [L_k,M_k)\cup(M_k,R_k]$, then 
	\begin{equation}
	g_c(x) = i\cdot \sgn(x-M_k)\frac{\prod_{j=1}^n\sqrt{|(x-L_j)(x-R_j)|}}{\prod_{j=1}^n |x-M_j|} = i \cdot \sgn(x-M_k) \frac{\sqrt{T_L(x)T_R(x)}}{T_M(x)}
	\end{equation}
	is purely imaginary. Therefore for $x\in \R \backslash \{ M_k\mid 1\le k\le n \}$ we have
	\begin{equation}
		\begin{aligned}
	\Re(g)(x) &=  -1 + \mu_c(x),\\
	\Im(g)(x) &= \sum_{k=1}^n\chi_{[L_k,R_k]}\sgn(x-M_k)\frac{\sqrt{T_L(x)T_R(x)}}{T_M(x)}-\sum_{j=1}^n \frac{a_j}{\pi}\frac{1}{x-M_j}.
		\end{aligned}
	\end{equation}
The function $\Im(g)(x)$ can be extended to $\R$ continuously since $g$ is continuous on $\C_{y\ge 0}$. The function $\Re(g)(x)$ is continuous at any $x\in\mathbb{R}$ and $\Re(g)\in L^2(\R)$, since as $|x|\rightarrow\infty$
	\begin{equation}\label{Re_g}
		|\Re(g)(x)| = \left|-1 + \frac{1-\frac{1}{2}\sum_{j=1}^n (L_j+R_j)\frac{1}{|x|} + O(\frac{1}{|x|^2})}{1-\sum_{j=1}^n M_j \frac{1}{|x|} + O(\frac{1}{|x|^2})}\right| \le \frac{C}{|x|}.
	\end{equation}
    Therefore Hilbert transform $H[\Re(g)]$ exists and it is a standard property that
	\begin{equation}\label{eqn:Hilbert-transform}
		H[\Re(g)] = \pv \frac{1}{\pi y} *(-1 + \mu_c(y)) =\pv \frac{-\tilde{W}'}{\pi} *(-1 + \mu_c(y)) =  \Im(g),
     \end{equation}
 where the second equality follows since $\tilde{W}'(x) = -\frac{1}{x}$. We finished proving \eqref{eqn:diff} by adding the term $-\tilde{W}'/\pi *\mu_d$ to both sides of \eqref{eqn:Hilbert-transform}. 
 
 Using the assumption \eqref{lem_complex2_1}, we can refine \eqref{Re_g} to $|\Re(g)(x)| \le  \frac{C}{|x|^2}$, and therefore $\mu = \Re(g)+\mu_d$ satisfies $\int_{\R} |\mu(x)| \rd{x}<\infty$. 
 
 Finally, we prove $\int_{\R} \mu = 0$. We take a contour integral of $g$ along
	\begin{equation}
		\begin{aligned}
		\gamma_{r, \epsilon} := &\{re^{i\theta}:0\le \theta\le \pi\} \cup  \Big(\cup_{P\in \{L_j, R_j \}_{j=1}^n} \{\epsilon e^{i\theta} + P:0\le \theta\le \pi\}\Big) \\
		& \cup \Big([-r, r]\backslash \cup_{P\in \{L_j, R_j \}_{j=1}^n}  [P-\epsilon, P+\epsilon] \Big) \\
		\end{aligned}
	\end{equation}
	for a large $r$ and small $\epsilon$, with counterclockwise direction. Since $g$ is analytic in a neighborhood of the domain enclosed by $\gamma = \gamma_{r,\epsilon}$, we have $\int_{\gamma} g(z) \rd{z} = 0$. Therefore its real part is
	\begin{equation}\label{eqn:contour}
\int_{\gamma \cap \R} (-1+ \mu_c(x))\rd{x} 
			+ \Re \Big(\int_{\{re^{i\theta}:0\le \theta\le \pi\}} g(z)\rd{z} \Big) 
			- \sum_{P \in \{ L_j, R_j \}_{j=1}^n} \Re \Big(\int_{\{\epsilon e^{i\theta}+P:0\le \theta\le \pi\}}
			g(z)\rd{z} \Big) =0.
\end{equation}

	The first term in \eqref{eqn:contour} converges to $\int_{\mathbb{R}}\mu(x)\rd{x}-\sum_{j=1}^n a_j$ as $r\rightarrow\infty$. For the second term, the contribution from $-1+ g_c(z)$ vanishes as $r\to \infty$, since on upper half plane $|-1+ g_c(re^{i\theta})| \le C/r^2$ under \eqref{lem_complex2_1}. The contribution from $g_d(z)$ for each $M_k$-term is
	\begin{equation}
\lim_{r\to\infty} \int_0^\pi \frac{a_k}{\pi}\frac{i}{re^{i\theta}-M_k}ire^{i\theta}\rd{\theta}=\lim_{r\to \infty}-\frac{a_k}{\pi}\int_0^\pi \frac{1}{1-\frac{M_k}{r}e^{i\theta}}\rd{\theta}=-a_k.
	\end{equation}
   The third term vanishes as $\epsilon\to 0$ since $g(z)$ is continuous in a neighborhood of $L_k$ and $R_k$. Therefore by letting $r\to \infty $ and $\epsilon\to 0$, we obtain $\int_{\mathbb{R}}\mu(x)\rd{x}=0$.
\end{proof}
\begin{remark}\label{rmk:stationary-R}
If \eqref{lem_complex2_1} holds then $|\mu(x)|\le \frac{C}{x^2}$ for large $|x|$, thus $\tilde{W}*\mu$ is well-defined on $\R$ except at $M_k$. We can imagine that $\mu_d$ generates an external potential $U = \tilde{W}*\mu_d$ and $\mu_c-1$ is a signed measure on $\R$. Since $(\tilde{W}*\mu)' =\pv \tilde{W}'*\mu = 0$ in $S=\supp \mu_c$, the total potential $V_U[\mu_c-1] = \tilde{W}*\mu$ is a constant on each connected component of $\supp \mu_c$. This is the reason why we consider $\mu$ as stationary distributions. 
\end{remark}

Then we apply the above lemma to obtain the following type I, II, III measures on $\mathbb{R}$, denoted as $\mu_\ti,\mu_\tii,\mu_\tiii$, which are stationary in the sense of Remark \ref{rmk:stationary-R}. Firstly by taking the sequence $-1/\pi <0<1/\pi$, we obtain
	\begin{equation}\label{typei}
		\textbf{Type \ti: }
		\mu_{\ti}(x) = -1+\frac{\sqrt{x^2-\pi^{-2}}}{|x|}\chi_{|x|\ge \pi^{-1}} + \delta(x), \quad (\tilde{W}*\mu_{\ti})'(x) = -\frac{\pi\sqrt{-(x^2-\pi^{-2})}}{x}\chi_{|x|\le \pi^{-1}}.
	\end{equation}
Next for $R>1$, by taking $-R< -1< 0=0<1<R$ we obtain
	\begin{equation}\label{typeii}
		\textbf{Type \tii: }
		\mu_{\tii,R}(x) = -1+\frac{\sqrt{(x^2-R^2)x^2}}{|x^2-1|}\chi_{|x|\ge R} + m(\delta_1+\delta_{-1}),
	\end{equation}
with
\begin{equation}\label{typeii-2}
m = \frac{\pi\sqrt{R^2-1}}{2}, \quad (\tilde{W}*\mu_{\tii})'(x) = -\frac{\pi x\sqrt{-(x^2-R^2)}}{x^2-1}\chi_{|x|\le R}.
\end{equation}
Finally, by taking $-R<-1< -L<0<L<1<R$ we obtain
\begin{equation}\label{typeiii}
		\textbf{Type \tiii: }
		\mu_{\tiii,L,R}(x) = -1+\frac{\sqrt{(x^2-R^2)(x^2-L^2)}}{|x^2-1|}\chi_{|x|\in [0,L]\cup [R,\infty)} + m(\delta_1+\delta_{-1}),
	\end{equation}
with
\begin{equation}\label{typeiii-2}
m = \frac{\pi\sqrt{(R^2-1)(1-L^2)}}{2},\quad (\tilde{W}*\mu_{\tiii})'(x) =- \frac{\pi\sgn(x)\sqrt{-(x^2-R^2)(x^2-L^2)}}{x^2-1}\chi_{L\le |x|\le R}.
\end{equation}

See Figure \ref{fig:over-R} for examples of constructions above. 

\subsection{Stationary Distributions over $\bT$}\label{ssec:stationary-T}
We now construct a family of stationary measures over $\bT$. It will contain the unique minimizer of $\cE_U$ with $U$ being generated by the Dirac mass configuration $\rho_d = m (\delta_M+\delta_{-M})$ in Proposition \ref{prop:energy-min} for all $m$ and $M$, therefore includes the unique minimizer of $\cG_{\alpha}$ in Theorem \ref{thm:main-characterization} (after subtracting $\rho_d$). 

The construction is similar to that of Lemma \ref{lem_complex}.  For this construction, we will always represent points in $\bT$ as $[-1/2, 1/2)$. We will use the conformal mapping $\omega(z)=i\frac{1-z}{1+z}$ that maps the unit disc to the upper half plane. Notice that $\omega(e^{2\pi i x})=\tan\pi x$ for any $-1/2<x<1/2$. We also similarly simplify the notation before the construction. Given a sequence of real numbers $-\frac{1}{2} < \kl_1 < \km_1 < \kr_1 \le \kl_2 < \km_2 < \kr_2 \le \cdots \le \kl_n < \km_n < \kr_n < \frac{1}{2}$, we denote
\begin{equation}
	 S:=[-1/2,\kl_1]\cup \bigcup_{j=1}^{n-1}[\kr_j,\kl_{j+1}]\cup [\kr_n,1/2)
\end{equation}
and
\begin{equation}
T_{\kl}(x):=\prod_{j=1}^n |\tan\pi x-\tan\pi \kl_j|, \quad T_{\kl, k}(x):= \prod_{j\neq k} |\tan\pi x-\tan\pi \kl_j|.
\end{equation}
Similarly for $T_{\km}(x)$ and $T_{\kr}(x)$. 
\begin{lemma}\label{lem_complexT}
Given a sequence of points on $\bT$ denoted by
	\begin{equation}
		-\frac{1}{2} < \kl_1 < \km_1 < \kr_1 \le \kl_2 < \km_2 < \kr_2 \le \cdots \le \kl_n < \km_n < \kr_n < \frac{1}{2}.
	\end{equation}
We define $\rho = \rho_c+\rho_d$ on $\mathbb{T}$ be given by 
	\begin{equation}
		\rho_c(x) =  \frac{\sqrt{T_{\kl}(x)T_{\kr}(x)}}{T_{\km}(x)}\chi_{S},\quad  \rho_d(x) = \sum_{j=1}^n \ka_j\delta(x-\km_j),
	\end{equation}
with
	\begin{equation}\label{Ak}
		\ka_k =\frac{\sqrt{T_{\kl}(\km_k)T_{\kr}(\km_k)}}{T_{\km,k}(\km_k)}\cdot \cos^2\pi \km_k>0.
	\end{equation}
Then there exists a constant $C_1$ such that 
	\begin{equation}
		(W*\rho)'(x)- \pi C_1= \left\{\begin{split}
			& -\pi\sgn(x-\km_k)\frac{\sqrt{T_{\kl}(x)T_{\kr}(x)}}{T_{\km}(x)},\quad x\in [\kl_k,\km_k)\cup(\km_k,\kr_k] \\
			& 0,\quad x\in S
		\end{split}\right.
	\end{equation}
\end{lemma}
\begin{proof}
	Let $L_j=w(e^{2\pi i \kl_j})=\tan\pi \kl_j,\,M_j = \tan\pi \km_j,\,R_j=\tan\pi \kr_j$. Then similar to the function $g$ in \eqref{g}, we define $\kg = \kg_c+\kg_d- C_{\kg}$ for $|z|\le 1$ with
	\begin{equation}
		\kg_c(z) = \frac{\prod_{j=1}^n\phi(\omega(z)-L_j)\phi(\omega(z)-R_j)}{\prod_{j=1}^n (\omega(z)-M_j)}, \quad \kg_d(z) = - \sum_{j=1}^n \ka_j \frac{i}{\omega(ze^{-2\pi i \km_j})},
	\end{equation}
	where $\phi$ is as appeared in \eqref{g}, $\ka_j$ is given in \eqref{Ak}, and $C_{\kg}=C_0+iC_1$ is a complex constant such that $\kg(0)=0$. Similarly with Lemma \ref{lem_complex}, here $\kg$ is analytic in the unit disc and continuous on $|z|=1$. At $z=-1$ we have $\lim_{z\rightarrow -1,\,|z|<1}\kg(z)=1- \sum_{j=1}^n \ka_j \frac{i}{\omega(-e^{-2\pi i \km_j})}$. 

	Then, similar to the proof of Lemma \ref{lem_complex}, we have
	\begin{equation}
	\Re(\kg)(e^{2\pi i x}) = \rho_c(x) - C_0,
	\end{equation}
	\begin{equation}
	\Im(\kg)(e^{2\pi i x}) = \sum_{k=1}^n\chi_{[\kl_k,\kr_k]}\sgn(x-\km_k)\frac{\sqrt{T_{\kl}(x)T_{\kr}(x)}}{T_{\km}(x)} -\sum_{j=1}^n \ka_j\cot\pi(x-\km_j) - C_1
	\end{equation}
	which are functions in $L^2(\mathbb{T})$. The kernel for the Hilbert transform on $\mathbb{T}$ is exactly
	\begin{equation}
		\cot\pi x  = -\frac{1}{\pi}W'(x).
	\end{equation}
    By construction $\kg(0) = 0$, we get
	\begin{equation}
		\Big(\pv  \cot\pi y* (\rho_c- C_0)\Big)(x) = \Im (\kg),
    \end{equation}
therefore
	\begin{equation}
		(W*\rho)'= - \pi \Big(\pv  \cot\pi y* \rho \Big)(x)
			= - \pi \Big(  \sum_{k=1}^n\chi_{[\kl_k,\kr_k]}\sgn(x-\km_k)\frac{\sqrt{T_{\kl}(x)T_{\kr}(x)}}{T_{\km}(x)} - C_1 \Big).
\end{equation}
\end{proof}

The above lemma is useful for our application to the energy minimizers only when $C_1=0$. Using $\omega(0)=i$ and $\kg_d(0) = -\sum_i \ka_i$, this condition is 
\begin{equation}\label{cond_C1}
\Im \Big( \frac{\prod_{j=1}^n\phi(i-L_j)\phi(i-R_j)}{\prod_{j=1}^n (i-M_j)} \Big) = 0
\end{equation}
Then we may apply the mean value principle to the analytic function $\kg$ to get
\begin{equation}
	\int_{\mathbb{T}}\rho_c \rd{x} =  C_0 =  \Re\Big(\frac{\prod_{j=1}^n\phi(i-L_j)\phi(i-R_j)}{\prod_{j=1}^n (i-M_j)}\Big) - \sum_i \ka_i,
\end{equation}
i.e.
\begin{equation}\label{cond_C0}
	\int_{\bT} \rho \rd{x} = \Re\Big(\frac{\prod_{j=1}^n\phi(i-L_j)\phi(i-R_j)}{\prod_{j=1}^n (i-M_j)}\Big).
\end{equation}

Then we will construct some energy minimizers. Using the sequence $-\frac{1}{\pi}\sin^{-1}2m <0 < \frac{1}{\pi}\sin^{-1}2m$ with $n=1$ in Lemma \ref{lem_complexT}, we obtain the following:
\begin{proposition}\label{prop_rhoi}
For $0<m\le 1/2$, we have the probability measures $\rho_{\ti,m}(x) = \rho_{\ti,c} + \rho_{\ti,d}$ on $\bT$ where
\begin{equation}\label{typeirho}
		\rho_{\ti, c}(x) = \sqrt{1-\frac{4m^2}{\sin^2\pi x}}\chi_{|x|\ge \frac{1}{\pi}\sin^{-1}2m}, \quad \rho_{\ti, d}(x) = 2m\delta(x). 
\end{equation}
The measure $\rho_{\ti, c}$ is the unique minimizer of $\cE_U$ in $\cM_{1-2m}$ for $U = W*\rho_{\ti, d}$. 
\end{proposition}
See Figure \ref{fig:over-T} (a) for an example of this construction. 
\begin{proof}
We verify \eqref{cond_C1} and invoke \eqref{cond_C0} by computing
\begin{equation}
	\frac{\phi(i-L_1)\phi(i-R_1)}{i-M_1} =\frac{ \phi(i+\tan(\sin^{-1}2m))\phi(i-\tan(\sin^{-1}2m)) }{i}= \frac{1}{\sqrt{1-4m^2}} = \int_{\bT} \rho \rd{x}, 
\end{equation}
where the corresponding measure $\rho$ in Lemma \ref{lem_complexT} is 
\begin{equation}
	\rho(x) = \frac{\sqrt{|(\tan\pi x - \tan(\sin^{-1}2m))(\tan\pi x + \tan(\sin^{-1}2m))|}}{|\tan\pi x|} + \frac{2m}{\sqrt{1-4m^2}}\delta(x).
\end{equation}
In order to construct a probability measure, we divide $\rho$ by its total mass and obtain $\rho_{\ti}$, which is the sum of $\rho_{\ti, c}$ and $\rho_{\ti, d}$ given in \eqref{typeirho}. 

By Lemma \ref{lem_complexT} we obtain the formula for $(W*\rho_{\ti})'$. It is positive in $[-\frac{1}{\pi}\sin^{-1}2m, 0)$ and negative in $(0, \frac{1}{\pi}\sin^{-1}2m]$ and zero otherwise, indicating that $W*\rho_{\ti}$ is smallest in $\supp \rho_{\ti,c}$. Therefore $\rho_{\ti, c} \in \cM_{1-2m}$ satisfies the characterizing condition \eqref{eqn:Emin-sediment} for $U = \rho_{\ti, d}*W$ with $M = 0$, and by Proposition \ref{prop:energy-min}, it is also the unique minimizer of $\cE_U$ and maximizer of $\ess\inf V_U$ in $\cM_{1-2m}$. 
\end{proof}
Similarly we can obtain the unique minimizer when $M\neq 0$ by applying Lemma \ref{lem_complexT}. Using the sequence $-R< -M< -L\le  L< M < R$ with $n=2$, we obtain the following:
\begin{proposition}\label{prop_rhoii}
For $0\le L < M < R < 1/2$, we have probability measures $	\rho_{\tii,M,R,L} = \rho_{\tii, c} + \rho_{\tii, d}$ on $\bT$ where
	\begin{equation}\begin{split}\label{typeiirho}
			\rho_{\tii,c}(x)  = & \frac{\sqrt{\sin\pi(x-R)\sin\pi(x+R)\sin\pi(x-L)\sin\pi(x+L)}}{\sin\pi(x-M)\sin\pi(x+M)}\chi_{|x|\in [0,L]\cup[R,1/2]} \\
			\rho_{\tii,d}(x) = &  m(\delta(x+M)+\delta(x-M)) 			
	\end{split}\end{equation}
where
	\begin{equation}\label{typeiirhoA}
		m = \frac{\sqrt{-\sin\pi(M-R)\sin\pi(M+R)\sin\pi(M-L)\sin\pi(M+L)}}{\sin2\pi M}.
	\end{equation}
The measure $\rho_{\tii, c}$ is the unique minimizer of $\cE_U$ where $U = W*\rho_{\tii, d}$ in $\cM_{m_1, m_2}$ with
	\begin{equation}\label{eqm1}
		m_1=\int_{(-M,M)}\rho_{\tii, M, R, L}\rd{x},\quad m_2=1-2m-m_1.
	\end{equation}
\end{proposition}
See Figure \ref{fig:over-T} (b), (c) and (d) for examples of this construction. 
\begin{proof}
Similarly with Proposition \ref{prop_rhoii}, we verify \eqref{cond_C1} and invoke \eqref{cond_C0} by computing
	\begin{equation}
		\frac{\prod_{j=1}^n\phi(i-L_j)\phi(i-R_j)}{\prod_{j=1}^n (i-M_j)} = \frac{\frac{i}{\cos\pi R}\cdot \frac{i}{\cos\pi L}}{-\frac{1}{\cos^2\pi M}} = \frac{\cos^2\pi M}{\cos\pi R\cos\pi L} = \int_{\bT} \rho \rd{x},
	\end{equation}
where the corresponding $\rho = \rho_c+\rho_d$ in Lemma \ref{lem_complexT} is
	\begin{equation}\begin{split}
			\rho_c(x) = & \frac{\sqrt{|(\tan^2\pi x-\tan^2\pi R)(\tan^2\pi x-\tan^2\pi L)|}}{|\tan^2\pi x-\tan^2\pi M|}\chi_{|x|\in [0,L]\cup[R,1/2]} \\
			\rho_d(x)= & \ka (\delta(x+M)+\delta(x-M)),
	\end{split}\end{equation} with
	\begin{equation}
		\ka = \frac{\sqrt{|(\tan^2\pi M-\tan^2\pi R)(\tan^2\pi M-\tan^2\pi L)|}}{|2\tan\pi M|}\cdot \cos^2\pi M.
	\end{equation}
We divide $\rho$ by its total mass to obtain $\rho_{\tii,M,R,L}(x)$, and still denote the corresponding terms by $\rho_{\tii, c}$ and $\rho_{\tii, d}$. 

By Lemma \ref{lem_complexT}, for $\rho_{\tii} = \rho_{\tii,M, R, L}$, we obtain the formula for $(W*\rho_{\tii}(x))'$. It indicates that $W*\rho_{\tii}$ is a constant in both connected components $[-L,L]$ and $\bT \backslash [-M, M]$ of $\supp \rho_{\tii, c}$. Therefore the measure $\rho_{\tii,c} \in \cM_{1-2m}$ satisfies the characterizing condition in \eqref{cor_Emin_1} for $m_1$ and $m_2$ given in \eqref{eqm1} and $U = W*\rho_{\tii, d}$, and by Corollary \ref{cor_Emin}, it is also the unique minimizer of $\cE_U$ in $\cM_{m_1,m_2}$.
\end{proof}

Combining Corollary \ref{cor_Emin} and Proposition \ref{prop_rhoi} and \ref{prop_rhoii}, given $U = W*m(\delta_M +\delta_{-M})$ and $m_1 \le 1-2m$, the unique minimizer of $\cE_U$ in $\cM_{m_1, 1-2m-m_1}$ must be in the form of either $\rho_{\ti, c}(x)$ (when $M = 0$) as in \eqref{typeirho} or $\rho_{\tii, c}(x)$ (when $M \neq 0$) as in \eqref{typeiirho}. Let's denote this unique minimizer in $\cM_{m_1, 1-2m-m_1}$ to be $\rho_{M, m, m_1}$. We thus obtain a three-parameter family of probability measures over $\bT$, 
\begin{equation}\label{eqn:3-parameter}
	\eta_{M, m, m_1}:=\rho_{M, m, m_1}+ m(\delta_M + \delta_{-M}),
\end{equation}
which includes the minimizer of $\cG_{\alpha}$ in $\cM_{\cD\ge d}$ by Theorem \ref{thm:main-characterization}. 

\subsection{Sediment Distributions over $\bT$}\label{ssec:sediment-T}
 In Section \ref{ssec:stationary-T} we have constructed a three-parameter family $\eta_{M, m, m_1} \in \cM$ in \eqref{eqn:3-parameter} that serves as candidates for the minimizer of $\cG_{\alpha}$ in $\cM_{\cD\ge d}$. These measures are all composed of Dirac masses in the form of $ m(\delta_M + \delta_{-M})$ together with a stationary measure in $\cM_{1-2m}$ with respect to $U = W*m(\delta_M + \delta_{-M})$. 
 
 We are going to show in this section that we can further narrow down candidates from stationary distributions to sediment distributions. While doing so, we will eliminate the parameter $m_1$ and reduce the number of parameters from three to two.

\begin{figure}
	\includegraphics[width=0.43\textwidth]{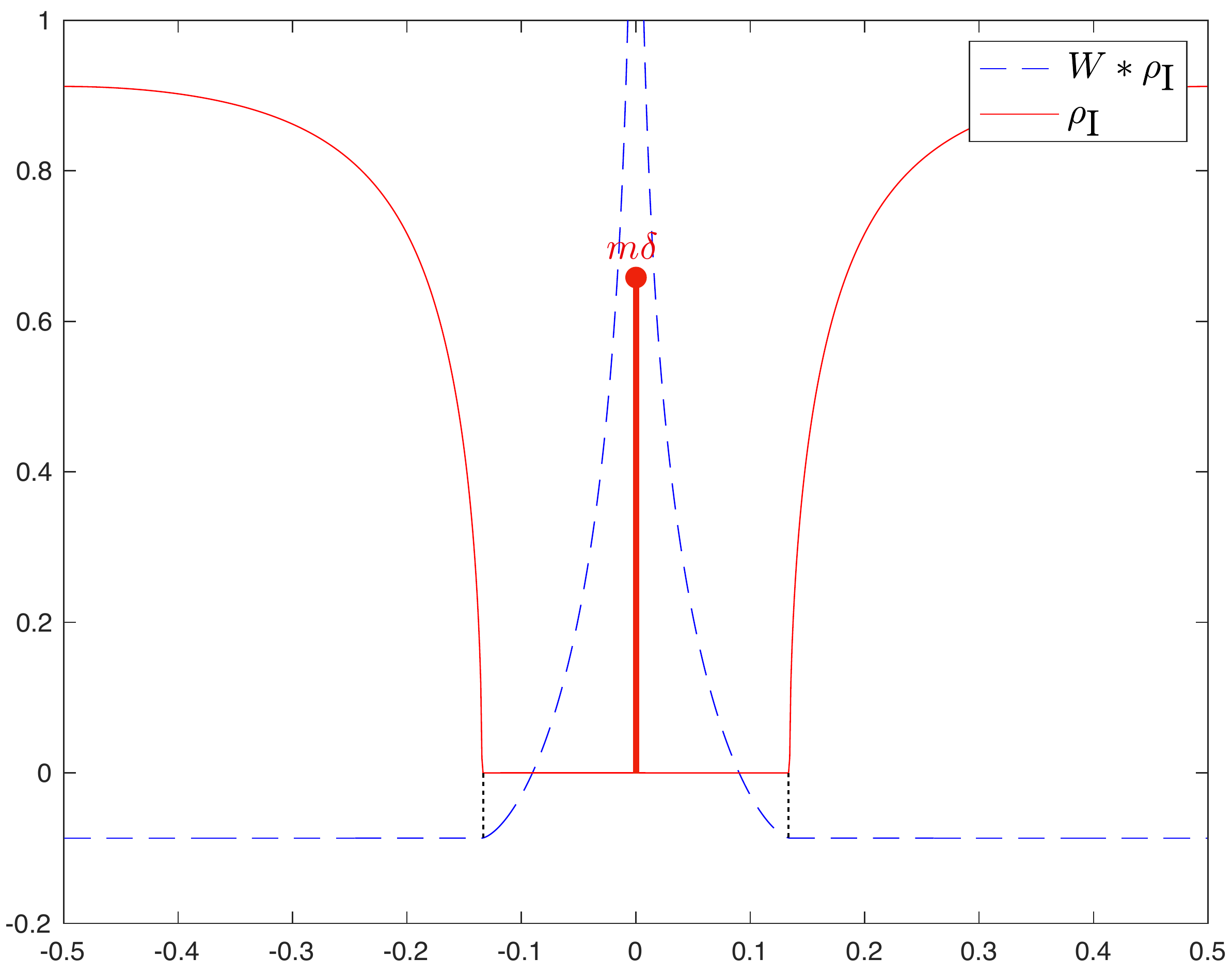}
	\includegraphics[width = 0.43\textwidth]{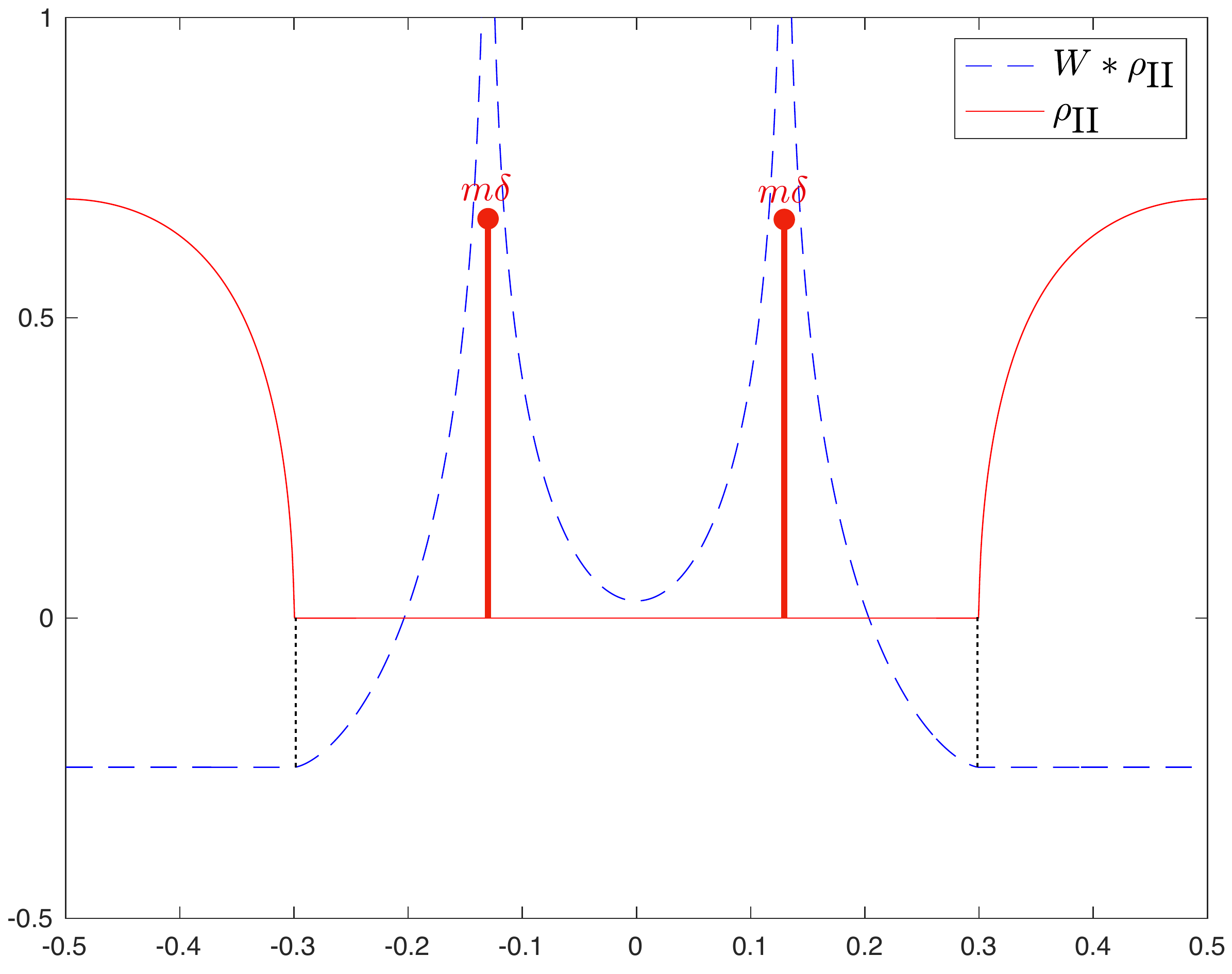}\\
	\includegraphics[width=0.43\textwidth]{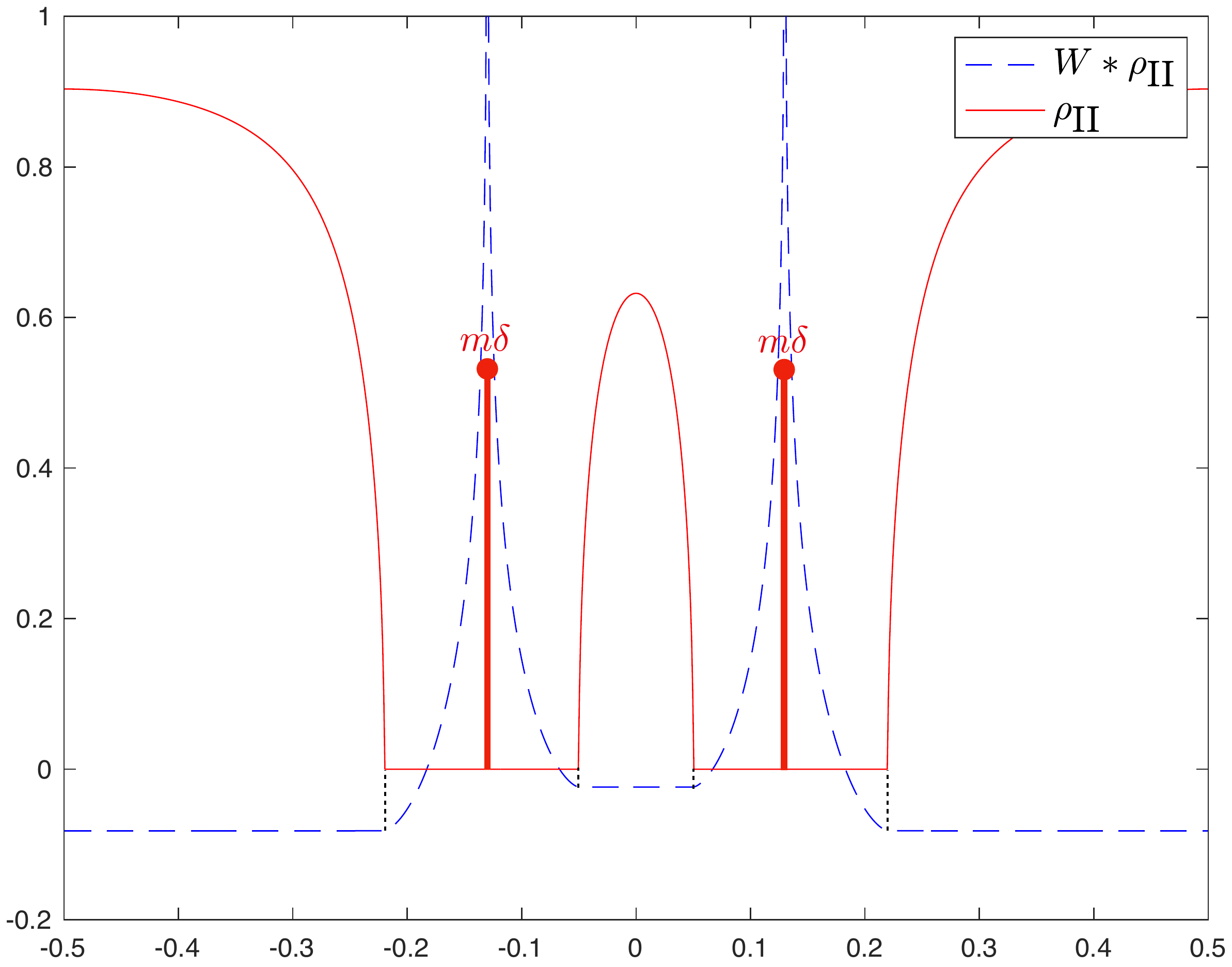}
	\includegraphics[width=0.43\textwidth]{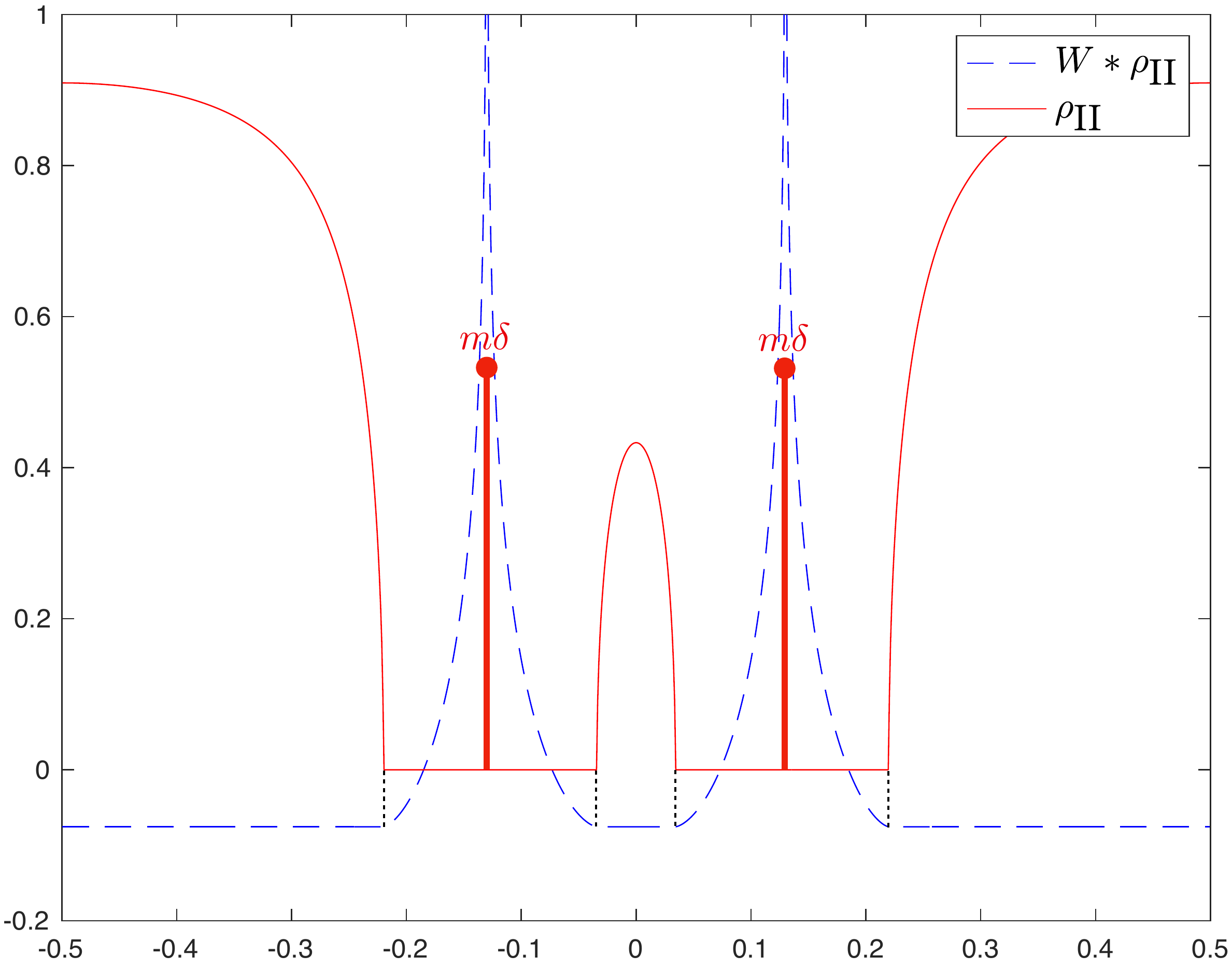}
	\caption{ Stationary Distribution over $\bT$: (a) Top left: Type $\ti$ over $\bT$ ($m= 0.41$); (b) Top right: Type $\tii$ over $\bT$ ($L=0$, $M=0.13$, $R =3$);
		(c) Bottom left: non-sediment Type $\tii$ over $\bT$ ($L=0.05$, $M=0.13$, $R=0.22$); (d) Bottom right: sediment Type $\tii$ over $\bT$ ($L\approx0.034$, $M=0.13$, $R=0.22$).}
	\label{fig:over-T}
\end{figure}

\begin{proposition}\label{prop:2-parameter}
	Fix $0<M<1/2$ and $0<m\le 1/2$. There exists a unique $m_1$ such that $\rho_{M, m, m_1}$ is a sediment distribution with respect to $U = W*m(\delta_M + \delta_{-M})$. Moreover, for fixed $M$ and $m$, $(W*\rho_{M,m, m_1})(0)-(W*\rho_{M, m, m_1})(1/2)$ is an increasing function of $m_1$. 
\end{proposition}
See Figure \ref{fig:over-T} (c) and (d) for a comparison between non-sediment distributions and sediment distributions. Before we prove Proposition \ref{prop:2-parameter}, we first give a lemma which is a comparison principle for moments against the decreasing function $\tilde{W}$.
\begin{lemma}\label{lem_dec}
	Let $U$ be a $C^1$ function on $(0,X)$ with $X\in (0,\infty]$ and $U'<0$, and $\mu_1,\mu_2$ be signed measures on $(0,X)$ with $\int_{(0,X)}\mu_1\rd{x}=\int_{(0,X)}\mu_2\rd{x}$. If
	\begin{equation}\label{lem_dec_0}
		\int_{(0,x)}\mu_1(y)\rd{y} \le \int_{(0,x)}\mu_2(y)\rd{y} ,\quad \forall x\in (0,X),
	\end{equation}
	and
\begin{equation}\label{lem_dec_1}
	\int_{(0,X)}(|\mu_1|+|\mu_2|)U\rd{x}<\infty,
\end{equation}
\begin{equation}\label{lem_dec_2}
   \lim_{x\rightarrow0+}U(x)\int_{(0,x)}(\mu_1(y)-\mu_2(y))\rd{y}=\lim_{x\rightarrow X^{-}}U(x)\int_{(0,x)}(\mu_1(y)-\mu_2(y))\rd{y}=0,
\end{equation}
	then
	\begin{equation}
		\int_{(0,X)}\mu_1 U\rd{x} \le \int_{(0,X)}\mu_2 U\rd{x}.
	\end{equation}
And strict inequality holds as long as $\mu_1\ne\mu_2$.
\end{lemma}

\begin{proof}
	Denote $m_i(x)=\int_{(0,x)}\mu_i(y)\rd{y}$. Then integration by parts gives
	\begin{equation}
		\int_{(\epsilon,R)}\mu_i U\rd{x} = U(R)\cdot m_i(R)-U(\epsilon)\cdot m_i(\epsilon) - \int_{(\epsilon,R)}m_i U'\rd{x},
	\end{equation}
	for $0<\epsilon<R<X$. Then
	\begin{equation}
		\begin{aligned}
		& \int_{(\epsilon,R)}\mu_2 U\rd{x}-\int_{(\epsilon,R)}\mu_1 U\rd{x} \\
		= & U(R)\cdot (m_2(R)-m_1(R))- U(\epsilon)\cdot (m_2(\epsilon)-m_1(\epsilon)) -\int_{(\epsilon,R)}(m_2-m_1) U'\rd{x}.
		\end{aligned}
	\end{equation}
	Taking $\epsilon\rightarrow0+$ and $R\rightarrow X-$, using the assumption \eqref{lem_dec_1} and $U'(x)<0$ and $m_1(x)<m_2(x)$ for $x\in (0, X)$, we get
	\begin{equation}
		\int_{(0,X)}\mu_2 U\rd{x}-\int_{(0,X)}\mu_1 U\rd{x}  = -\int_{(0,X)}(m_2-m_1) U'\rd{x} \ge 0,
	\end{equation}
	where the last inequality is strict if $\mu_1\ne\mu_2$.
\end{proof}

\begin{proof}[Proof of Proposition \ref{prop:2-parameter}]
	The uniqueness of $m_1$ follows directly from the uniqueness of minimizer of $\cE_U$ in $\cM_{1-2m}$ in Corollary \ref{cor_Emin}. 
	
	To show the monotonicity of $(W*\rho_{M, m, m_1})(0)-(W*\rho_{M, m, m_1})(1/2)$, we first show the monotonicity of $L$ and $R$ as a function of $m_1$ when $M$ and $m$ is fixed. The measure $\rho_{M, m, m_1}$ has the form of $\rho_{\tii, c}(x)$ as given in \eqref{typeiirho} where $R$ and $L$ are implicitly determined by \eqref{typeiirhoA} and \eqref{eqm1}. For $i = 1,2$, let's say $R_i$ and $L_i$ are determined by taking $m_1 = m_{1,i}$, and $\rho_i = \rho_{M, m, m_{1,i}}$ is the corresponding $\rho_{\tii, c}$ in \eqref{typeiirho}. Using $\sin\pi(M-R)\sin\pi(M+R)=(\cos2\pi R - \cos2\pi M)/2$, we can see from \eqref{typeiirhoA} that either $L_1<L_2<M<R_1< R_2$ or $L_2<L_1< M<R_2<R_1$. Now using \eqref{typeiirho} and \eqref{eqm1}, we see that if $m_{1,1}< m_{1,2}$ then $L_1<L_2<M<R_1< R_2$. 
	
	This implies that if $m_{1,1}< m_{1,2}$, then $\rho_1(x)\le \rho_2(x)$ for $0< x< M$ and $\rho_1(x) \ge \rho_2(x)$ for $M<x<1/2$. Therefore
		\begin{equation}
			\int_{|x|< X}\rho_1(x)\rd{x} \le \int_{|x|< X}\rho_2(x)\rd{x} 
		\end{equation}
		for any $0\le X \le 1/2$. Then, since $\rho_1,\rho_2,W$ are even and $W(x)$ is decreasing on $(0,\frac{1}{2})$, we have
	\begin{equation}
	(W*\rho_1)(0)= \int_{\mathbb{T}}W(x)\rho_1(x)\rd{x} \le \int_{\mathbb{T}}W(x)\rho_2(x)\rd{x}=(W*\rho_2)(0),
	\end{equation}
	by Lemma \ref{lem_dec} (with $X=1/2$). Similarly we have $(W*\rho_1)(1/2)\ge (W*\rho_2)(1/2)$ by applying Lemma \ref{lem_dec} to $\rho_i(1/2-x)$. Then the conclusion follows.
\end{proof}

\subsection{Admissible Distributions over $\R$}\label{ssec:admissible-R}
In this section, we will give a statement for distributions over $\R$ that serves as the analogue of Proposition \ref{prop:2-parameter} in Section \ref{ssec:sediment-T}. We will not only prove similar qualitative results like existence and uniqueness in Proposition \ref{prop:2-parameter}, but also provide quantitative analysis for this family of distributions over $\R$. 

Recall the three types of distributions defined in \eqref{typei}, \eqref{typeii} and \eqref{typeiii}. The distributions $\mu_{\tii}$ and $\mu_{\tiii}$ form a family parametrized by two parameters $R$ and $L$ where $\mu_{\tii}$ corresponds to cases where $L = 0$. We will first make the observation in Proposition \ref{prop:Wmu-infinity} that $(\tilde{W}*\mu)(\infty) = 0$ for all three types of $\mu$. Then by imposing the condition 
\begin{equation}\label{eqn:admissible}
(\tilde{W}*\mu)(x) = \ess \inf (\tilde{W}*\mu), \quad \text{ for } x\in \supp \mu_c, 
\end{equation}
we further cut down by one parameter and obtain a family of stationary measures $\mu$ over $\R$ just parametrized by $R$ in Proposition \ref{prop_mu}. We call this family of signed measures over $\R$ together with $\mu_{\ti}$ the \emph{admissible} distributions. Figure \ref{fig:over-T} (c) and (d) illustrate the difference between admissible distributions and non-admissible distributions. Although technically speaking they are not measures and do not satisfy the sediment condition, they serve a similar role on $\R$ with that of sediment measures (together with the Dirac masses) over $\bT$. Such connection will be made clearer in Section \ref{sec:R-to-T}.

\begin{proposition}\label{prop:Wmu-infinity}
For $\mu$ be given in \eqref{typei}, \eqref{typeii} or \eqref{typeiii}, we have
\begin{equation}\label{limln}
	\lim_{x\to \infty}(\tilde{W}*\mu)(x)=0.
\end{equation}
\end{proposition}
\begin{proof}
Let $\mu$ be given by \eqref{typei}, \eqref{typeii} or \eqref{typeiii}.  We know from Lemma \ref{lem_complex} that $\int_\mathbb{R}\mu\rd{x}=0$ and $|\mu(x)|\lesssim \frac{1}{|x|^2}$. To see \eqref{limln}, we use the mean-zero property of $\mu$ and write
	\begin{equation}\begin{split}
			\int_{\mathbb{R}}\tilde{W}(x-y)\mu(y)\rd{y} = & \int_{\mathbb{R}}(\tilde{W}(x-y)-\tilde{W}(x))\mu(y)\rd{y} = \int_{\mathbb{R}}-\ln|1-\frac{y}{x}|\cdot \mu(y)\rd{y} \\
	\end{split}\end{equation}
	for large $|x|$. We can assume $x>0$ since $\tilde{W}*\mu$ is even. 
	
	For $|y|\le \sqrt{x}/2$, we have $\big|\ln|1-\frac{y}{x}|\big| \lesssim \frac{1}{\sqrt{x}}$, thus
	\begin{equation}
		\left|\int_{|y|\le \sqrt{x}/2}-\ln|1-\frac{y}{x}|\cdot \mu(y)\rd{y}\right|\lesssim \frac{1}{\sqrt{x}}\int_{\R}|\mu(y)|\rd{y}.
	\end{equation}
Next for $\sqrt{x}/2<|y|\le x/2$, similarly we have $\big|\ln|1-\frac{y}{x}|\big| \lesssim \frac{|y|}{x}$ and $|\mu(y)| \lesssim 1/y^2$, thus
	\begin{equation}
		\left|\int_{\sqrt{x}/2<|y|\le x/2}-\ln|1-\frac{y}{x}|\cdot\mu(y)\rd{y}\right| \lesssim \frac{1}{x}\int_{\sqrt{x}/2<|y|\le x/2}\frac{1}{|y|}\rd{y} \lesssim \frac{\ln x}{x}.
	\end{equation}
Finally for $|y|>x/2$ using a change of variable $y=x y_1$ we get
	\begin{equation}
		\left|\int_{|y|>x/2}-\ln|1-\frac{y}{x}|\cdot \mu(y)\rd{y}\right| \lesssim \int_{|y|>x/2}\big|\ln|1-\frac{y}{x}|\big|\frac{1}{y^2}\rd{y} \lesssim \frac{1}{x}\int_{|y_1|>1/2}\big|\ln|1-y_1|\big|\frac{1}{y_1^2}\rd{y_1} \lesssim \frac{1}{x}.
	\end{equation}
Therefore \eqref{limln} is proved. 
\end{proof}
\begin{remark}\label{rmk:Wmu-infinity}
It follows from Proposition \ref{prop:Wmu-infinity} and the stationary property of $\mu$ (in  \eqref{eqn:diff}) that $(\tilde{W}*\mu)(x)=0$ when $|x|\ge 1/\pi$ for Type $\ti$ and $(\tilde{W}*\mu)(x)=0$ when $|x|\ge R$ for Type $\tii$ and $\tiii$.
\end{remark}

\begin{figure}
	\includegraphics[width=0.43\textwidth]{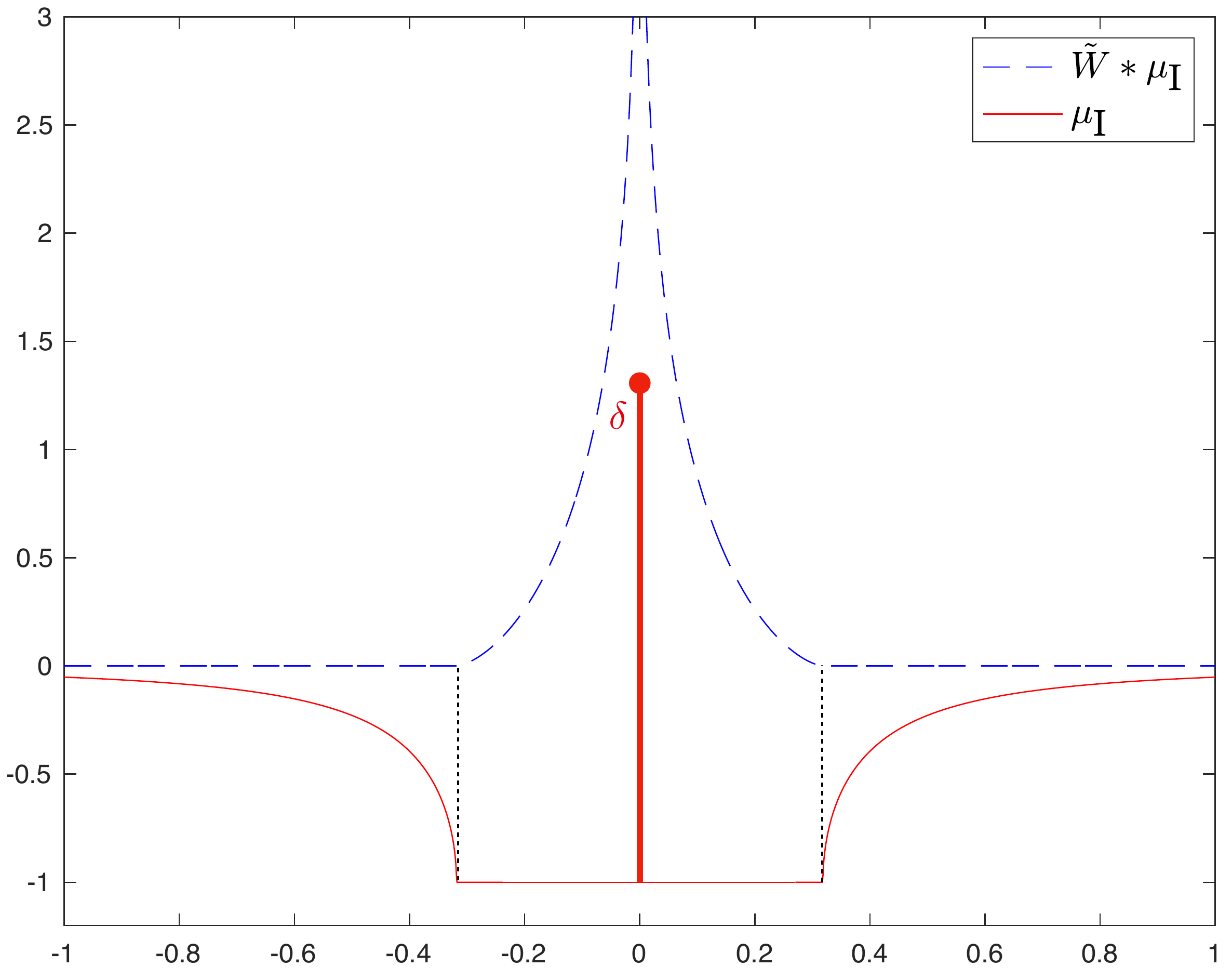}
	\includegraphics[width = 0.43\textwidth]{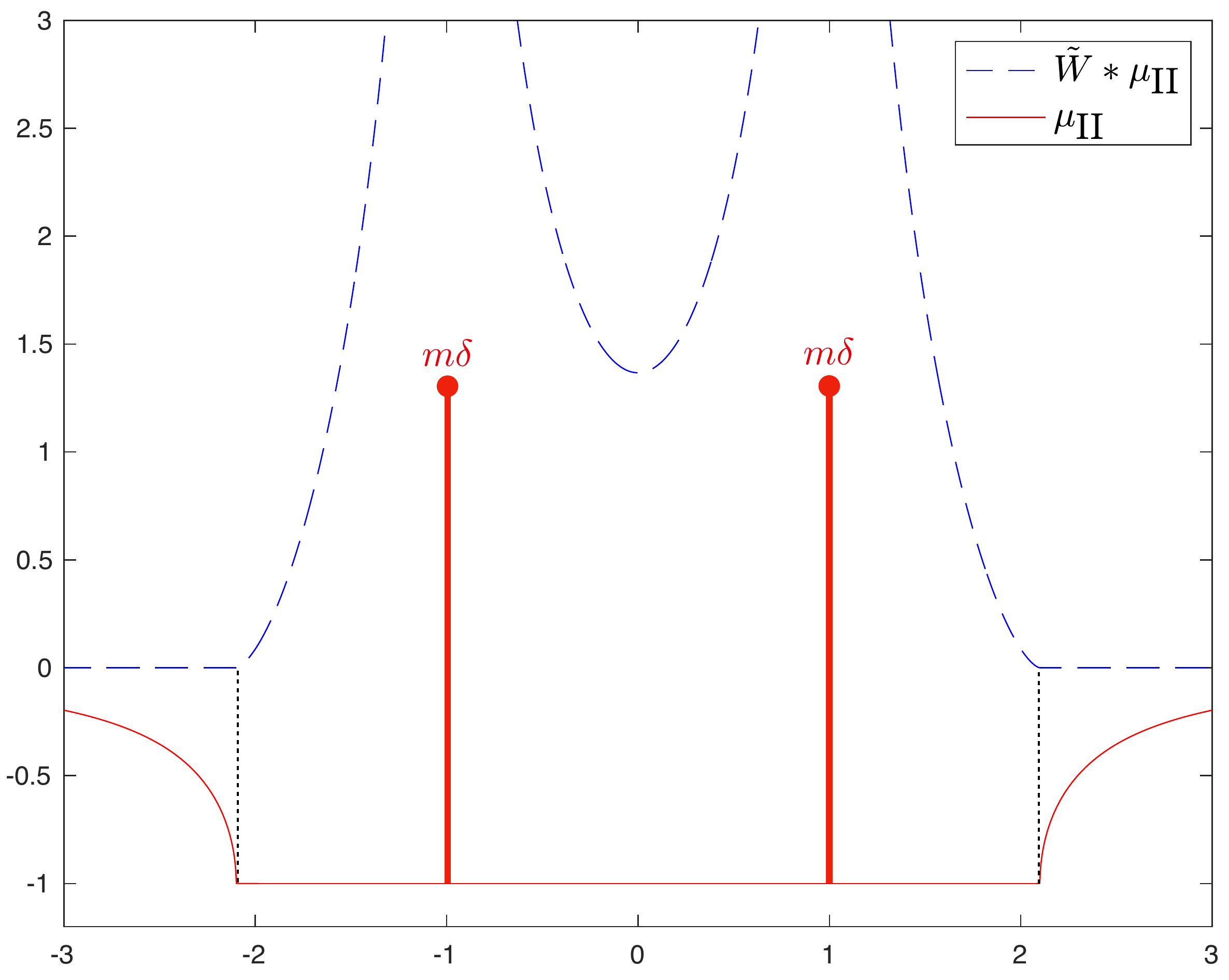}\\
	\includegraphics[width=0.43\textwidth]{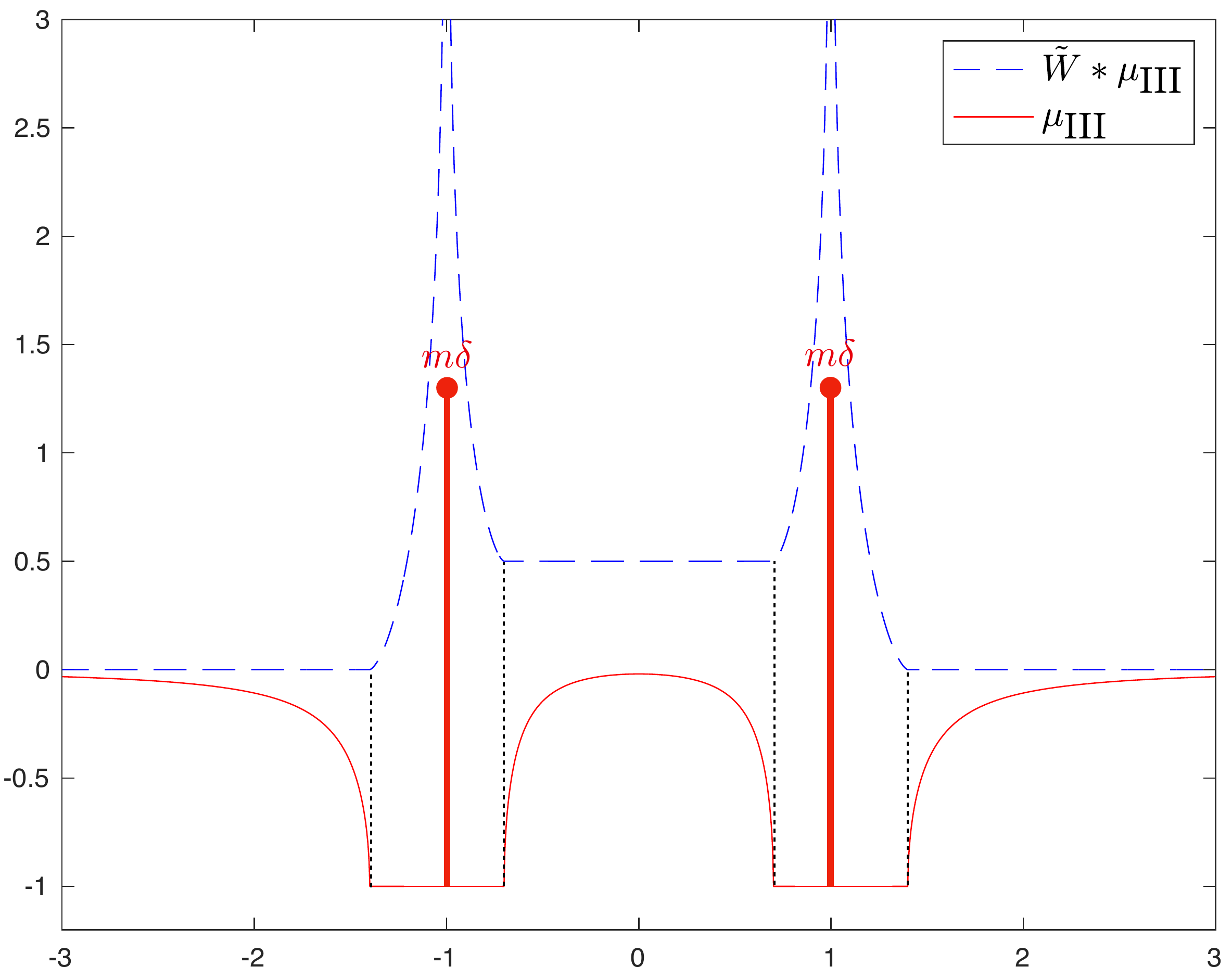}
	\includegraphics[width=0.43\textwidth]{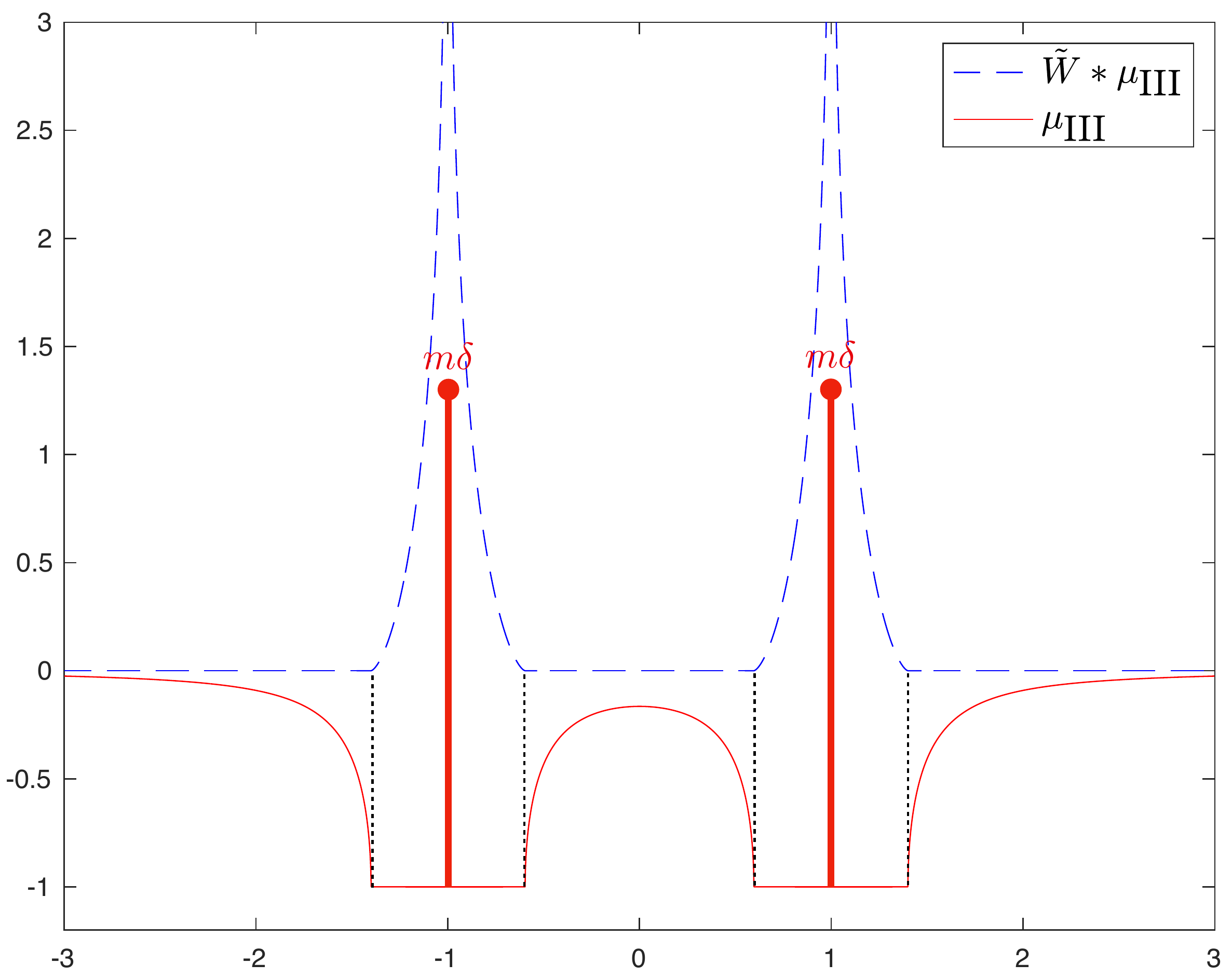}
	\caption{Stationary Distribution over $\R$: (a) Top left: Type $\ti$ over $\R$; (b) Top right: Type $\tii$ over $\R$ ($L=0$, $R =2.1$);
		(c) Bottom left: non-admissible Type $\tiii$ over $\R$ ($L=0.7$, $R=1.4$); (d) Bottom right: admissible Type $\tiii$ over $\R$ ($L\approx 0.60$, $R = 1.4$).}
	\label{fig:over-R}
\end{figure}

In the following proposition, by imposing \eqref{eqn:admissible}, we will further get rid of the parameter $L$ and obtain a one-parameter family of distributions parametrized by $R$. We first explain how the condition \eqref{eqn:admissible} is equivalent to a condition on $(\tilde{W}*\mu)(0)$. For Type $\tii$, we recall in Lemma \ref{lem_complex} that $\supp \mu_c = (-\infty, -R] \cup [R, \infty)$ and $(\tilde{W}*\mu)' = \pv \tilde{W}'*\mu$ is positive in $[0, 1)$ and negative in $(1, R]$ and $\tilde{W}*\mu$ approaches to $\infty$ as $x\to 1$. Therefore \eqref{eqn:admissible} is equivalent to $(\tilde{W}*\mu)(0)\ge 0$ since $(\tilde{W}*\mu )(\infty) = 0$ by Proposition \ref{prop:Wmu-infinity}. For Type $\tiii$, we recall that $\supp \mu_c = (-\infty, -R] \cup [-L,L] \cup[R, \infty)$ and $(\tilde{W}*\mu)'$ is positive in $[L, 1)$ and negative in $(1, R]$, therefore \eqref{eqn:admissible} is equivalent to $(\tilde{W}*\mu)(0)=0$. 

We now give the formula for $(\tilde{W}*\mu)(0)$ for $\mu$ constructed in \eqref{typeii} and \eqref{typeiii} by integrating against $(\tilde{W}*\mu)'$. Although $\tilde{W}*\mu$ is not continuous at $x=1$, we still have
\begin{equation}\label{eqn:pv}
	\begin{split}
\pv\int_L^R (\tilde{W}*\mu)'(x)\rd{x} & = \lim_{\epsilon \to 0^+} \Big((\tilde{W}*\mu)(1-\epsilon) -  (\tilde{W}*\mu)(L) + (\tilde{W}*\mu)(R) -  (\tilde{W}*\mu)(1+\epsilon) \Big)\\
& =( \tilde{W}*\mu)(R) - ( \tilde{W}*\mu)(L) = -( \tilde{W}*\mu)(L),
\end{split}\end{equation}
where we plug in $\mu = (-1+\mu_c) + \mu_d$ by Lemma \ref{lem_complex} for the second equality. The third equality follows from Remark \ref{rmk:Wmu-infinity}. Then by the explicit expression of $(\tilde{W}*\mu)'$ in \eqref{typeii-2} and \eqref{typeiii-2} we obtain
\begin{equation}\label{eqn:Phi}
	(\tilde{W}*\mu)(0) = -\pv\int_L^R (\tilde{W}*\mu)'(x)\rd{x} = \pv\int_L^R\frac{\pi\sqrt{-(x^2-R^2)(x^2-L^2)}}{x^2-1}\rd{x}: = \Phi(L, R).
\end{equation}
We denote the last integral by $\Phi(L, R)$. 

In order to study $(\tilde{W}*\mu)(0)$ for Type $\tii$ and $\tiii$ distributions over $\R$, we first prove some useful properties of this function $\Phi(L, R)$. 

\begin{lemma}\label{lem:Phi}
	For $0\le L <1<R$, let $\Phi(L, R)$ be given in \eqref{eqn:Phi}. 
	\begin{enumerate}
		\item[\textnormal{(i)}]The function $\Phi(L, R)$ is strictly increasing in both $L$ and $R$. 
		\item[\textnormal{(ii)}] There exists a unique real number $R= R_c \approx 1.8102$ such that $\Phi(0, R) = 0$. 
		\item[\textnormal{(iii)}] If $\Phi(L, R) = 0$, then $L+R<2$ and $L^2+R^2>2$.
	\end{enumerate}
\end{lemma}
See Figure \ref{fig:Phi} (a) for a graph of $\Phi(L, R)  = 0$. 
\begin{proof}
	\textbf{Proof of (i)}:
We denote $\mu_{L,R}$ for $\mu$ of Type $\tii$ (where $L=0$) and Type $\tiii$. Recall the expression of $\mu_{L, R}$ and $m(L, R)$ in \eqref{typeiii-2}.
	If $R_2>R_1$, then $m(L,R_2)>m(L,R_1)$, and 
	\begin{equation}
		\mu_{L,R_2}(x) \le  \mu_{L,R_1}(x) \text{ for } |x|>1, \quad \mu_{L,R_2}(x) \ge  \mu_{L,R_1}(x) \text{ for } |x|\le1.
	\end{equation}
	This verifies the condition \eqref{lem_dec_0} with $\mu_i=\mu_{L,R_i}$ using the mean zero property of $\mu_{\tiii}(x)$. One can also check the condition \eqref{lem_dec_1} with $U=\tilde{W}$ since $\mu_i$ are $L^\infty$ functions near 0, have the mean-zero property, and decay at least  as $1/|x|^2$ at infinity. Therefore, noticing that $\mu_i$ are even, we can apply Lemma \ref{lem_dec} (with $X=\infty$) to get $\Phi(L,R_2)> \Phi(L,R_1)$ by the decreasing property of $\tilde{W}(x)$ in $|x|$. Similarly one can prove the monotonicity in $L$.
	
	\textbf{Proof of (ii)}:
	Since $\Phi(0, R)$ is strictly increasing in $R$, it suffices to show that 
	\begin{equation}
		\lim_{R\to 1^+}\Phi(0, R)<0, \quad \lim_{R\to \infty}\Phi(0, R)>0.
	\end{equation}
	Notice that $\Phi(0, R)$ can be explicitly evaluated as
	\begin{equation}
		\pv\int_0^R\frac{x\sqrt{-(x^2-R^2)}}{x^2-1}\rd{x} = \sqrt{R^2-1}\ln(R+\sqrt{R^2-1})-R, \quad \text{ for } R>1,
	\end{equation}
     which indicates the sign of limit as $R\to 1^{+}$ and $R\to \infty$. We denote the unique root in $(1, \infty)$ of $ \sqrt{R^2-1}\ln(R+\sqrt{R^2-1})-R = 0$ by $R_c$. It is approximately $1.8102$.
	
	\textbf{Proof of (iii)}:
	Note that it follows $\Phi(0, R_c) = 0$ and the monotonicity of $\Phi$ that if $\Phi(L, R) =0$ for some $L\ge 0$, then $R\le R_c$. Now since $\Phi(L, R)$ is increasing in both $L$ and $R$, it suffices to show for any $R_c\ge R>1$ that
	\begin{equation}
		\Phi(2-R, R) > 0,\quad \Phi(\sqrt{2-R^2}, R) < 0.
	\end{equation}
	
	To see $\Phi(\sqrt{2-R^2},R) < 0$, we use a change of variable $u=x^2$ and then write $\frac{2}{\pi}\Phi(\sqrt{2-R^2},R)$ as
	\begin{equation}\begin{split}
		& \pv\int_{2-R^2}^{R^2} \frac{\sqrt{(R^2-u)(u-(2-R^2))}}{(u-1)\sqrt{u}}\rd{u} \\
			= & \lim_{\epsilon\rightarrow0+} \int_{2-R^2}^{1-\epsilon} \frac{\sqrt{(R^2-u)(u-(2-R^2))}}{(u-1)\sqrt{u}}\rd{u} + \int_{1+\epsilon}^{R^2} \frac{\sqrt{(R^2-u)(u-(2-R^2))}}{(u-1)\sqrt{u}}\rd{u} \\
			= & \lim_{\epsilon\rightarrow0+} -\int_{1+\epsilon}^{R^2} \frac{\sqrt{(R^2-u)(u-(2-R^2))}}{(u-1)\sqrt{2-u}}\rd{u} + \int_{1+\epsilon}^{R^2} \frac{\sqrt{(R^2-u)(u-(2-R^2))}}{(u-1)\sqrt{u}}\rd{u} \\
			= & \lim_{\epsilon\rightarrow0+}\int_{1+\epsilon}^{R^2} \frac{\sqrt{(R^2-u)(u-(2-R^2))}}{(u-1)}(\frac{1}{\sqrt{u}}-\frac{1}{\sqrt{2-u}})\rd{u} < 0. \\
	\end{split}\end{equation}
	
	Similarly to see $\Phi(2-R,R) > 0$, we write $\frac{1}{\pi}\Phi(2-R,R)$ as
	\begin{equation}\begin{split}
		 & \pv\int_{2-R}^R\frac{\sqrt{(R^2-x^2)(x^2-(2-R)^2)}}{x^2-1}\rd{x} \\
			= & \lim_{\epsilon\rightarrow0+} \int_{2-R}^{1-\epsilon}\frac{\sqrt{(R^2-x^2)(x^2-(2-R)^2)}}{x^2-1}\rd{x} + \int_{1+\epsilon}^R\frac{\sqrt{(R^2-x^2)(x^2-(2-R)^2)}}{x^2-1}\rd{x}\\
 = & \lim_{\epsilon\rightarrow0+} \!\!\int_{1+\epsilon}^R \!\!\! \frac{\sqrt{(R^2-(2-x)^2)((2-x)^2-(2-R)^2)}}{(2-x)^2-1}\rd{x} +\!\!\! \int_{1+\epsilon}^R \!\!\!\frac{\sqrt{(R^2-x^2)(x^2-(2-R)^2)}}{x^2-1}\rd{x}	\\
		\!	= & \lim_{\epsilon\rightarrow0+} \int_{1+\epsilon}^R\frac{\sqrt{(R^2-(2-x)^2)((2-x)^2-(2-R)^2)}}{(2-x)^2-1}+\frac{\sqrt{(R^2-x^2)(x^2-(2-R)^2)}}{x^2-1}\rd{x}.\\
	\end{split}\end{equation}
The last integrand is clearly positive since
	\begin{equation}\begin{split}
		&-(R^2-(2-x)^2)((2-x)^2-(2-R)^2)(x^2-1)^2 + (R^2-x^2)(x^2-(2-R)^2)(1-(2-x)^2)^2 \\
		=& -8(R-1)^2(x-1)^3(R(2-R)-x(2-x)) > 0. 
	\end{split}\end{equation}
\end{proof}

Now we are ready to state the following. 
\begin{proposition}\label{prop_mu}
	Let $\mu$ be given in \eqref{typeii} or \eqref{typeiii}. 
	\begin{enumerate} 
		\item[\textnormal{(i)}] For Type $\tii$, $(\tilde{W}*\mu_{\tii})(0)\ge 0$ iff $R\ge R_c$, where $R_c \approx 1.8102$ is the unique number in $(1, \infty)$ satisfying
		\begin{equation}
			\sqrt{R^2-1}\ln\Big(R+\sqrt{R^2-1}\Big) -R= 0.
		\end{equation} In this case $\mu_{\tii}(x)<0$ for $|x|\ne 1$. 
		\item[\textnormal{(ii)}]For Type $\tiii$, $(\tilde{W}*\mu_{\tiii})(0) =0$ iff $1<R<R_c$ and $L=L(R)$, where $L(R)\in (0,1)$ is the unique number satisfying
		\begin{equation}\label{LRcond}
			\Phi(L,R):=\pv\int_L^R \frac{\pi\sqrt{-(x^2-R^2)(x^2-L^2)}}{x^2-1}\rd{x} = 0.
		\end{equation}
	In this case $\mu_{\tiii}(x)<0$ for $|x|\ne 1$. This function $L(R)$ is decreasing in $R$.
	\end{enumerate}
\end{proposition}
\begin{proof}[Proof of Proposition \ref{prop_mu}]
\textbf{Proof of (i)}: 
	Lemma \ref{lem:Phi}(i) and (ii) implies that $\Phi(0, R) = (\tilde{W}*\mu_{\tii})(0)\ge 0$ if and only if $R\ge R_c$. To see that $\mu_{\tii}(x)<0$ for $|x|\ne 1$, we notice that for $x\ge R$,
	\begin{equation}
		\mu_{\tii}(x) = -1 + \sqrt{\frac{1-\frac{R^2}{x^2}}{(1-\frac{1}{x^2})^2}} <  -1 + \sqrt{\frac{1-\frac{R^2}{x^2}}{1-\frac{2}{x^2}}} < 0, \text{ for } R\ge R_c. 
	\end{equation}
	
	\textbf{ Proof of (ii)}: 
By Lemma \ref{lem:Phi}, the function $\Phi(L, R)$ is increasing in both $L$ and $R$. Therefore for each $1<R<R_c$ we have $\Phi(0,R)< \Phi(0, R_c) = 0$. On the other hand, by \eqref{eqn:Phi},
	\begin{equation}
		\lim_{L\rightarrow1-}\Phi(L,R) = \pi\int_1^R\sqrt{\frac{R^2-x^2}{x^2-1}}\rd{x} > 0.
	\end{equation}
	Therefore, there exists a unique $L=L(R)$ such that $\Phi(L,R)=0$. The decreasing property of $L(R)$ then follows from the monotonicity of $\Phi(L, R)$ in $R$ and $L$. To see $\mu_{\tiii}(x)<0$ for $|x|\neq 1$ when $1<R<R_c$ and $L=L(R)$, it suffices to note
	\begin{equation}\begin{split}
			(x^2-R^2)(x^2-L^2) = & x^4 - (R^2+L^2)x^2 + R^2L^2\\
			= &x^4 - (R^2+L^2)x^2 + \frac{1}{4}\big((R+L)^2-(R^2+L^2)\big)^2 \\ 
			< & x^4 - 2x^2 + \frac{1}{4}(4-2)^2 = (x^2-1)^2.
	\end{split}\end{equation}
	The last inequality follows from Lemma \ref{lem:Phi}(iii).
\end{proof}

It follows from Proposition \ref{prop_mu} that for each $R>1$, we get a unique stationary distribution satisfying \eqref{eqn:admissible}. We now collect these distributions and the Type $\ti$ distribution and their scaling functions into a family. 
\begin{definition}[Admissible Distributions]\label{def:admissible}
Let $\mu$ be a signed measure on $\R$. Then we say $\mu$ is \emph{admissible} if it is one of the following:
\begin{itemize}
	\item $\mu_{\ti}(\frac{\cdot}{\lambda})$ for some $\lambda>0$,
	\item $\mu_{\tii,R}(\frac{\cdot}{\lambda})$ for $R\ge R_c$ and some $\lambda>0$,
	\item $\mu_{\tiii,L(R),R}(\frac{\cdot}{\lambda})$ for $1<R< R_c$ and some $\lambda>0$.
\end{itemize}
We call $\lambda$ the \emph{scaling factor} of $\mu$. 
\end{definition}
Aside from the Type $\ti$ distribution, this family is parametrized by two parameters $\lambda$ and $R$. By the notation in Lemma \ref{lem_complex}, we have $\mu = -1+\mu_c+\mu_d$. By abuse of notation, we will also denote the corresponding term for a general admissible $\mu$ with a scaling factor $\lambda$ by $\mu_c$ and $\mu_d$. 

\section{An Optimization Problem over $\R$}\label{sec:min-value-R}
In this section, our goal is to determine the minimal value of an analogous goal functional $\tilde{\cG}$ for admissible distributions over $\R$. Recall that we have constructed and defined admissible distributions in Section \ref{ssec:admissible-R} and Definition \ref{def:admissible}. Now we define for admissible $\mu$ that
\begin{equation}\label{tildeH}
	\tilde{\cH}[\mu] = \int_{\mathbb{R}} \tilde{W}*\mu\rd{x},\quad \tilde{\cD}[\mu_{\ti}(\frac{\cdot}{\lambda})] = \lambda,\quad \tilde{\cD}[\mu_{\tii \text{ or }\tiii}(\frac{\cdot}{\lambda})] = \int_{[-\lambda,\lambda]}\mu_{\tii \text{ or }\tiii}(\frac{x}{\lambda})\rd{x},
\end{equation}
and
\begin{equation}\label{eqn:def-tilde-G}
	\tilde{\cG}[\mu]=\frac{\tilde{H}[\mu]}{\tilde{D}^2[\mu]}.
\end{equation}
The relation between $\tilde{\cH}$, $\tilde{\cD}$ and the functionals $\cH$, $\cD$ on $\mathbb{T}$ will be revealed in Section \ref{sec:R-to-T} by Theorem \ref{prop:rho-circ-H-D} and Theorem \ref{thm_Dcomp}. Then we state the main result of this chapter as an optimization problem.
\begin{theorem}\label{thm:min-value-R}
	For admissible $\mu$, we have $\tilde{\cG}[\mu] \ge \frac{1}{2}$. The equality holds if and only if $\mu$ is of Type $\ti$.
\end{theorem}
See Figure \ref{fig:Phi} (b) for a graph of $\tilde{\cG}[\mu]$ for $\mu$ of Type $\tii$ and $\tiii$ indexed by $R$. 
\begin{figure}
	\includegraphics[width=7cm, height = 5cm]{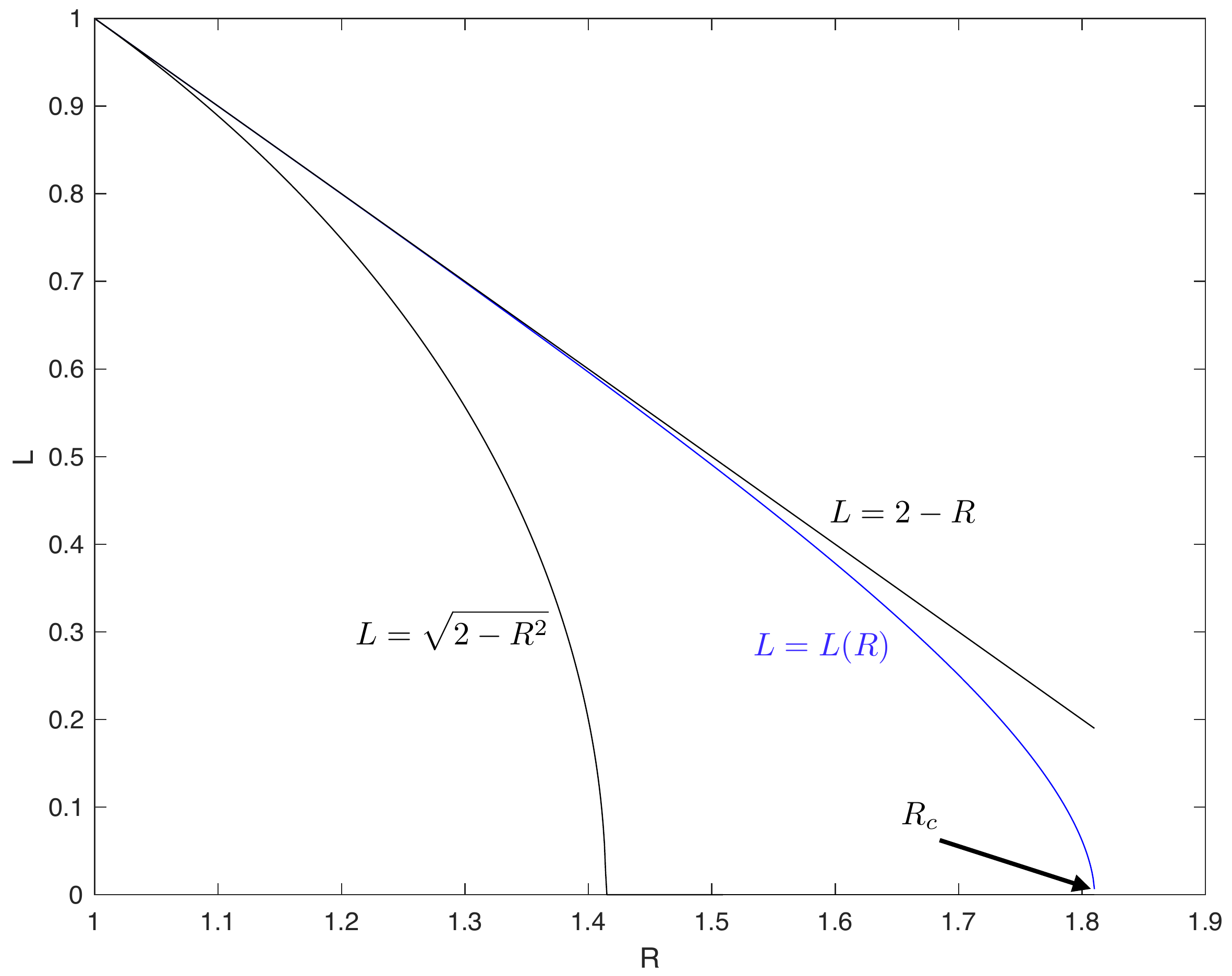}
	\includegraphics[width= 7cm, height = 5cm]{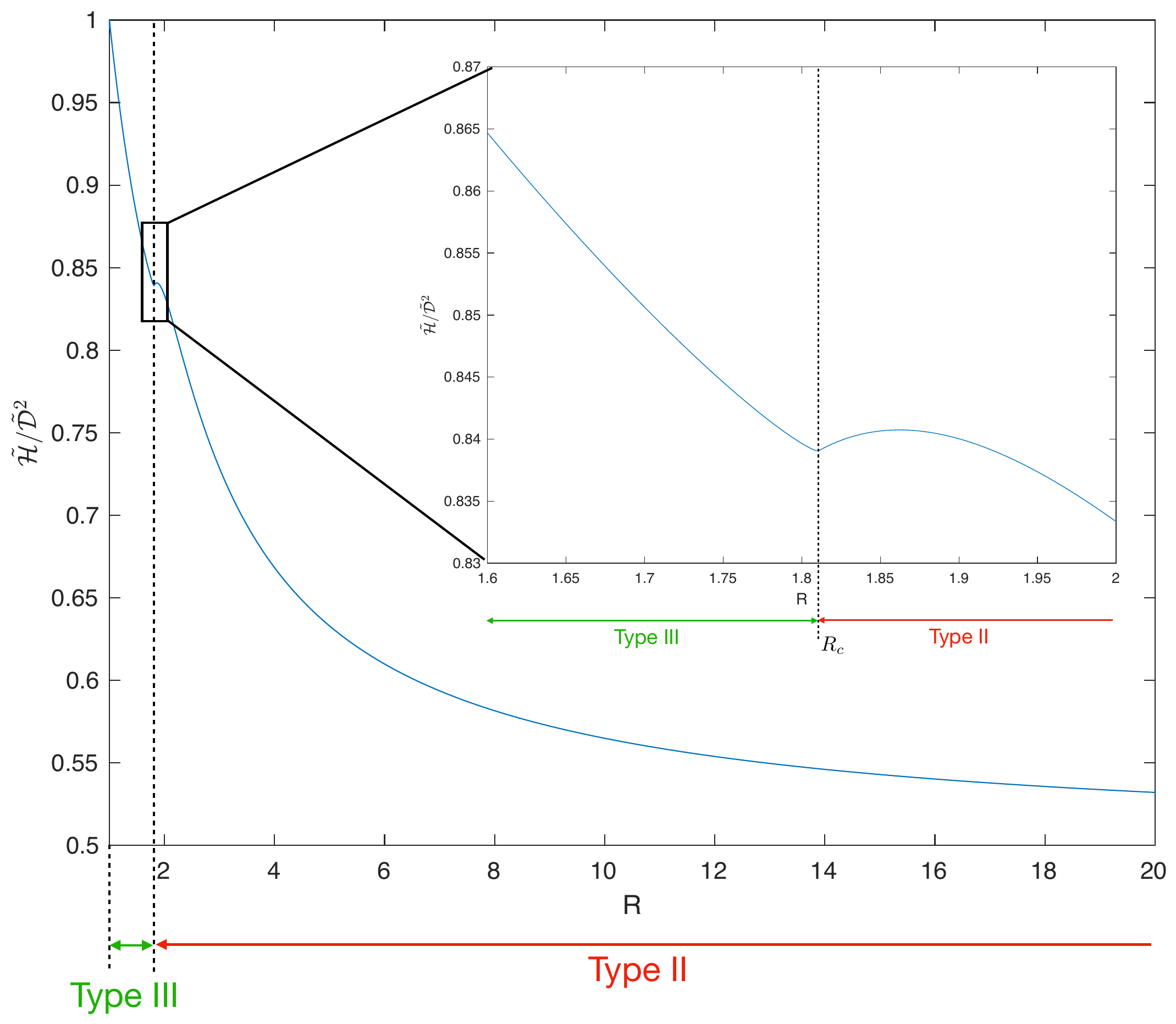}\\
	\caption{ (a) Left: Graph for $\Phi(L, R) = 0$; (b) Right: Graph for $\tilde{\cG}[\mu_{L(R), R}]$. }
	\label{fig:Phi}
\end{figure}

\subsection{Dimension Reduction}\label{ssec:dimension-reduction}
Recall that admissible distribution is either of Type $\ti$, parametrized by $\lambda>0$, or of Type $\tii$ and $\tiii$, parametrized by two parameters $\lambda>0$ and $R>1$. In this section, we will show that the goal functional $\tilde{\cG}$ remains invariant under scaling, i.e., $\tilde{\cG}(\mu)$ does not depend on the $\lambda$ parameter in all cases. 
\begin{lemma}\label{lem:scaling}
	If $\mu$ is admissible, then $\tilde{\cG}[\mu] = \tilde{\cG}[\mu(\frac{\cdot}{\lambda})]$ for any $\lambda>0$.
\end{lemma}
\begin{proof}
	By change of variables,
	\begin{equation}\begin{split}
			\tilde{\cH}[\mu(\frac{\cdot}{\lambda})] = & \int_{\mathbb{R}}(\tilde{W}*\mu(\frac{\cdot}{\lambda}))(x)\rd{x} = -\int_{\mathbb{R}}\int_{\mathbb{R}}\ln|y|\mu(\frac{x-y}{\lambda})\rd{y}\rd{x}\\
			= & -\lambda^2\int_{\mathbb{R}}\int_{\mathbb{R}}\ln|\lambda y|\mu(x-y)\rd{y}\rd{x}
			=  -\lambda^2\int_{\mathbb{R}}\int_{\mathbb{R}}\ln| y|\mu(x-y)\rd{y}\rd{x}=\lambda^2\tilde{\cH}[\mu]
	\end{split}\end{equation}
	where we use the mean-zero property of $\mu$, shown in Lemma \ref{lem_complex}, in the fourth inequality. 
	
	Then we discuss $\tilde{\cD}$. Firstly $\tilde{\cD}[\mu_{\ti}(\frac{\cdot}{\lambda})]=\lambda=\lambda\tilde{D}[\mu_{\ti}]$ by definition. If $\mu = \mu_{\tii, R}$ or $\mu_{\tiii, L(R), R}$, then
	\begin{equation}
		\tilde{\cD}[\mu(\frac{\cdot}{\lambda})]= \int_{[-\lambda,\lambda]}\mu(\frac{x}{\lambda})\rd{x}= \lambda\int_{[-1,1]}\mu(x)\rd{x}=\lambda \tilde{\cD}[\mu].
	\end{equation}

     Therefore $\tilde{\cG}[\mu(\frac{\cdot}{\lambda})] = \lambda^2\tilde{\cH}[\mu]/(\lambda\tilde{D}[\mu])^2 = \tilde{\cG}[\mu] $. 
\end{proof}

It follows from Lemma \ref{lem:scaling} that it suffices to consider $\mu_{\ti}$, $\mu_{\tii}$ and $\mu_{\tiii}$ to determine minimal value of $\tilde{\cG}$. We first compute this value for $\mu_{\ti}$ directly.  
\begin{proposition}\label{prop:type-1-value}
	For admissible $\mu$ of Type $\ti$, we have $\tilde{\cG}[\mu] = \frac{1}{2}$. 
\end{proposition}
\begin{proof}
It suffices to consider $\mu = \mu_{\ti}$ by Lemma \ref{lem:scaling}. Clearly $\tilde{\cD}[\mu]=1$. By Remark \ref{rmk:Wmu-infinity} and \eqref{typei}
	\begin{equation}\label{tildeH1}\begin{split}
			\tilde{\cH}[\mu] = & 2\int_0^{\frac{1}{\pi}}(\tilde{W}*\mu)(x)\rd{x} = -2\int_0^{\frac{1}{\pi}}\int_x^{\frac{1}{\pi}}(\tilde{W}*\mu)'(y)\rd{y}\rd{x}\\
			= &-2\int_0^{\frac{1}{\pi}}y(\tilde{W}*\mu)'(y)\rd{y} 
			=  2\pi\int_0^{\frac{1}{\pi}}\sqrt{-(y^2-\pi^{-2})}\rd{y} = \frac{1}{2}.
	\end{split}\end{equation}
\end{proof}

\subsection{Large $R$}\label{ssec:large-R}
In this section, we consider the value $\tilde{\cG}[\mu]$ for $\mu = \mu_{\tii, R}$ with $R\ge R_c$ and $\mu_{\tiii, L(R), R}$ with $R$ away from $1$. We first compute explicitly $\tilde{\cG}[\mu_{\tii,R}]$ in Proposition \ref{prop:type-2-value}. Then we estimate $\tilde{\cG}[\mu_{\tiii, L(R), R}]$ when $R$ is away from $1$ ( $R_c \ge R\ge 1.1$) using the monotonicity of $\tilde{\cH}$ and $\tilde{\cD}$ proved in Lemma \ref{lem_HDmon}.

\begin{proposition}\label{prop:type-2-value}
	For admissible $\mu$ of Type $\tii$, we have $\tilde{\cG}[\mu] > \frac{1}{2}$. 
\end{proposition}
\begin{proof}
Let $\mu=\mu_{\tii,R}$ for $R\ge R_c$. We have $\tilde{\cD}[\mu]=\int_{[-1,1]}\mu\rd{x}=2m-2 = \pi\sqrt{R^2-1}-2$ where $m = m(R)$ is given in \eqref{typeii-2}. 

Similar with \eqref{tildeH1}, we compute
\begin{equation}\begin{split}
		\tilde{\cH}[\mu] = &  -2 \pv\int_0^R y(\tilde{W}*\mu)'(y)\rd{y} = 2\pi\pv\int_0^R\frac{ y^2\sqrt{R^2-y^2}}{y^2-1}\rd{y}. \\
\end{split}\end{equation}
This integral can be calculated explicitly by reducing to rational integrals. We get
\begin{equation}\begin{split}
		\tilde{\cH}[\mu] = & \frac{\pi^2}{2}(R^2-2), \quad \tilde{\cG}[\mu]=\frac{\tilde{\cH}[\mu]}{\tilde{\cD}[\mu]^2} = \frac{\frac{\pi^2}{2}(R^2-2)}{(\pi\sqrt{R^2-1}-2)^2}.
\end{split}\end{equation}
Then $\tilde{\cG}[\mu]>\frac{1}{2}$ is equivalent to $R>\sqrt{(\frac{1}{\pi}+\frac{\pi}{4})^2+1}\approx1.49$, which holds since $R\ge R_c$.
\end{proof}

Before we study $\tilde{\cG}[\mu_{\tiii}]$, we first give a lemma on the monotonicity of $\tilde{\cH}$ and $\tilde{\cD}$. 
\begin{lemma}\label{lem_HDmon}
For $1<R<R_c$, both functions $\tilde{\cH}[\mu_{\tiii,L(R),R}]$ and $\tilde{\cD}[\mu_{\tiii,L(R),R}]$ are strictly increasing in $R$.
\end{lemma}
\begin{proof}
	Let $\mu=\mu_{\tiii,L(R),R}$. We first consider $\tilde{\cH}$. Similar with \eqref{tildeH1}, we compute
	\begin{equation}\label{tildeH2}\begin{split}
			\tilde{\cH}[\mu] = & 2\int_L^R (\tilde{W}*\mu)(y)\rd{y} = 2\pv\int_L^R (R-x)(\tilde{W}*\mu)'(x)\rd{x}  \\
			= & 2\pi\pv\int_L^R (x-R)\frac{\pi\sqrt{-(x^2-R^2)(x^2-L^2)}}{x^2-1}\rd{x}\\
			= & 2\pi\pv\int_L^R (x-1)\frac{\pi\sqrt{-(x^2-R^2)(x^2-L^2)}}{x^2-1}\rd{x}
			= 2\pi\int_L^R\frac{ \sqrt{(R^2-x^2)(x^2-L^2)}}{x+1}\rd{x} \\
	\end{split}\end{equation}
	where the third equality uses \eqref{typeii-2}, and the second last equality follows from \eqref{LRcond} for $L=L(R)$. Then $\tilde{\cH}[\mu_{\tiii, L(R), R}]$ is strictly increasing in $R$ since $L(R)$ is strictly decreasing in $R$.
	
	By definition of $\tilde{\cD}$ and the mean-zero property of $\mu$, $\tilde{\cD}$ being increasing in $R$ is equivalent to $\int_{(1,\infty)}\mu_{\tiii,L(R),R}\rd{x}$ being decreasing in $R$. For $i=1,2$ and $1< R_i< R_c$, we denote $L_i: = L(R_i)$ and $\mu_i:=\mu_{\tiii,L_i,R_i}$. Suppose the contrary, then there exists $1<R_1<R_2<R_c$ such that
	\begin{equation}\label{LR12}
		\int_{(1,\infty)}\mu_1(x) \rd{x}\le \int_{(1,\infty)}\mu_{2}(x)\rd{x}. 
	\end{equation}
	We now compare the numerator of $\mu_i$ in \eqref{typeiii}, which is
	\begin{equation}\label{xLR}
		(x^2-R_1^2)(x^2-L_1^2)-(x^2-R_2^2)(x^2-L_2^2) = (R_2^2+L_2^2-R_1^2-L_1^2)x^2+(R_1^2L_1^2-R_2^2L_2^2).
	\end{equation}
     By the assumption \eqref{LR12} and the fact that $\mu_1(R_2)> \mu_2(R_2)=0$, we see that $\mu_1-\mu_2$ must be decreasing in $x$ and negative when $x$ is large enough. Therefore there exists some $x_0\in (R_2,\infty)$ such that
	\begin{equation}\label{eqX1}
		\mu_1(x)\ge \mu_2(x) \text{ for } 0\le x< 1 \text{ and } 1<x \le x_0, \quad\quad \mu_1(x) \le \mu_2(x) \text{ for } x > x_0.
	\end{equation}	
	We claim that 
	\begin{equation}\label{eqX}
		\int_{(0,X)} \mu_{1}(x)\rd{x} > \int_{(0,X)}\mu_{2}(x)\rd{x}, \quad \forall X>0.
	\end{equation}
 In fact, this can be seen by separating into the following cases. If $0<X\le 1$, we use the first inequality in \eqref{eqX1}; if $1<X\le x_0$, we firstly notice that \eqref{LR12} implies that $\int_{(0,1]} \mu_{1}(x)\rd{x} > \int_{(0,1]}\mu_{2}(x)\rd{x}$ by the mean-zero property of $\mu$, then we use the first inequality in \eqref{eqX1} for $(1, X)$; if $X>x_0$, then the second inequality in \eqref{eqX1} implies that $\int_{[X,\infty)} \mu_{1}(x)\rd{x} < \int_{[X,\infty)}\mu_{2}(x)\rd{x}$, then we apply the mean-zero property of $\mu$. 
	
	Now we apply Lemma \ref{lem_dec} with $X=\infty$ and get
	\begin{equation}
		(\tilde{W}* \mu_1) (0)= \int_{\mathbb{R}} \tilde{W}(x)\mu_{1}(x)\rd{x} > \int_{\mathbb{R}} \tilde{W}(x)\mu_{2}(x)\rd{x} = (\tilde{W}* \mu_2)(0), 
	\end{equation}
	which contradicts the property $(\tilde{W}*\mu)(0) = 0$ for $\mu$ of Type $\tiii$, proved in Proposition \ref{prop_mu}.
\end{proof}

Now we consider $\mu_{\tiii, L(R), R}$ with $R_c \ge R\ge 1.1$. We take a discretization $R_0<R_1<\cdots<R_n=R_c$ of $R$ with $n = 19$ and denote $\mu_i: = \mu_{\tiii, L(R_i), R_i}$ and $\tilde{\cH}_i:= \tilde{\cH}[\mu_i]$ and $\tilde{\cD}_i:= \tilde{\cD}[\mu_i]$. We then numerically verify in Table \ref{table} that 
\begin{equation}\label{num}
	\frac{\tilde{\cH}_k}{\tilde{\cD}_{k+1}^2}>\frac{1}{2},
\end{equation}
for $0\le k\le n-1$. \footnote{Approximation of integrals are computed in Matlab with error no more than $10^{-4}$. As a result, the error of $\frac{\tilde{\cH}_k}{\tilde{\cD}_{k+1}^2}$ is no more than $10^{-3}$ for any $k$. } \\
	\begin{table}[!htbp]
		\begin{center}
			\label{table}
			\begin{tabular}{|c|c|c|c|c|c|c|c|c|c|}
				\hline
				& & & & & & & & &  \\[-0.3cm]
				$k$ & $R_k$ & $\tilde{\cH}_k$ & $\tilde{\cD}_k$ & $\tilde{\cH}_k/\tilde{\cD}_{k+1}^2$ & $k$ & $R_k$ & $\tilde{\cH}_k$ & $\tilde{\cD}_k$ & $\tilde{\cH}_k/\tilde{\cD}_{k+1}^2$  \\[0.1cm]
				\hline
				0 & 1.1000 & 0.0986 & 0.3188 & 0.5765 & 10 & 1.4297 & 1.7954 & 1.4174 & 0.7500 \\
				\hline
				1 & 1.1292 & 0.1645 & 0.4135 & 0.6290 & 11 & 1.4677 & 2.1224 & 1.5472 & 0.7512 \\
				\hline
				2 & 1.1592 & 0.2495 & 0.5114 & 0.6650 &	12 & 1.5067 & 2.4858 & 1.6809 & 0.7515 \\
				\hline
				3 & 1.1900 & 0.3550 & 0.6125 & 0.6906 &	13 & 1.5467 & 2.8879 & 1.8187 & 0.7512 \\
				\hline
				4 & 1.2216 & 0.4824 & 0.7170 & 0.7090 &	14 & 1.5878 & 3.3312 & 1.9607 & 0.7504 \\
				\hline
				5 & 1.2541 & 0.6331 & 0.8248 & 0.7225 &	15 & 1.6300 & 3.8188 & 2.1070 & 0.7492 \\
				\hline
				6 & 1.2874 & 0.8088 & 0.9361 & 0.7323  &	16 & 1.6733 & 4.3538 & 2.2577 & 0.7477 \\
				\hline
				7 & 1.3216 & 1.0111 & 1.0509 & 0.7394  &	17 & 1.7177 & 4.9404 & 2.4131 & 0.7459 \\
				\hline
				8 & 1.3567 & 1.2417 & 1.1694 & 0.7445 &	18 & 1.7633 & 5.5844 & 2.5736 & 0.7437 \\
				\hline
				9 & 1.3927 & 1.5025 & 1.2915 & 0.7479 &	19 & 1.8102 & 6.3003 & 2.7403 & \\
				\hline
			\end{tabular}
		\end{center}
		\caption{Numerical verification of $\frac{\tilde{\cH}}{\tilde{\cD}^2}$.}
	\end{table}

It then follows from Lemma\ref{lem_HDmon} that for any $R\in (R_k,R_{k+1}]$ and $\mu = \mu_{\tiii, L(R), R}$
\begin{equation}
	\frac{\tilde{\cH}[\mu]}{\tilde{\cD}[\mu]^2} \ge \frac{\tilde{\cH}_k}{\tilde{\cD}_{k+1}^2}>\frac{1}{2}.
\end{equation}
We thus prove the following. 
\begin{proposition}\label{prop:type-3-value-large}
	For $1.1 \le R\le R_c$ and $\mu = \mu_{\tiii, L(R), R}$, we have $\tilde{\cG}[\mu]>1/2$. 
\end{proposition}

\subsection{Small $R$}
In this section, we give an estimate of $\tilde{\cG}[\mu]$ for $\mu = \mu_{\tiii, L(R), R}$ when $R$ is close to $1$. In Section \ref{ssec:large-R} we have shown Theorem \ref{thm:min-value-R} for $R\ge 1.1$, therefore we focus on $1<R\le 1.1$. We start with giving a lower bound for $\tilde{\cH}$ and an upper bound for $\tilde{\cD}$ in terms of $R$ and $L = L(R)$. They serve as good approximations of $\tilde{\cH}$ and $\tilde{\cD}$ when $R$ is close to $1$.

\begin{lemma}\label{lem:estimate-H-D}
For any $ 1< R< R_c$ and $\mu = \mu_{\tiii, L(R), R}$, we have
\begin{equation} 
	\tilde{\cH}[\mu] \ge   \frac{\pi^2}{2(R+1)R}\cdot (\frac{R^2-L^2}{2})^2, \quad \tilde{\cD}[\mu]\le \pi(1-L)\Big(1+\frac{2}{\pi}(1-L)\Big).
\end{equation}
\end{lemma}
\begin{proof}
The lower bound for $\tilde{\cH}$ follows from estimating the integral in \eqref{tildeH2}. By a change of variable $x=\sqrt{y}$, we have
	\begin{equation}\label{tildeH3}\begin{split}
		\tilde{\cH}[\mu] =& \pi\int_{L^2}^{R^2}\frac{\sqrt{(R^2-y)(y-L^2)}}{(\sqrt{y}+1)\sqrt{y}}\rd{y} 
		\ge \frac{\pi}{(R+1)R}\int_{L^2}^{R^2}\sqrt{(R^2-y)(y-L^2)}\rd{y} \\=&  \frac{\pi^2}{2(R+1)R}\cdot (\frac{R^2-L^2}{2})^2.
\end{split}	\end{equation}

For the upper bound of $\tilde{\cD}$, we estimate
	\begin{equation}\label{tildeD1}
		\tilde{\cD}[\mu] = \int_{[-1, 1]} \mu(x) \rd{x} = 2m-2+2\int_0^L\frac{\sqrt{(R^2-x^2)(L^2-x^2)}}{1-x^2}\rd{x},
	\end{equation}
where $m = m(L, R) = \pi \sqrt{(R^2-1)(1-L^2)}/2$ in \eqref{typeiii-2}. Noticing that
	\begin{equation}
		\Big(\frac{\sqrt{(R^2-x^2)(L^2-x^2)}}{1-x^2}\cdot\frac{1-x}{\sqrt{(L-x)(2-L-x)}} \Big)^2= \frac{(R+x)(L+x)}{(1+x)^2}\cdot\frac{R-x}{2-L-x} < 1,
	\end{equation}
	since $L+R<2$ and $LR<1$ by Lemma \ref{lem:Phi}, we get 
	\begin{equation}\label{tildeD2}
		\tilde{\cD}[\mu] < 2m-2+2\int_0^L\frac{\sqrt{(L-x)(2-L-x)}}{1-x}\rd{x}.
	\end{equation}
	
	The last integral in \eqref{tildeD2} can be related to the total mass of Type $\ti$. In fact, by a change of variable $y=\frac{1-x}{\pi(1-L)}$,
	\begin{equation}\label{tildeD3}
		\int_0^L\frac{\sqrt{(L-x)(2-L-x)}}{1-x}\rd{x} = \pi(1-L)\int_{1/\pi}^{1/(\pi(1-L))}\frac{\sqrt{y^2-\pi^{-2}}}{y}\rd{y}
	\end{equation}
	By the even and mean-zero property of $\mu_\ti$, we have
	\begin{equation}
		\int_{1/\pi}^\infty\Big(\frac{\sqrt{y^2-\pi^{-2}}}{y}-1\Big)\rd{y} = -\frac{1}{2}+\frac{1}{\pi}.
	\end{equation}
	Therefore the integral in \eqref{tildeD3} equals
	\begin{equation}\label{tildeD4}
		\begin{split}
			 & \pi(1-L)\left(\frac{1}{\pi(1-L)}-\frac{1}{\pi} -\frac{1}{2}+\frac{1}{\pi} - \int_{1/(\pi(1-L))}^\infty \Big(\frac{\sqrt{y^2-\pi^{-2}}}{y}-1\Big)\rd{y}\right) \\
			= & 1-\frac{\pi(1-L)}{2} + \frac{1}{\pi}(1-L)\int_{1/(\pi(1-L))}^\infty \frac{1}{y(y+\sqrt{y^2-\pi^{-2}})}\rd{y} \\
			< & 1-\frac{\pi(1-L)}{2} + \frac{1}{\pi}(1-L)\int_{1/(\pi(1-L))}^\infty \frac{1}{y^2}\rd{y} 
			\quad < 1-\frac{\pi(1-L)}{2} + (1-L)^2. \\
	\end{split}\end{equation}
Combining \eqref{tildeD2} and \eqref{tildeD4}, we get
\begin{equation}\begin{split}
		\tilde{\cD}[\mu] < & 2m-2+2(1-\frac{\pi(1-L)}{2} + (1-L)^2) \\
		= & \pi(1-L)\Big(\sqrt{\frac{(R+1)(R-1)(1+L)}{1-L}}-1+\frac{2}{\pi}(1-L)\Big) \\
		<  &\pi(1-L)\Big(1+\frac{2}{\pi}(1-L)\Big),\\
\end{split}\end{equation}
where the last inequality holds since $R-1<1-L$ and thus $(R+1)(1+L)<(R+1)(3-R)<4$. 
\end{proof}

\begin{proposition}
		For $1 < R\le 1.1$ and $\mu = \mu_{\tiii, L(R), R}$, we have $\tilde{\cG}[\mu]>1/2$. 
\end{proposition}
\begin{proof}
Using Lemma \ref{lem:estimate-H-D} and $R^2+L^2>2$ and  $R+L<2$ from Lemma \ref{lem:Phi}, we obtain that
\begin{equation}\begin{split}
		\frac{\tilde{\cH}[\mu]}{\tilde{\cD}[\mu]^2} > &  \frac{1}{2(R+1)R}\cdot \frac{(R^2-L^2)^2}{4(1-L)^2}\cdot \frac{1}{(1+\frac{2}{\pi}(1-L))^2} \\
		> &  \frac{1}{2(R+1)R}\cdot \frac{(2-2L^2)^2}{4(1-L)^2}\cdot \frac{1}{(1+\frac{2}{\pi}(1-L))^2} \\
		= &  \frac{(1+L)^2}{2(R+1)R}\cdot \frac{1}{(1+\frac{2}{\pi}(1-L))^2} 
		>  \frac{(1+L)^2}{2(3-L)(2-L)}\cdot \frac{1}{(1+\frac{2}{\pi}(1-L))^2}. \\
\end{split}\end{equation}
We notice that the last expression is equal to 1 when $L=1$, therefore it is larger than $1/2$ when $L$ is close to 1, equivalently when $R$ is close to 1. 
In fact, this can be quantified by observing that the last expression is increasing in $L$. When $L=0.85$, it is approximately $0.58$, therefore it must be greater than $1/2$ for any $0.85\le L < 1$. Since $R\ge \sqrt{2-L^2}$, we have verified $\tilde{\cG}[\mu]>1/2$ for any $1<R\le \sqrt{2-0.85^2}$, and in particular, it is true for $1<R\le 1.1$.
\end{proof}

\section{From $\R$ to $\bT$}\label{sec:R-to-T}
In this section, we will compare the values of $\tilde{\cG}[\mu]$ with $\cG[\rho]$, where $\mu$ is admissible over $\R$ and $\rho$ (without its Dirac masses) is sediment over $\bT$. We have constructed these two family of distributions in Section \ref{ssec:admissible-R} and \ref{ssec:sediment-T} respectively. In Section \ref{sec:min-value-R}, we have determined $\tilde{\cG}[\mu]\ge 1/2$ for all admissible $\mu$. In order to establish a comparison, we first introduce a way to associate a signed measure $\rho_{\circ}\in \cM$ with each admissible distribution $\mu$ over $\R$, via periodization, in Section \ref{ssec:periodization}. We then relate the quantities $\tilde{\cH}[\mu]$ and $\tilde{\cD}[\mu]$ with $\cH[\rho_{\circ}]$ and $\cD[\rho_{\circ}]$ in Theorem \ref{prop:rho-circ-H-D} and Section \ref{ssec:compare1}. In particular, we show that $\cG[\rho_{\circ}]\ge \tilde{\cG}[\mu]$. Finally we prove a comparison principle using the convexity of logarithmic potential, showing that for each minimizer $\rho \in \cM$ of $\cG$ in $\cM_{\cD\ge d}$ there exists a unique admissible $\mu$ whose associated $\rho_{\circ}$ satisfying $\cG[\rho]\ge \cG[\rho_{\circ}]>1/2$. 

\subsection{Periodization}\label{ssec:periodization}
In this section, we bridge the admissible distributions over $\R$ to distributions over $\bT$ via a periodization. For each signed measure $\mu$ with $\int_{\R} |\mu(x)| \rd{x} < \infty$ over $\R$, we define its \emph{periodization}
\begin{equation}\label{rhocirc0}
	\mu_{\bT}(x): = \sum_{j\in\mathbb{Z}} \mu(x-j),  \quad  x\in \bT. 
\end{equation}
This definition does not depend on the choice of representatives for $x$ in $\bT$. Now for each admissible $\mu$, we associate a signed measure over $\bT$
\begin{equation}\label{rhocirc}
	\rho_\circ[\mu](x): = 1+\mu_{\bT}(x) ,\quad x\in\mathbb{T}.
\end{equation}
If $\mu$ is in \eqref{typei}, \eqref{typeii} and \eqref{typeiii} up to a scaling factor $\lambda$ , then $\mu= -1+ \mu_c+ \mu_d$ using the notation in Lemma \ref{lem_complex}, we can break down $\rho_{\circ}[\mu] = \rho_{\circ, c}[\mu] + \rho_{\circ, d}[\mu]$ where
\begin{equation}
	\rho_{\circ, c}[\mu](x):= 1+ (-1+\mu_c)_{\bT}(x),\quad\quad  \rho_{\circ, d}[\mu](x):= (\mu_d)_{\bT}(x).
\end{equation}
For simplicity, we will drop the associated $\mu$ in these notations when there is no confusion. 

We now characterize properties of $\rho_{\circ, c}$ and $\rho_{\circ, d}$ for $\mu$. 

\begin{proposition}\label{prop:rho-circ-c-d}
	Given $\mu$ an admissible distribution over $\R$ and $\rho_{\circ} = \rho_{\circ}[\mu]$. Let $m = m(R, L(R))$ be given by \eqref{typeiii-2} for $\mu$ of Type $\tii$ and $\tiii$ and $\lambda$ the scaling factor of $\mu$. Assume $\lambda\le 1$ for $\mu$ of type $\ti$ and $\lambda \le \frac{1}{2}$, $\lambda m<\frac{1}{2}$ for $\mu$ of type $\tii$ or $\tiii$. \\
Then $\rho_\circ $ is even with $\int_{\mathbb{T}}\rho_\circ\rd{x}=1$, and
\begin{itemize}
	\item[\textnormal{(i)}] $\rho_{\circ, d}$ is $\lambda \delta $ for $\mu$ of Type $\ti$ and is $\lambda m(\delta_{\lambda} + \delta_{-\lambda})$ for $\mu$ of Type $\tii$ or $\tiii$.
	\item[\textnormal{(ii)}]	$\rho_{\circ, c}$ is H\"older continuous with H\"older exponent $1/2$. 
	\item[\textnormal{(iii)}]$\rho_{\circ,c}$ is $C^{2}$ and $\rho_{\circ, c}''< 0$ on $\bT \backslash B$, where $B = \{ \lambda \pi^{-1} \}$ for $\mu$ of Type $\ti$ and $B = \{ \pm \lambda L, \pm \lambda R \}$ for $\mu$ of Type $\tii$ and $\tiii$. 
\end{itemize}
If $\mu$ is of Type $\tii$ or $\tiii$ then 
\begin{itemize}
   \item[\textnormal{(iv)}] $\{x\in(-\lambda,\lambda):\rho_\circ(x)\ge 0\}$ is either $\emptyset$ or $[-L_\circ, L_\circ]$ for some $L_{\circ}\in [0, \lambda L )$. 
\item[\textnormal{(v)}] $\{x\in(\lambda,1-\lambda):\rho_\circ(x)\ge 0\}$ is either $\emptyset$, or $[R_\circ,1-R_\circ]$ for some $R_{\circ}\in (\lambda R, 1/2]$. 
\end{itemize}
\end{proposition}
\begin{proof}
 Since $\mu$ is even and mean zero over $\R$, it is clear that $\rho_{\circ}$ is even and has $\int_{\bT} \rho_{\circ}[\mu] (x) \rd{x} = 1$. It follows from the expression of $\mu_d$ and that $\rho_{\circ, d}$ has the stated form. Combining the expression of $\mu_c$ and the fact that $|\mu_c'(x)|\lesssim |x|^{-3}$ for large $|x|$, one can see that $\rho_{\circ, c}$ is H\"older continuous with exponent $1/2$. For $x\notin B$ (when considered as a subset of $\R$), it is clear that $-1+\mu_c(x)$ is $C^{2}$ and $|\mu''_{c}(x)|\lesssim |x|^{-4}$ for large $|x|$ and $\mu_c''(x)\le 0$, therefore we $\rho_{\circ, c}$ is $C^{2}$ and 
\begin{equation}\label{ddmuneg}
	\rho_{\circ, c}''=\sum_{j\in\mathbb{Z}}\mu_c''(x-j) < 0, \text{ for } x \in \bT\backslash B. 
\end{equation}
	
For $\mu$ of Type $\tii$ and $\tiii$, in $(-\lambda, \lambda)$ and $(\lambda, 1-\lambda)$ by the previous discussion, we have $\rho_{\circ} = \rho_{\circ, c}$. By Proposition \ref{prop_mu}, $\sum_{j\in\mathbb{Z},j\ne 0}\mu(x-j)<0$, therefore 
\begin{equation}
	\supp(\rho_{\circ, c})_+ \subseteq \supp(1+\mu(x)\cdot \chi_{[-1/2,1/2)})_+.
\end{equation}
We separate the discussion depending on the size of $\lambda R$. If $\lambda R\le 1/2$, then  
\begin{equation}
	 \supp(1+\mu(x)\chi_{[-1/2,1/2)})_+= [-\lambda L,\lambda L]\cup [\lambda R,1-\lambda R],
\end{equation}
where $[-\lambda L, \lambda L]$ is replaced with $\emptyset$ if $L = 0$. Then it follows from $\rho_{\circ, c}''<0$ that $\supp(\rho_{\circ, c})_+$ consists at most two possibly empty symmetric intervals, therefore must be $[-L_\circ, L_\circ]$ with $L_{\circ}\in [0, \lambda L)$ and $[R_\circ,1-R_\circ]$ with $R_{\circ}\in (\lambda R, 1/2] $. If $\lambda R>1/2$, then
\begin{equation}
	\supp(1+\mu(x)\chi_{[-1/2,1/2)})_+= [-\lambda L,\lambda L],
\end{equation}
and it is replaced with $\emptyset$ if $L = 0$. Notice that $R+L<2$ by Lemma \ref{lem:Phi}, therefore $|\lambda R- (-\lambda L)| <1$ when $\lambda <1/2$, thus $\lambda R$ as a point in $\bT$ does not lie in $[-\lambda L, \lambda L]$. Therefore we again have $\rho_{\circ, c}''<0$ on $[-\lambda L, \lambda L]$, and the statement follows. 
\end{proof}

Next we establish the key connection between the functionals $\tilde{\cH}$ and $\tilde{\cD}$ over $\R$ and the functionals $\cH$ and $\cD$ over $\bT$. Note that $\rho_{\circ}$ is not necessarily a measure but only a signed measure over $\bT$. Therefore we extend the functionals $\cH$ by defining
\begin{equation}
	\cH[\rho] = -\ess \inf (W*\rho),
\end{equation}
for arbitrary signed measures $\rho$ over $\bT$. 
\begin{theorem}[The First Comparison]\label{prop:rho-circ-H-D}
	Given $\mu$ an admissible distribution over $\R$ and $\rho_{\circ} = \rho_{\circ}[\mu]$. Let $m = m(R, L(R))$ be given by \eqref{typeiii-2} for $\mu$ of Type $\tii$ and $\tiii$ and $\lambda$ the scaling factor of $\mu$. Assume $\lambda\le 1$ for $\mu$ of type $\ti$ and $\lambda \le \frac{1}{2}$, $\lambda m<\frac{1}{2}$ for $\mu$ of type $\tii$ or $\tiii$. Then we have
\begin{equation}\label{prop_rhocirc_2}
	\tilde{\cH}[\mu] =- \inf_{x\in \bT} (W*\rho_{\circ}) =  \cH[\rho_{\circ}].
\end{equation}
Moreover, $(W*\rho_{\circ}) (x) = \cH[\rho_{\circ}]$ for $ x\in \supp (\rho_{\circ, c})_+$. On the other hand,
\begin{equation}
	\tilde{\cD}[\mu] > \int_{[-\lambda , \lambda]} ((\rho_{\circ})_+ -1 ) \rd{x}. 
\end{equation}
\end{theorem}
\begin{proof}
Firstly we observe
	\begin{equation}
		\int_{[-\lambda,\lambda]}((\rho_\circ)_+-1)\rd{x} <   \int_{[-\lambda,\lambda]} (\max\{1+\mu(x),0\}-1)\rd{x}=\int_{[-\lambda,\lambda]} \mu\rd{x}= \tilde{\cD}[\mu].
	\end{equation}
	
We then focus on the functional $\cH$. By definition of $\rho_{\circ}$, we obtain 
	\begin{equation}
		\hat{\rho}_\circ(k) = \left\{\begin{split}
			& 1,\quad k=0 \\
			& \cF[\mu](k),\quad k\in\mathbb{Z},\,k\ne 0
		\end{split}\right.
	\end{equation}
using the mean-zero property $\hat{\mu}(0)=0$. Here $\hat{\rho}_{\circ}$ is the Fourier coefficients over $\bT$ and $\cF[\mu]$ is the Fourier transform over $\R$. Notice that
\begin{equation}
	\cF[\tilde{W}](\xi) = \frac{1}{|\xi|},\,\xi\ne 0, \quad \hat{W}(k) = \frac{1}{|k|},\,k\ne 0,
\end{equation}
and $\hat{W}(0) = 0$, therefore
\begin{equation}
	\cF[\tilde{W}*\mu](\xi) = \frac{\cF[\mu](\xi)}{|\xi|} \text{ for }\xi\ne 0,  \quad \cF[\tilde{W}*\mu](0)=\int_{\mathbb{R}}(\tilde{W}*\mu)\rd{x}=\tilde{\cH}[\mu].
\end{equation}
We denote $\Lambda(x):=\sum_{j\in\mathbb{Z}}\delta(x-j)$ for $x\in \R$. By Poisson summation formula, we have $\cF[\Lambda](\xi) = \Lambda(\xi)$. Therefore 
\begin{equation}
\cF[\tilde{W}*\mu*\Lambda](\xi) =  \cF[\tilde{W}*\mu]\cdot \Lambda(\xi) = \sum_{0\ne j \in \mathbb{Z}} \frac{\cF[\mu](j)}{|j|} \delta(\xi-j) + \tilde{\cH}[\mu] \delta(\xi).
\end{equation}
Now we compare with $W*\rho_\circ$ where
\begin{equation}
	\widehat{W*\rho_{\circ}}(k) =\frac{\hat{\mu}(k)}{|k|}  \text{ for }k\ne 0,  \quad 	\widehat{W*\rho_{\circ}}(0)=0,
\end{equation}
we obtain by taking Fourier inversion
\begin{equation}\label{prop_rhocirc_1}
	(W*\rho_\circ )(x) + \tilde{\cH}[\mu]= ( \tilde{W}*\mu*\Lambda) (x) = (\tilde{W}*\mu)_{\bT}(x).
\end{equation}

Recall in Proposition \ref{prop_mu} that $\mu - \mu_d <0$ and $\tilde{W}*\mu \ge 0$. Therefore in order to prove \eqref{prop_rhocirc_2}, it suffices to show that $\supp (\tilde{W}*\mu)_{\bT} \neq \bT$. By Remark \ref{rmk:Wmu-infinity} $\tilde{W}*\mu$ obtains its minimal value $0$ in $\supp \mu_c$, and therefore $\supp (\tilde{W}*\mu) \subset \{x : \mu=-1\}$. It follows from the assumption on $\lambda$ that in all cases $\int_{\R} \mu_d <1$. By the mean zero property of $\mu$, we see that $|\supp(\tilde{W}*\mu)|< \int_{\R} \mu_d <1$. Therefore we show \eqref{prop_rhocirc_2}. 

If $\lambda R<1/2$, then $(\tilde{W}*\mu)_{\bT}$ is supported on $\lambda L<|x|<\lambda R$, therefore $W*\rho_{\circ}$ obtains its minimum exactly on $[-\lambda L, \lambda L]\cup [\lambda R, 1-\lambda R]$ which contains $\supp (\rho_{\circ, c})_+$. If $\lambda R\ge 1/2$, then similarly with Theorem \ref{prop:rho-circ-c-d}, we know that $\lambda R$ does not lie in $[-\lambda L, \lambda L]$. Therefore $(\tilde{W}*\mu)_{\bT}$ is zero on $[-\lambda L, \lambda L]$ which contains $\supp(\rho_{\circ, c})_+$. Therefore we show that $W*\rho_{\circ}$ obtains minimal value $\cH[\rho_{\circ}]$ on $\supp (\rho_{\circ, c})_+$. 
\end{proof}

\subsection{A Comparison Principle via Convexity}\label{ssec:compare1}
In this section, our main goal is to establish the following comparison principle between energy minimizers associated with different external potentials. The proof of this principle essentially takes advantage of the convexity of log potential. 
\begin{proposition}\label{lem_comp1}
Given $0<M<\frac{1}{2}$, $m>0$, and two external potentials over $\bT$
	\begin{equation}\label{Uflat0}
		U_\flat =W * m(\delta_M + \delta_{-M}), \quad U_\sharp =U_\flat+U_*,
	\end{equation}
where $U_*\in C^1(\bT)$ is even. Let $m_1,m_2\ge 0$, and $\rho_\flat$ and $\rho_\sharp$ be the energy minimizer for $U_\flat$ and $U_\sharp$ in $\cM_{m_1,m_2}$ (c.f. \eqref{Mm1m2}) respectively. If 
	\begin{equation}\label{lem_comp1_1}
		U_*'(x)<0,\quad \forall x\in (0,\frac{1}{2})\cap \supp\rho_\sharp,
	\end{equation}
and $\rho_\flat$ and $\rho_\sharp$ are H\"older continuous functions, then 
	\begin{equation}
		( W*\rho_\flat)(0)>( W*\rho_\sharp)(0),\quad ( W*\rho_\flat)(\frac{1}{2})<( W*\rho_\sharp)(\frac{1}{2}).
	\end{equation}
\end{proposition}

We first give some preparations. 
\begin{lemma}\label{lem_holder}
	Let $u$ be H\"older continuous on $\mathbb{T}$ with exponent $\beta$. Then its Hilbert transform $H[u]$ is defined everywhere over $\bT$ and is H\"older continuous with exponent $\beta_1$ for any $0<\beta_1<\beta$.
\end{lemma}
\begin{proof}
	Recall that the kernel for Hilbert transform over $\bT$ is exactly $W'$, and it is an odd function. Therefore
	\begin{equation}
		H[u](x) = \pv\int_{\mathbb{T}}W'(y)u(x-y)\rd{y} = \int_{\mathbb{T}}W'(y)(u(x-y)-u(x))\rd{y}.
	\end{equation}
	Here the last integrand is integrable because $|W'(y)|\lesssim |y|^{-1}$ near $y=0$ and $|u(x-y)-u(x)|\le C |y|^\beta$. Therefore $H[u](x)$ is defined everywhere. 
	
	To show the H\"older continuity, we take $x_1<x_2$ and denote $\epsilon=x_2-x_1>0$. Then $|H[u](x_2)-H[u](x_1)|$ is bounded since
		\begin{equation}
		\int_{|y|<\epsilon} (|W'(y)(u(x_2-y)-u(x_2))|+|W'(y)(u(x_1-y)-u(x_1))|)\rd{y} \le C \epsilon^{\beta},
		\end{equation}
	\begin{equation}
		 \int_{|y|>\epsilon} W'(y)\big((u(x_2-y)-u(x_2))-(u(x_1-y)-u(x_1))\big)\rd{y} 
			\le C\epsilon^\beta\int_{|y|\ge \epsilon}|W'(y)|\rd{y} 
			\le C\epsilon^{\beta_1}
\end{equation}
	for any $\beta_1<\beta$. This finishes the proof.
\end{proof}

Let $m_0> 0$, and $\rho\in\cM_{m_0}$ be even. Define its cumulative function as
\begin{equation}
	\km[\rho](x) = \int_{[0, x]} \rho(y)\rd{y},\quad x\in [0,\frac{1}{2}].
\end{equation}
It is clear that $\km[\rho](0)=0$ and $\km[\rho](\frac{1}{2})=\frac{m_0}{2}$. When $\rho$ is H\"older continuous, we also define the inverse function $\kX[\rho](m)$ for $m\in [0, m_0/2]$ so that $\km(\kX[\rho](m)) = m$. It is clear that $\kX[\rho]$ is strictly increasing and piece-wise continuous. We will write $\km(x)$ and $\kX(m)$ in short when there is no confusion. 
\begin{lemma}\label{lem_X}
	Let $m_0>0$ and $\rho \in \cM_{m_0}$ be even and H\"older continuous. Let $
	\km(x)$ be the cumulative function for $\rho$ and $\kX(m)$ be its inverse function. Then we have the change-of-variable formula
	\begin{equation}\label{lem_X_1}
		(W*\rho)(x) = \int_0^{m_0/2}(W(x-\kX(m))+W(x+\kX(m)))\rd{m}
	\end{equation}
	and
	\begin{equation}\label{lem_X_2}
		(W'*\rho)(x) = \pv\int_0^{m_0/2}(W'(x-\kX(m))+W'(x+\kX(m)))\rd{m}
	\end{equation}
	for $x\in [0,1/2]$.
\end{lemma}
\begin{proof}
Since $\rho$ is H\"older continuous, both $W'*\rho$ and $W*\rho$ are well-defined everywhere by Lemma \ref{lem_holder}, and $\kX(m)$ is guaranteed to be strictly increasing and piece-wise continuous. By symmetry of $\rho$ and a substitution $y = \kX(m)$, we get
\begin{equation}
	\begin{split}
		(W*\rho)(x) & = \int_{0}^{1/2} \big(W(x-y) + W(x+y) \big) \rho(y) \rd{dy} \\
		& =  \int_0^{m_0/2}\big(W(x-\kX(m))+W(x+\kX(m))\big)\rd{m}.
	\end{split}
\end{equation}

We can consider $W'*\rho$ similarly, and it suffices to prove 
\begin{equation}\label{lem_X_21}
	\pv \int_0^{1/2} W'(x-y)\rho(y)\rd{y} = \pv\int_0^{m_0/2}W'(x-\kX(m))\rd{m}.
\end{equation}
Without loss of generality, we can assume $x\in \supp\rho$ and $\rho(x)>0$, since otherwise $\pv$ can be removed. Denote $M_0 = \km (x)$. Then by previous analysis, we obtain that 
\begin{equation}\label{lem_X_22}
	\int_{[0,m_0/2]\backslash (M_0-\epsilon,M_0+\epsilon)} W'(x-\kX(m))\rd{m} = \int_{[0,\kX(M_0-\epsilon)]\cup[\kX(M_0+\epsilon),1/2]} W'(x-y)\rho(y)\rd{y},
\end{equation}
therefore it suffices to show that
\begin{equation}\label{lem_X_23}
 \int_{[x + (x-\kX(M_0-\epsilon)), \kX(M_0+\epsilon)]} W'(x-y)\rho(y)\rd{y} \to 0, \text{ as } \epsilon \to 0.
\end{equation}
By the H\"older continuity of $\rho$, we get $\rho(x\pm \epsilon) = \rho(x)+ O(\epsilon^{\alpha})$. Therefore when $\epsilon$ is small enough, $|\kX(M_0 \pm \epsilon)- \kX(M_0)|\lesssim \epsilon $. The above limit then easily follows. 
\end{proof}

Now we are ready to give the following comparison principle, which is the key to proving Proposition \ref{lem_comp1}. We will compare two measures $\rho_{\flat}$ and $\rho_{\sharp}$. For the notation, we will write $\km_{\flat}$ for $\km[\rho_{\flat}]$ and $\kX_{\flat}$ for $\kX[\rho_{\flat}]$. Similarly for $\rho_{\sharp}$.
\begin{lemma}\label{lem_displace}
	Let $\rho_\flat$ and $\rho_\sharp$ be in $\cM_{m_0}$ that are H\"older continuous and even. If
	\begin{equation}\label{lem_displace_1}
		\km_\flat(x_0) < \km_\sharp(x_0)
	\end{equation}
	for some $x_0\in (0,1/2)$, then there exists $0< x_\sharp < x_\flat \le 1/2$ and $x_\flat\in\supp\rho_\flat$ and $x_\sharp\in\supp\rho_\sharp$ such that
	\begin{equation}\label{lem_displace_2}
\quad \km_\flat(x_\flat)=\km_\sharp(x_\sharp),\quad ( W'*\rho_\flat)(x_\flat) \ge ( W'*\rho_\sharp)(x_\sharp).
	\end{equation}
\end{lemma}
\begin{proof}
For convenience, we will define $\kX_{\sharp}(m)$ to be the smallest $x \in [0,1/2]$ such that $\km_{\sharp}(x) = m$, and $\kX_{\flat}(m)$ to be the largest $x\in [0,1/2]$ such that $\km_{\flat}(x)= m$. Then $\kX_{\sharp}$ is lower-semicontinuous and $\kX_{\flat}$ is upper-semicontinuous. By the assumption \eqref{lem_displace_1}, we see that 
	\begin{equation}
		\sup_{m\in [0,m_0/2]} (\kX_\flat(m)-\kX_\sharp(m)) > 0.
	\end{equation}
	This supremum can be achieved, say at $m_s \in [0, m_0/2]$, since $\kX_\flat-\kX_\sharp$ is upper-semicontinuous, and it is also the maximum of the difference. We then denote $x_\sharp=\kX_\sharp(m_s)$ and $x_\flat=\kX_\flat(m_s)$. They clearly satisfy $x_{\sharp} \in\supp\rho_\sharp$ and $x_{\flat} \in\supp\rho_\flat$, and $x_\sharp<x_\flat$, and $\quad m_\flat(x_\flat)=m_s=m_\sharp(x_\sharp)$. So it suffices to prove
	$( W'*\rho_\flat)(x_\flat) - ( W'*\rho_\sharp)(x_\sharp) \ge 0$.
	
	By Lemma \ref{lem_X} we can write $( W'*\rho_\flat)(x_{\flat})-( W'*\rho_\sharp)(x_{\sharp})$ as
	\begin{equation}\label{W2term}
 \pv\int_0^{m_0/2}\big( W'(x_{\flat}-\kX_\flat(m))+ W'(x_{\flat}+\kX_\flat(m))\big)  -\big( W'(x_{\sharp}-\kX_\sharp(m))+ W'(x_{\sharp}+\kX_\sharp(m))\big)\rd{m},
\end{equation}
where the integrand is
\begin{equation}\label{W4term}
\int_{x_{\sharp}-\kX_\sharp(m)}^{ x_{\flat}-\kX_\flat(m)} W''(z)\rd{z} + \int_{x_{\sharp}+\kX_\sharp(m)}^{x_{\flat}+\kX_\flat(m)} W''(z)\rd{z},
\end{equation}
for any $m\neq m_s$. 

It follows from the definition of $m_s$ that $x_{\sharp}-\kX_\sharp(m) \le x_{\flat}-\kX_\flat(m)$. And one can check that in both integrals the domain does not contain $z = 0$. Therefore if $x_{\sharp}+\kX_\sharp(m)\le x_{\flat}+\kX_\flat(m)$ then the conclusion follows from the positivity of $W''$. Suppose not, we must have $\kX_{\flat}(m) - \kX_{\sharp}(m)< - (x_{\flat} - x_{\sharp})<0$, then we have
\begin{equation}
	\eqref{W4term}= \Big (\int_{x_{\sharp}-\kX_\sharp(m)}^{x_\flat-\kX_{\sharp}(m)}+\int_{x_\flat-\kX_\sharp(m)}^{x_\flat-\kX_\flat(m)} + \int_{x_\sharp+\kX_\sharp(m)}^{x_\flat+\kX_\sharp(m)}+\int_{x_\flat+\kX_\sharp(m)}^{x_\flat+\kX_\flat(m)} \Big) W''(z) \rd{z}.
\end{equation}
We can again show that $z=0$ is not in the domain of any of these integrals: trivially true when $m<m_s$ and use $\kX_{\flat}(m)< \kX_{\sharp}(m)$ for $m>m_s$. Now it is clear that the first and third integral are positive, and the second and fourth integral can combine as
	\begin{equation}\label{W4term1}
\int_{x_\flat-\kX_\sharp(m)}^{x_\flat-\kX_\flat(m)} W''(z)\rd{z} +\int_{x_\flat+\kX_\sharp(m)}^{x_\flat+\kX_\flat(m)}  W''(z)\rd{z}  = \int_{\kX_\flat(m)}^{\kX_\sharp(m)} \big( W''(x_\flat-z)- W''(x_\flat+z)\big)\rd{z}
\end{equation}
In the last integral, we have $z$ and $x_{\flat}$ both in $[0,1/2]$. Then it follows from the fact that $W''(x) = \pi/\sin^2 \pi x$ is even and decreasing in $(0,1/2]$ that the integrand is always positive. We thus finish proving $( W'*\rho_\flat)(x_\flat) \ge ( W'*\rho_\sharp)(x_\sharp)$. 
\end{proof}

Now we are ready to prove the main proposition in this section. 
\begin{proof}[Proof of Proposition \ref{lem_comp1}]
	Recall the notation that for external potential $U_{\flat}$, we obtain $\rho_{\flat} \in \cM_{m_0}$ as the energy minimizer as proved in Proposition \ref{prop:energy-min}, and the generated potential is $V_{\flat}: = U_{\flat} + W*\rho_{\flat}$. By Proposition \ref{prop:energy-min}, we have $V_{\flat}'(x) = 0$ on $\supp (\rho_{\flat})$. We denote the cumulative function $\km_{\flat}$ for $\rho_{\flat}$. It follows from the assumption that $\km_{\flat}(M) = m_1/2$. Similarly everything holds also for $U_{\sharp}$. 
	
	We claim that
	\begin{equation}\label{claimX}
		\km_\flat(x)\ge \km_\sharp(x),\quad 0\le x \le 1/2.
	\end{equation}
	Suppose not, then we apply Lemma \ref{lem_displace} to get $x_\flat$ and $x_\sharp$ satisfying \eqref{lem_displace_2}. Since $V_{\flat}' = W'*\rho_{\flat} + U_{\flat}'= 0$ on $\supp (\rho_{\flat})$ and similarly for $\rho_{\sharp}$, \eqref{lem_displace_2} then implies
	\begin{equation}
		U_{\sharp}'(x_{\sharp}) \ge U_{\flat}'(x_{\flat}). 
	\end{equation}
On the other hand,
	\begin{equation}\label{dUcomp}
		U_\flat'(x_\flat)-U_\sharp'(x_\sharp)=\big(U_\flat'(x_\flat)-U_\flat'(x_\sharp)\big)+ \big(U_\flat'(x_\sharp)-U_\sharp'(x_\sharp)\big) = \int_{x_{\sharp}}^{x_{\flat}} U_{\flat}''(x) \rd{x} - U_{*}'(x_{\sharp})>0. 
	\end{equation}
Here the first term is positive since $U_{\flat}''>0$ by the expression of $U_{\flat}$ and $M \notin [x_{\sharp}, x_{\flat}]$ since $\km_\flat(M)=\km_\sharp(M)=m_1/2$. The second term is positive since $-U_{*}'(x_{\sharp})>0$ by the assumption. Therefore we get a contradiction and prove the claim \eqref{claimX}. 

Finally, applying Lemma \ref{lem_dec} with $X=\frac{1}{2}$, we get $( W*\rho_\flat)(0)>( W*\rho_\sharp)(0)$, and the other conclusion $( W*\rho_\flat)(\frac{1}{2})<( W*\rho_\sharp)(\frac{1}{2})$ can be obtained similarly.
\end{proof}

\subsection{Comparison between Minimizers}\label{ssec:comparison2}
In this section, our main goal is to compare the value of $\cG$ for $\rho \in \cM$ over $\bT$ with the value of $\tilde{\cG}$ for admissible distributions over $\R$, so that we prove our main result $\cG\ge 1/2$ in Theorem \ref{thm:main}.

Recall that in Theorem \ref{thm:main-characterization} and Proposition \ref{prop:energy-min}, we have shown that the unique minimizer of $\cG$ for $\rho \in \cM_{\cD\ge d}$ must be the energy minimizer of $\cE$ (together with the Dirac masses $m(\delta_M +\delta_{-M})$) where the external potential is in the format of $U = W*m(\delta_M +\delta_{-M})$ form some $m \in (0, 1/2]$ and $M \in [0, 1/2]$. We will compare $\rho$ with $\rho_{\circ}$ constructed via periodization. By description of $\rho_{\circ, d}$ in Theorem \ref{prop:rho-circ-c-d}, we can find a unique $\rho_{\circ}$ with $\rho_{\circ, d} = m(\delta_M +\delta_{-M})$: if $M = 0$, then we have $\mu$ of Type $\ti$ and $\lambda = 2m$; if $M \neq 0$, then we have $\mu$ of Type $\tii$ or $\tiii$ with $\lambda = M$, and $R$ is uniquely determined by $m(R, L(R)) = m/\lambda$, since $m(R, L(R))$ is increasing in $R$ as $L(R)$ is decreasing in $R$ by Proposition \ref{prop_mu}. 
Our main theorem for this section is the following comparison theorem. It follows directly from the following theorem, together with Theorem \ref{thm:min-value-R}, that $\cG[\rho]\ge 1/2$ for $\rho \in \cM$. 

\begin{theorem}[The Second Comparison]\label{thm_Dcomp}
	Let $0<d< 1 $. Let $\rho =m (\delta_M + \delta_{-M})+ \rho_1$ be an even minimizer of $\cG$ in $\cM_{\cD \ge d}$ given in Theorem \ref{thm:main-characterization}. Let $\mu$ be the unique admissible distribution with the associated $\rho_{\circ, d} = m (\delta_M + \delta_{-M})$. Assume $M>0$, then we have $\lambda=M$ and
	\begin{equation}\label{thm_Dcomp_0}
		\cH[\rho] \ge  \tilde{\cH}[\mu],  \quad \quad	\cD[\rho] \le  \tilde{\cD}[\mu].
	\end{equation}
\end{theorem}
\begin{proof}
It is clear that $\lambda = M$ as explained above. It suffices to prove the inequality in \eqref{thm_Dcomp_0}. Let $\rho_{\circ}$ be associated with $\mu$ as in \eqref{rhocirc}.  

We firstly study $\cH$. By Theorem \ref{prop:rho-circ-H-D}, we have $\cH[\rho_{\circ}] = \tilde{\cH}[\mu]$, therefore it suffices to compare $\cH[\rho_{\circ}]$ and $\cH[\rho]$. By Proposition \ref{prop:rho-circ-c-d}, we can write 
\begin{equation}
	\rho_\circ = m (\delta_M+\delta_{-M}) + \rho_{\circ,J} + \rho_{\circ ,K} + \rho_{\circ, -}
\end{equation}
where $\rho_{\circ ,J}$ and $\rho_{\circ ,K}$ are nonnegative (possibly identically zero) and supported on $J_\circ=[R_\circ,1-R_\circ]$ and $K_\circ=[-L_\circ,L_\circ]$ respectively, and $\rho_{\circ ,-}$ is nonpositive, supported on $(J_\circ \cup K_\circ )^c$. It follows from Proposition \ref{prop:energy-min} and Theorem \ref{prop:rho-circ-H-D} that $\rho_{\circ,J} + \rho_{\circ ,K}$ is the unique minimizer of $\cE_U$ in $\cM_{m_0}$ where
\begin{equation}
	m_0 = 1-2m - \int_{\bT}\rho_{\circ, -}(x) \rd{x}, \quad\quad U =  W * \big(m(\delta_M+\delta_{-M})  + \rho_{\circ ,-}\big),
\end{equation}
since $V_U[\rho_{\circ,J} + \rho_{\circ ,K}]= U + W*(\rho_{\circ,J} + \rho_{\circ ,K}) = W*\rho_{\circ}$ takes minimal value on $\supp(\rho_{\circ,J} + \rho_{\circ,K}) = \supp (\rho_{\circ, c})_+$. It then follows from Proposition \ref{prop:energy-min} that $\rho_{\circ,J} + \rho_{\circ ,K}$ is also the unique maximizer of $\ess \inf V_U$, therefore 
\begin{equation}
	-\cH[\rho] = \ess\inf (W*\rho) = \ess\inf V_U[\rho_1 - \rho_{\circ, -}] \le  \ess\inf ( W*\rho_\circ ) = -\cH[\rho_{\circ}]= - \tilde{\cH}[\mu].
\end{equation}

Next we study $\cD$. By Theorem \ref{prop:rho-circ-H-D}, it suffices to prove
\begin{equation}\label{Drhocomp}
	\cD[\rho] \le \int_{[-M,M]}((\rho_\circ)_+-1)\rd{x}.
\end{equation}
	If $\int_{\mathbb{T}} (\rho_{\circ ,J}+\rho_{\circ ,-})\rd{x} \le 0$, then 
\begin{equation}
	\begin{split}
&\int_{[-M,M]}((\rho_\circ )_+-1)\rd{x} \ge \int_{\mathbb{T}}(\rho_\circ -1)\rd{x} + (1-2M)\\
= &1-2M  = \int_{\bT} \rho  \rd{x} - \int_{[-M, M]}1 \rd{x} \ge  \int_{[-M, M]} (\rho -1) \ge \cD[\rho]. \\
	\end{split}
\end{equation}
Here the second inequality we use the above assumption $\int_{\mathbb{T}} (\rho_{\circ ,J}+\rho_{\circ ,-})\rd{x} \le 0$. Therefore we assume 	$\int_{\mathbb{T}} (\rho_{\circ ,J}+\rho_{\circ ,-})\rd{x}>0$ for the rest of the proof. This in particular implies that $J\neq \emptyset$ and $M R< 1/2$ by Proposition \ref{prop:rho-circ-c-d}. Let $x_J\in J_\circ$ be the unique number in $(0,1/2)$ such that
\begin{equation}\label{Jsharp1}
	\int_{R_{\circ}\le |x|\le x_J}\rho_{\circ ,J}\rd{x} + \int_{\mathbb{T}} \rho_{\circ ,-}\rd{x} = 0.
\end{equation}

We will formulate energy minimization problems for $\rho_\circ $, by specifying external potentials $U_{\sharp}$ and $U_{\flat}$ and their corresponding minimizer $\rho_{\sharp}$ and $\rho_{\flat}$. For small $\epsilon>0$ and $r\ge 0$, define a cutoff function
\begin{equation}
	\phi_{r,\epsilon}(x) = \max\left(0,1-\frac{1}{\epsilon}\dist\big(x,\{y\in [-1/2,1/2):|y|>x_J+r\}\big)\right),
\end{equation}
where $\dist$ denotes the distance function on $\bT$. Combined with \eqref{Jsharp1}, for each $\epsilon$, there exists a unique $r = r(\epsilon)\in (0,\epsilon)$ such that
\begin{equation}
	\int_{\mathbb{T}}\rho_{\circ ,J}(\chi_{J_\circ }-\phi_{r,\epsilon})\rd{x} + \int_{\mathbb{T}} \rho_{\circ ,-}\rd{x} = 0.
\end{equation}
We thus write $\phi_{\epsilon}$ in short for $\phi_{r(\epsilon),\epsilon}$. Then we define
\begin{equation}\label{Usharp}
	U_\sharp =  W * \Big(m (\delta_M + \delta_{-M})+ \rho_{\circ ,J} \cdot (\chi_{J_\circ }-\phi_\epsilon) + \rho_{\circ ,-}\Big),
\end{equation}
and
\begin{equation}
	\rho_\sharp =  \rho_{\circ ,J}\phi_\epsilon + \rho_{\circ ,K}.
\end{equation}
By definition, $\rho_{\sharp} \in \cM_{1-2m}$ and is the unique minimizer of $\cE_{U_{\sharp}}$ since $V_{U_{\sharp}}[\rho_{\sharp}]  = W*\rho_{\circ}$ obtains minimal value on  $\supp\rho_\sharp$. Meanwhile it is also the unique minimizer in $\cM_{m_1, m_2}$ where 
\begin{equation}
	m_1 = \int_{K_\circ }\rho_\circ \rd{x}, \quad\quad  m_2 = 1-2m-m_1.
\end{equation}
Next we define
\begin{equation}\label{Uflat}
	U_\flat =  W * \big(m(\delta_M+\delta_{-M})\big),
\end{equation}
and let $\rho_\flat$ be the unique minimizer in $\cM_{m_1,m_2}$ for $\cE_{U_{\flat}}$. It is clear that 
\begin{equation}
	U_*:=U_\sharp-U_\flat=W*(\rho_{\circ ,J}(\chi_{J_\circ }-\phi_\epsilon) + \rho_{\circ ,-}),
\end{equation}
 is $C^1$ since $\rho_{\circ, c}$ is H\"older continuous. We will verify \eqref{lem_comp1_1} in Lemma \ref{lem:U-prime}. Then by Lemma \ref{lem_comp1} to get 
\begin{equation}
	( W*\rho_\flat)(0) > ( W*\rho_\sharp)(0),\quad ( W*\rho_\sharp)(\frac{1}{2}) > ( W*\rho_\flat)(\frac{1}{2}).
\end{equation}
By the construction of $U_\sharp$ and $W(x)'<0$ for $x\in (0,1/2)$, we have 
\begin{equation}
	U_\sharp(0)<U_\flat(0), \quad \quad U_\sharp(\frac{1}{2})>U_\flat(\frac{1}{2}).
\end{equation}
Therefore
\begin{equation}\label{rhoflatcomp}
	( W*\rho_\flat)(0)+U_\flat(0) > ( W*\rho_\sharp)(0)+U_\sharp(0) \ge  ( W*\rho_\sharp)(\frac{1}{2}) +U_\sharp(\frac{1}{2}) > ( W*\rho_\flat)(\frac{1}{2})+U_\flat(\frac{1}{2}).
\end{equation}
Here the second inequality is a consequence of  $\rho_\sharp$ being an energy minimizer in $\cM_{1-2m}$, together with the fact that $1/2\in \supp\rho_\sharp$. There $\rho_{\flat} + m(\delta_M+\delta_{-M})$ must be of Type $\tii$ in Proposition \ref{prop_rhoii}. Notice that we can assume $0\in \supp \rho $ since otherwise $\supp \rho \cap (-M, M) = \empty$ and \eqref{Drhocomp} is trivial. It follows from Corollary \ref{cor_Emin} that $(W*\rho)(0)\le(W*\rho)(1/2)$. By Proposition \ref{prop:2-parameter}, $( W*\rho)(0)-( W*\rho)(1/2)$ is increasing in $m_1=\int_{(-M,M)}\rho\rd{x}$ , therefore
\begin{equation}
	\int_{(-M,M)}\rho_\flat\rd{x} \ge \int_{(-M,M)}\rho\rd{x}.
\end{equation}
	Combined with the fact that
\begin{equation}
	\int_{(-M,M)}(\rho_\circ )_+\rd{x} =m_1= \int_{(-M,M)}\rho_\flat \rd{x},
\end{equation}
and that $\cD[\rho] = \int_{[-M, M]} \rho \rd{x}$, we then obtain \eqref{Drhocomp}.
\end{proof}

We now verify \eqref{lem_comp1_1} for the constructed potential $U_{\sharp}$ and $U_{\flat}$ in Theorem \ref{thm_Dcomp}.  
\begin{lemma}\label{lem:U-prime}
Given $U_{\flat}$ and $U_{\sharp}$ defined in \eqref{Uflat} and \eqref{Usharp}. The potential $U_*:=U_\sharp-U_\flat$ satisfies \eqref{lem_comp1_1} if $\epsilon>0$ is small enough.
\end{lemma}
\begin{proof}
We write
	$U_*=W*F(x)$ where 
	\begin{equation}\label{eqF}
		F(x) := \rho_{\circ ,J}(\chi_{J_\circ }-\phi_\epsilon) + \rho_{\circ ,-} = \left\{\begin{array}{ll}
			\rho_{\circ,-}, & 0\le |x| \le R_\circ \\
			\rho_{\circ,J}, & R_\circ < |x| \le x_J+r-\epsilon \\
			\rho_{\circ,J}\cdot \frac{x_J+r-x}{\epsilon}, & x_J+r-\epsilon < |x| \le x_J+r \\
			0, & x_J+r < |x| \le \frac{1}{2}
		\end{array}\right.
	\end{equation}
	Here $r = r(\epsilon)$ is in $(0,\epsilon)$ and $F(x)$ has mean-zero $F(x)\le 0$ on $\{0\le |x| \le R_\circ\}$ and $F(x)\ge 0$ on the complement. Our goal is to show that $(W*F)'<0$ on $\{x_J+r-\epsilon \le x < 1/2 \}$ and $0< x\le L_{\circ}$. 
	
	First notice that for any $0<x<y<1/2$ and $z\in (0,x)\cup (y, 1/2)$ we have
	\begin{equation}\label{W4_2}
		\begin{split}
		&W'(z-y)+W'(z+y)-W'(z-x)-W'(z+x) \\
		 = &-\int_{z-y}^{z-x}W''(u)\rd{u} + \int_{z+x}^{z+y}W''(u)\rd{u} = \int_x^y (-W''(z-u)+W''(z+u))\rd{u}<0,
		\end{split}
	\end{equation}
similar to the proof of \eqref{W4term1}. Then due to the even property and the signs of $F$, it is clear that an integration of \eqref{W4_2} gives $(W*F)'<0$ on $\{x_J+r \le x < 1/2\}$ and $0< x\le K_{\circ}$.
	
We then focus on $x_J+r-\epsilon \le x_0 \le x_J+r$. Then we claim that
	\begin{equation}\label{claimF}
		F(2x_0-x) \ge F(x),\quad \forall x\in [x_0,x_J+r]
	\end{equation}
	for $\epsilon$ small enough. In fact, Proposition \ref{prop:rho-circ-c-d} shows that $\rho_\circ$ is smooth near $x_0$ with $\rho_\circ(x_0)>0$. In particular, 
	\begin{equation}
		\rho_\circ(2x_0-x)\ge \rho_\circ(x_0)  - (\rho_\circ'(x_0)+1)(x-x_0),\quad \rho_\circ(x) \le \rho_\circ(x_0)  + (\rho_\circ'(x_0)+1)(x-x_0)
	\end{equation}
	Therefore, considering the possibility of $2x_0-x$ lying in the second or third piece of \eqref{eqF}, we have
	\begin{equation}
		\begin{split}
			&F(2x_0-x)-F(x) \\
			\ge & \big(\rho_\circ(x_0)  - (\rho_\circ'(x_0)+1)(x-x_0)\big)\min \Big\{1,\frac{x_J+r-(2x_0-x)}{\epsilon}\Big\} \\
			& - \big(\rho_\circ(x_0)  + (\rho_\circ'(x_0)+1)(x-x_0)\big)\frac{x_J+r-x}{\epsilon} \\
     \ge & \frac{1}{\epsilon}\rho_\circ(x_0)\big(\min\{\epsilon,x_J+r-(2x_0-x)\}-(x_J+r-x)\big) - 2\big|\rho_\circ'(x_0)+1\big|(x-x_0).
	\end{split}\end{equation}
	Notice that $(x_J+r-(2x_0-x))-(x_J+r-x)=2(x-x_0)$ and $\epsilon-(x_J+r-x)\ge (x_J+r-x_0)-(x_J+r-x)=x-x_0$. Therefore the last quantity above is nonnegative if $\epsilon$ is small enough, which proves \eqref{claimF}.
	
	Then we define a function $G$ by
	\begin{equation}\label{eqG}
		G(x) :=  \left\{\begin{array}{ll}
			\rho_{\circ,-}, & 0\le |x| \le R_G \\
			F(2x_0-x), & 2x_0-(x_J+r) < |x| \le x_0 \\
			F(x), & x_0 < |x| \le x_J+r \\
			0, & \text{otherwise}.
		\end{array}\right.
	\end{equation}
	Here $R_G\in (0,R_\circ)$ is determined by $\int_{\mathbb{T}}G\rd{x}=0$, which is possible since the positive parts of $G$ are below $F$, by \eqref{claimF}. Then $F-G$ is supported on $\{|x|\le x_0\}$ with the same sign properties as $F$, and we may apply \eqref{W4_2} to show that $(W*(F-G))(x_0)>0$, and in fact, bounded  from below uniformly in $\epsilon$ and $x_0$. Since $G$ is symmetric around $x_0$ in a small neighborhood of $x_0$, this neighborhood makes no contribution to $(W*G)(x_0)$, and other parts only contribute $O(\epsilon)$ to $(W*G)(x_0)$ since they are away from $x_0$ with total mass $O(\epsilon)$. Therefore we get the positivity of $(W*F)(x_0)$ for all $x_J+r-\epsilon \le x_0 \le x_J+r$ if $\epsilon$ is small enough.	
\end{proof}

\section{From Continuum to Discrete}\label{sec:approximation}
In this chapter we finalize the proof of Theorem \ref{thm:probability-distribution} on the minimization problem of $\cG$ for measures. Then we show that it implies Theorem \ref{thm:main} via a construction of polynomials whose roots distribution approximates the minimizing distribution over $\bT$. Finally based on our proof of Theorem \ref{thm:main}, we prove Theorem \ref{thm:real-root} on giving a sharp bound of the number of signed real roots for an arbitrary complex polynomial. 

\subsection{Proof of Main Theorem}\label{ssec:final}
In this section, we collect all results we prove before and deduce the main theorem.
\begin{proof}[Proof of Theorem \ref{thm:main} and Theorem \ref{thm:probability-distribution}]
	Recall that we first extend the functional $\cD$ and $\cH$ from the set of discrete probability measures 
	\begin{equation}
		\cM_{\text{emp}}: = \Big\{\rho\in\cM: \rho=\sum_{j=1}^Nc_j\delta(x-x_j),\,N\in\mathbb{N},\, c_j\in \R_{\ge 0},\,x_j\in\mathbb{T}\Big\},
	\end{equation}
to all probability measures $\cM$ on $\bT$. We then show in Theorem \ref{thm:main-characterization} that the minimizer of $\cG$ in $\cM_{\cD \ge d}$ must be in the form of $\rho = m(\delta_M + \delta_{-M}) + \rho_1$ where $\rho_1$ is the sediment distribution in $\cM_{1-2m}$ w.r.t. $U = W*m(\delta_M + \delta_{-M})$. Depending on the size of $M$ and $m$, we have constructed the sediment distributions in Proposition \ref{prop_rhoi} ($M = 0$) and \ref{prop_rhoii} ($M\neq 0$) and Proposition \ref{prop:2-parameter}. We then show that $\cG[\rho] > 1/2$ for sediment $\rho$ with $M\neq 0$ by combining Theorem \ref{thm:min-value-R} and Theorem \ref{thm_Dcomp}. 

We now compute the value of $\cG[\rho]$ for sediment $\rho$ with $M = 0$. In Proposition \ref{prop_rhoi} we construct $\rho = \rho_{\ti, m}(x) \in \cM$ parametrized by $m$. By expression of $\rho$, we have $\cD[\rho] =2m$ is taken at the origin since $\rho(x)<1$ for $x\neq 0$. On the other hand, it follows from the construction that $W*\rho$ is stationary. By Lemma \ref{lem_complexT}, by letting $C_1 = 0$, we obtain
\begin{equation}\label{dWrhoi}
	(W*\rho_\ti)'(x) = - \pi\frac{\sqrt{-(\sin^2\pi x-(2m)^2)}}{\sin\pi x}\chi_{|x|\le \frac{1}{\pi}\sin^{-1}2m}.
\end{equation}
Due to the sign of $W*\rho_{\ti}(x)$, we see that $\cH[\rho] = -(W*\rho)(1/2)$. We can compute the value using mean-zero property of $W*\rho$
\begin{equation}\label{eqn:W-rho-1/2}\begin{split}
		(W*\rho_\ti)(\frac{1}{2})
		= & \frac{-2}{1-\frac{2\sin^{-1}2m}{\pi}}\int_0^{ \frac{\sin^{-1}2m}{\pi}}W*\rho_\ti\rd{x}\\
		= & \frac{-2}{1-\frac{2\sin^{-1}2m}{\pi} }\Big(\frac{\sin^{-1}2m}{\pi}\cdot(W*\rho_\ti)(\frac{1}{2}) -
		\int_0^{\frac{\sin^{-1}2m}{\pi}}(W*\rho_\ti)'(x)x\rd{x}\Big),
\end{split}\end{equation}
where the last equality follows from integration by parts and the fact that $(W*\rho_{\ti})(1/2) = (W*\rho_{\ti})(\frac{\sin^{-1}(2m)}{\pi})$. It then follows from \eqref{eqn:W-rho-1/2} and \eqref{dWrhoi} that
\begin{equation}\label{Wrhoi}\begin{split}
		-(W*\rho_\ti)(\frac{1}{2})= & -2\int_0^{\frac{1}{\pi}\sin^{-1}2m}(W*\rho_\ti)'(x)x\rd{x} = 2\pi \int_0^{\frac{1}{\pi}\sin^{-1}2m}x\sqrt{\frac{4m^2}{\sin^2\pi x}-1}\rd{x} \\
		= & \pi \int_0^{\frac{1}{\pi}\sin^{-1}2m}\frac{1}{\pi}\sin^{-1}(2my)\sqrt{\frac{1}{y^2}-1}\cdot\frac{2m}{\pi\sqrt{1-(2my)^2}}\rd{y} \\
		= & \frac{8m^2}{\pi} \int_0^1\sqrt{1-y^2}\cdot  \frac{\sin^{-1}(2my)}{2my\sqrt{1-(2my)^2}}\rd{y}> 2m^2.\\
\end{split}\end{equation}
Therefore we prove the inequality in Theorem \ref{thm:probability-distribution}. One can also evaluate the limit $\lim_{m\to 0+} (W*\rho_\ti)(\frac{1}{2})/m^2  = 2$, thus
\begin{equation}\label{rhoi_1}
	\lim_{m\rightarrow0+}\cG[\rho_{\ti,m}] = \frac{1}{2}.
\end{equation}
Therefore we also show that $\sqrt{2}$ sharp in Theorem \ref{thm:probability-distribution}, thus finish proving Theorem \ref{thm:probability-distribution}. Since $\cM_{\text{emp}}\subset \cM$, we also prove the inequality in Theorem \ref{thm:main}. 

It now suffices to prove that Theorem \eqref{eqn:main} is sharp. We will do so by constructing $\rho_{\epsilon}$ in
	\begin{equation}
	\cM_{\text{emp-rat}} = \Big\{\rho\in\cM: \rho=\sum_{j=1}^Nc_j\delta(x-x_j),\,N\in\mathbb{N}, c_j\in \Q_{\ge 0}, x_j\in\mathbb{T}\Big\},
\end{equation}
the set of empirical measures with rational coefficients, such that $\cG[\rho_{\epsilon}]< 1/2+\epsilon$. For each $\rho = \rho_{\ti, m} = 2m\delta +\rho_{\ti, c} \in \cM$, we claim that there exists $\{ \rho_n \} \subset \cM_{\text{emp}}$ such that $\rho_n\wc \rho$ and  $\limsup_{n\to \infty}\cH[\rho_n]\le \cH[\rho]$. For each $n$, we define
\begin{equation}
	\rho_n = 2m\delta + \sum_{j=0}^{n-1} \Big(m_{j,1}\delta(x-\frac{j}{n}) + m_{j,2}\delta(x-\frac{j+1}{n}) \Big),
\end{equation}
where $m_{j,1},m_{j,2}\ge 0$ are determined by the moment conditions
\begin{equation}
	\int_{j/n}^{(j+1)/n}\Big(m_{j,1}\delta(y-\frac{j}{n}) + m_{j,2}\delta(y-\frac{j+1}{n})-\rho_{\ti, c}(y)\Big) y^k \rd{y} = 0, \quad k = 0, 1.
\end{equation}
The weak convergence of $\{\rho_n\}$ to $\rho$ is clear. Similar to the proof of \eqref{eqn:micro-diffusion-away} in Lemma \ref{lem:micro-diffusion}, we may show that
\begin{equation}\label{jj1}
	\Big(W*\big(m_{j,1}\delta(\cdot-\frac{j}{n}) + m_{j,2}\delta(\cdot-\frac{j+1}{n})- \rho_{\ti, c}\chi_{[j/n,(j+1)/n]}\big)\Big)(x) \ge 0,\quad \forall x\notin [j/n,(j+1)/n].
\end{equation}
using the convexity of $W$. Therefore applying \eqref{jj1} to those $j$ with $x\notin [j/n,(j+1)/n)$, we have
\begin{equation}
	(W*(\rho_n -\rho))(x) \ge \Big(W*\big((m_{j_x,1}\delta(\cdot-\frac{j_x}{n}) + m_{j_x,2}\delta(\cdot-\frac{j_x+1}{n}) ) -  \rho_{\ti, c}\chi_{[j_x/n,(j_x+1)/n]}\big)\Big)(x),
\end{equation}
where $j_x$-th interval contains $x$. For large $n$, we have $(W*\delta)(x)>0$ for $|x|< 1/n$ since $W$ is positive near 0. We can also bound the other term using $\rho_{\ti, c}\le 1$ and $W= - \log |2 \sin \pi x|$,
\begin{equation}
	|\big(W*(\rho_{\ti, c}\chi_{[j_x/n,(j_x+1)/n]})\big)(x)| \le \sup_{a\in\mathbb{T}}\int_a^{a+1/n}|W|\rd{x} \lesssim \frac{\log n}{n}.
\end{equation}
Therefore we obtain
\begin{equation}
	(W*\rho_n)(x)-(W*\rho)(x) \ge -C \frac{\log n}{n},
\end{equation}
which implies
\begin{equation}
	\cH[\rho_n] \le \cH[\rho] + C \frac{\log n}{n}.
\end{equation}
and therefore $\limsup_{n\rightarrow\infty}\cH[\rho_n] \le \cH[\rho]$. Combining with Lemma \ref{lem:D-H-property}, we have that $\lim_{n\to \infty} \cD[\rho_n] = \cD[\rho]$ and $\lim_{n\to \infty} \cH[\rho_n] = \cH[\rho]$, and therefore $\lim_{n\to\infty} \cG[\rho_n] = \cG[\rho]$. 

Therefore for each $\epsilon>0$, we can find $\rho = \rho_{\ti, m}$ such that $\cG[\rho]< 1/2+\epsilon$. For this $\rho$, we can construct $\{ \rho_{n}\}$ as above so that $\cG[\rho_n] < 1/2+2\epsilon$ for $n$ large enough. For a $\rho_{\text{emp}} =\sum_{j=1}^Nc_j\delta(x-x_j),\,N\in\mathbb{N}\in \cM_{\text{emp}}$, we can view $\cG[\rho_{\text{emp}}]$ as a function of $\vec{c}=(c_1,\dots,c_N)$, and indeed a continuous function in terms of $\vec{c}$. Therefore, by replacing each $c_j$ by a nearby rational number while keeping $\sum_{j=1}^N c_j=1$, one can find a $\rho_{n}'\in\cM_{\text{emp-rat}}$ with $\cG[\rho_{n}'] < \frac{1}{2}+3\epsilon$. Since $\epsilon$ is arbitrary, we finish proving the sharpness of the constant $\sqrt{2}$ in Theorem \ref{thm:main}. 
\end{proof}
\begin{remark}
Notice that we actually show that $\cG[\rho]$ is strictly larger than $1/2$. Indeed, by Theorem \ref{thm:min-value-R} we see that $\cG[\rho]$ cannot equal to $1/2$ for $\rho$ with $M>0$, and we have just computed in the proof above that $\cG[\rho]>1/2$ for $\rho$ with $M=0$.   Therefore $1/2$ cannot be achieved for any $\rho \in \cM$, although we have constructed a family of $\rho_m$ where $\cG[\rho_m]$ can be arbitrarily close to $1/2$. It then also follows that $1/2$ cannot be achieved for any polynomial. 
\end{remark}

\subsection{Application towards Real Roots}\label{ssec:real-final}
In this section, we give the proof for Theorem \ref{thm:real-root}, which is a consequence of Theorem \ref{thm:main}. We can again extend the discrete question on polynomials to a continuous question about probability measures. For each $\rho \in \cM$, we define a functional 
\begin{equation}
	\cR[\rho]: = \int_{\{ 0 \}} \rho \rd{x} = \int_{\{0\}}(\rho-1) \rd{x}.
\end{equation}
It is easy to see that if $\rho_f = \frac{1}{n}\sum_j \theta_j$ is the empirical measure from a degree $n$ polynomial $f(z)$, then $\cR[\rho_f] = N_+(f)/n$. 

\begin{proof}[Proof of Theorem \ref{thm:real-root}]
It follows from the definition of $\cD$ that
\begin{equation}
	\cR[\rho] \le \cD[\rho] \le \sqrt{2}\cdot \sqrt{\cH[\rho]}.
\end{equation}
This implies that for $\rho = \rho_f$ that
\begin{equation}
	N_+(f) \le \sqrt{2} \cdot \sqrt{\cH[f]} \cdot n.
\end{equation}
Therefore it suffices to prove the inequality is sharp. 

Notice that in the proof of Theorem \ref{thm:main}, we have constructed $\rho_{\ti, m}$ such that $\cG[\rho_{\ti, m}]< 1/2+\epsilon$ for any $\epsilon>0$. By the expression of $\rho_{\ti, m}$ we see that $\cD[\rho_{\ti, m}] = \cR[\rho_{\ti,m}]$, therefore we also have for these $\rho_{\ti, m}$ that $\cH[\rho_{\ti, m}]/ \cR[\rho_{\ti, m}]^2 <1/2+\epsilon$. We then construct $\rho_n \wc \rho_{\ti, m}$ in the same way to approximate $\rho  = \rho_{\ti, m}$. Since $\lim_{n\to \infty} \cD[\rho_n] = 2m = \cR[\rho_n]$, we also have $\lim_{n\to \infty} \cH[\rho_n]/\cR[\rho_n]^2 = 1/2$. Therefore we can choose $\rho_n$ such that $ \cH[\rho_n]/\cR[\rho_n]^2 < 1/2+2\epsilon$. Finally the construction for $\rho_n'$ is the same since $\cR$ is also continuous in $\vec{c}$ when $\rho_{\text{emp}} = \sum_{j=1}^Nc_j\delta(x-x_j)$. Therefore we can find $\rho_n'$ with $\cH[\rho_n']/\cR[\rho_n']^2  < 1/2+3\epsilon$. Since $\epsilon$ is arbitrary,  we finish proving the sharpness of constant $\sqrt{2}$. 

The upper bound for $N_{\theta}(f)$ is exactly the same since $N_{\theta}(f(z)) = N_+(f(z\cdot e^{-2\pi \theta i }))$ and $\cH[f(z)] = \cH[f(z\cdot e^{-2\pi\theta i})]$. 
\end{proof}

\section{Formulation in Harmonic Functions}\label{sec:conjugate-function}
In his 1952 work \cite{Ganelius}, Ganelius formulates a question in harmonic functions and uses it to improve the constant in the original Erd\H{o}s-Tur\'an inequality proved by \cite{ET}. This approach of harmonic functions has been further developed by Mignotte in \cite{Mignotte}. In this section, our goal is to show that our sharp version of Erd\H{o}s-Tur\'an inequality in turn implies a sharp upper bound for harmonic functions in Ganelius' formulation.

\begin{theorem}[Ganelius, 1952]
	Let $f(z)= u(z)+i v(z)$ be an analytic function in $|z|<1$ with $f(0)=0$. Suppose $u(z)<H$ and $\frac{\partial v}{\partial \theta} (z)<K$ in $|z|<1$ where $H, K >0$, then there exists $C>0$ such that
	\begin{equation}\label{thm_gane_0}
	|v(z_1)-v(z_2)|\le C \sqrt{HK},  \quad \quad \text{ for } |z_1|,|z_2|< 1.
\end{equation}
	Moreover the constant $C$ can be taken to be $\sqrt{2\pi} \sqrt{\pi/k} \approx 4.64$ where $k = \sum_{m\ge 0} (-1)^{m-1} (2m+1)^{-2}$ is the Catalan constant. \footnote{Notice that $\theta$ is taken to be in $[0,1]$ in this current formulation whereas in \cite{Ganelius} and \cite{Mignotte} $\theta$ is taken to be in $[0, 2\pi]$. This results in a change of $\sqrt{2\pi}$ in the constant $C$.}
\end{theorem}

We now prove Theorem \ref{thm:conjugate-functions}, which gives the improvement of Ganelius's theorem by replacing $C$ with the sharp constant $\sqrt{2\pi}$.
\begin{proof}[Proof of Theorem \ref{thm:conjugate-functions}]
In order to prove inequality \eqref{thm_gane_1}, we first note that it suffices to prove \eqref{thm_gane_1} for $|z|= 1$ in the case where $f(z)$ is analytic in $|z|<1+\epsilon$ and $u\le H$ and $\frac{\partial v}{\partial \theta} (z)\le K$. Indeed, if $|z_i|<1-\epsilon$ for $i = 1, 2$, we can consider $g(z) = f(z(1-\epsilon))$ instead. We still have the bound $u_g\le H$ and $\frac{\partial v_g}{\partial \theta} (z)\le K(1-\epsilon)\le K$. Now since $f(z)$ is analytic on $|z|< 1+\epsilon$, $u$ and $v$ are both harmonic functions on $|z|<1+\epsilon$. The \eqref{thm_gane_1} for $z_i$ and $f$ then follows from that for $z_i/(1-\epsilon)$ and $g$. By maximal value principle, $\sup_{|z_1|,|z_2|\le 1}|\tilde{v}(z_1)-\tilde{v}(z_2)|$ is achieved at $|z_1|,|z_2|= 1$. Meanwhile we can also assume $K=1$ without loss of generality by multiplying $f$ by $1/K$.
	
Denote the restriction of $u$ and $v$ on $|z|=1$ as $\tilde{u}(\theta)$ and $\tilde{v}(\theta)$. We now define
	\begin{equation}
		\rho(\theta) = 1-\tilde{v}'(\theta). 
	\end{equation}
Then $\rho \in \cM$	since $\int_{\bT} \rho(\theta) \rd{\theta} = 1$ and $\rho(\theta)>0$ by $\frac{\partial v}{\partial \theta} (z)\le 1$. Given $f(0)=0$, it is a standard property of Hilbert transform on $\bT$ that
	\begin{equation}
		\tilde{u} =\frac{1}{\pi} \pv W'*\tilde{v} = \frac{1}{\pi}W*\tilde{v}' = -\frac{1}{\pi} W*\rho.
	\end{equation}
	Therefore, $W*\rho \ge -\pi H$ and $\cH[\rho] \le \pi H$. By Theorem \ref{thm:probability-distribution}, we have $\cD[\rho] \le \sqrt{2}\cdot \sqrt{\cH[\rho]}$, thus
	\begin{equation}
		\sqrt{H} \ge \sqrt{\cH[\rho]/\pi} \ge \frac{1}{\sqrt{2\pi}}\cD[\rho] = \frac{1}{\sqrt{2\pi}}\sup_I \int_I(\rho-1)\rd{x} = \frac{1}{\sqrt{2\pi}}\sup_{a,b\in\mathbb{T}}(\tilde{v}(a)-\tilde{v}(b)),
	\end{equation}
which proves the \eqref{thm_gane_1}.
	
	Then we show that the constant in \eqref{thm_gane_1} is sharp. By the proof of Theorem \ref{thm:probability-distribution}, for any $\epsilon>0$, we may take $\rho_{\ti,m}\in\cM$ as given by \eqref{typeirho} for some $m>0$ such that $\frac{\cH[\rho_{\ti,m}]}{\cD[\rho_{\ti,m}]^2} < \frac{1}{2}+\epsilon$. Proposition \ref{prop_rhoi} shows that $\ess\inf(W*\rho_{\ti,m})=-\cH[\rho_{\ti,m}]$ is achieved on an interval $[\frac{1}{\pi}\sin^{-1}2m,1-\frac{1}{\pi}\sin^{-1}2m]$. Therefore, by taking convolution with a compactly supported mollifier $\phi$, we obtain a nonnegative smooth function $\rho=\phi*\rho_{\ti,m}\in\cM$ with $\cH[\rho]= \cH[\rho_{\ti,m}]$ since $W*\rho = \phi*(W*\rho_{\ti,m})$. Also, one can choose $\phi$ so that $\cD[\rho]$ is arbitrarily close to $\cD[\rho_{\ti,m}]$ by Lemma \ref{lem:D-H-property}, and this guarantees we can find continuous $\rho$ such that $\frac{\cH[\rho]}{\cD[\rho]^2} < \frac{1}{2}+\epsilon$. Now we define
	\begin{equation}
		\tilde{v}(\theta) = \int_0^\theta (1-\rho(t))\rd{t} -c_{\rho},\quad \tilde{u}(\theta) = \frac{1}{\pi} \pv (W'*\tilde{v})(\theta)= -\frac{1}{\pi}(W*\rho)(\theta),
	\end{equation}
	where $c_{\rho}$ is a constant which makes $\int_{\mathbb{T}}\tilde{v}(\theta)\rd{\theta}=0$. Here $\tilde{v},\tilde{u}$ are both smooth. Therefore, we may construct an analytic function $f=\tilde{u}+i\tilde{v}$ in $|z|<1$ by Poisson integral
	\begin{equation}
		f(re^{2\pi i\theta}) = \int_{\mathbb{T}} P_r(\theta-t)(\tilde{u}(t)+i\tilde{v}(t))\rd{t},\quad P_r(\theta)=  \frac{1-r^2}{1-2r\cos(2\pi\theta)+r^2}.
	\end{equation}
	and $f$ is continuous on $|z|\le 1$. Clearly $f(0)=0$ because $\tilde{u},\tilde{v}$ are mean-zero. By the harmonic property of $u$, we see that
	\begin{equation}
		\sup_{|z|<1} u(z) = \sup_{|z|=1}u(z) = -\frac{1}{\pi}\inf (W*\rho) =\frac{1}{\pi} \cH[\rho].
	\end{equation}
Now using $v_{r}(\theta) = (P_r * \tilde{v})(\theta)$, we have $\partial v /\partial \theta  = P_r * \tilde{v}'$. Again the harmonic property of $P_r * \tilde{v}'$ implies
	\begin{equation}
		\sup_{|z|<1}\partial v /\partial \theta(z) = \sup_{|z|=1}\partial v /\partial \theta(z) = \sup_{\theta\in\mathbb{T}}(1-\rho(\theta)) \le 1.
	\end{equation}
	Therefore $f$ satisfies the assumptions of this theorem with $H=\frac{1}{\pi}\cH[\rho]$ and $K=1$. On the other hand, since $v$ is harmonic on $|z|<1$ and continuous on $|z|\le 1$,
	\begin{equation}
		\sup_{|z_1|,|z_2|<1}|v(z_1)-v(z_2)| = \sup_{a,b\in\mathbb{T}}(\tilde{v}(a)-\tilde{v}(b)) = \sup_I \int_I(\rho-1)\rd{x} = \cD[\rho].
	\end{equation}
Therefore we find $f$ such that 
	\begin{equation}
		\sup_{|z_1|,|z_2|<1}|v(z_1)-v(z_2)| \ge  \sqrt{\pi H}/\sqrt{1/2+\epsilon}.
	\end{equation}
	Since $\epsilon$ is arbitrary, this shows the sharpness of the constant in \eqref{thm_gane_1}.	
\end{proof}

\begin{remark}
	By taking the difference with $v(0)$, the inequality \eqref{thm_gane_0} and \eqref{thm_gane_1} in both theorems imply that 
	\begin{equation}\label{thm_gane_2}
		|v(z)|\le C \sqrt{HK},  \quad \quad \text{ for } |z|< 1,
	\end{equation}
	with the same constant $C$, i.e. $C = \sqrt{2\pi} \sqrt{\pi/k}$ and $C=\sqrt{2\pi}$ respectively. However, the constant $\sqrt{2\pi}$ in \eqref{thm_gane_2} is not necessarily sharp.
\end{remark}

\section{Appendix: Continuity of Potential}
We list several results on the continuity of the generated potential $V = W*\rho$. Recall that a function $f$ is \emph{lower semicontinous} at $x = a\in \bT$ if 
\begin{equation}
	\liminf_{x\to a} f(x) \ge f(a).
\end{equation}

\begin{proposition}\label{prop:semi}
	Assume $W:\mathbb{T} \to (-\infty, \infty]$ satisfies {\bf (H1)}-{\bf (H4)}. For $\rho \in \cM$, denote $V = W*\rho$. 
	\begin{enumerate}
	\item[\textnormal{(i)}]
		The generated potential $V$ is lower semicontinuous and is $C^2$ on $\mathbb{T}\backslash \supp \rho$.
	\item[\textnormal{(ii)}]
		There exists a constant $C>0$ such that $V''(x)\ge C$ at every $x\notin \supp\rho$.
	\item[\textnormal{(iii)}]
		Let $(x_1,x_2)\subseteq (\supp\rho)^c$ be an interval with endpoints $x_1,x_2\in \supp\rho$. Then $V$ is right continuous at $x_1$ and left continuous at $x_2$.
	\end{enumerate}
\end{proposition}
\begin{proof}
\textbf{Proof of (i):}
The lower semicontinuity of $V$ for general $W$ can be found in \cite[Lemma 2]{Ca13} and for $W = -\ln|2\sin (\pi x)|$ is automatic. 
For the sake of completeness, we include the proof here. The continuity of $W$ away from 0, together with {\bf (H3)}, implies that $W$ is bounded from below, say, by $-C_1$. By including the possibility of $\infty$ value, we have $\lim_{x\rightarrow x_0}W(x) = W(x_0)$ for any $x_0\in\mathbb{T}$. Therefore, for any $x\in\mathbb{T}$ and sequence $\{x_n\}$ with $\lim_{n\rightarrow\infty}x_n = x$,
\begin{equation}\begin{split}
		V(x)+C_1 = & \int_{\mathbb{T}} (W(x-y)+C_1)\rho(y)\rd{y} = \int_{\mathbb{T}} \lim_{n\rightarrow\infty}(W(x_n-y)+C_1)\rho(y)\rd{y} \\
		\le & \liminf_{n\rightarrow\infty}\int_{\mathbb{T}}(W(x_n-y)+C_1)\rho(y)\rd{y} =  \liminf_{n\rightarrow\infty}V(x_n)+C_1.
\end{split}\end{equation}
The inequality uses Fatou's lemma on the nonnegative functions $W(x_n-y)+C_1$. The lower semicontinuity of $V$ then follows.

\textbf{Proof of (ii):} By {\bf (H4)}, $W$ is $C^2$ on $\mathbb{T}\backslash\{0\}$ with $W''(y)\ge C_W>0$ for any $y\in\mathbb{T}\backslash\{0\}$. If $x\notin\supp\rho$, then there exists $\epsilon>0$ such that $(x-\epsilon,x+\epsilon)\notin \supp\rho$, and
\begin{equation}
	V(x) = \int_{\mathbb{T}}W(x-y)\rho(y)\rd{y} = \int_{\mathbb{T}\backslash(x-\epsilon,x+\epsilon)}W(x-y)\rho(y)\rd{y}
\end{equation}
This shows that $V$ is $C^2$ in $(x-\epsilon,x+\epsilon)$ since $W(x-y)$ is $C^2$ on the domain of integral. Also, for
\begin{equation}
	V''(x) = \int_{\mathbb{T}\backslash(x-\epsilon,x+\epsilon)}W''(x-y)\rho(y)\rd{y} \ge C_W \int_{\mathbb{T}\backslash(x_0-\epsilon,x_0+\epsilon)}\rho(y)\rd{y}=C_W>0.
\end{equation}

\textbf{Proof of (iii):}
By reflection around $(x_1+x_2)/2$, it suffices to prove that $V$ is right continuous at $x_1$. Also, since we already know that $V$ is lower semicontinuous, it suffices to prove: for any $\epsilon>0$, there exists $\alpha>0$ such that
\begin{equation}
	V(y) < V(x_1)+\epsilon,\quad \forall y\in (x_1,x_1+\alpha).
\end{equation}
Also, it suffices to work with the case $V(x_1)<\infty$. First notice that $(x_1,x_2)$ is an interval in $\mathbb{T}$, and therefore we can take the representative $x_1<x_2<x_1+1$. We will take $\alpha\le\frac{x_2-x_1}{2}$. Since $\supp\rho\cap (x_1,x_2)=\emptyset$, we may write
\begin{equation}
	V(y) = \int_{[x_2-1,x_1]}W(y-z)\rho(z)\rd{z},
\end{equation}
where $[x_2-1,x_1]=(x_1,x_2)^c$ when considered as an interval of $\mathbb{T}$. Then
\begin{equation}
	V(y)-V(x_1) =  \int_{[x_2-1,x_1]}(W(y-z)-W(x_1-z))\rho(z)\rd{z} =  \int_{[x_2-1,x_1)}\int_{x_1-z}^{y-z}W'(u)\rd{u}\rho(z)\rd{z},
\end{equation}
where we dropped the point $z=x_1$ in the integral because $V(x_1)<\infty$ does not allow $\rho$ to have a Dirac mass at $x_1$. It then follows from the monotonicity of $W'$ that 
\begin{equation}
	V(y)-V(x_1)\le -W'(\frac{x_2-x_1}{2})(y-x_1),
\end{equation}
therefore when $\alpha$ is small enough, we have $V(y)-V(x_1)<\epsilon$. 
\end{proof}

\section{Appendix: Continuity of Energy}
In this part, we will study properties of the energy functional $\cE$, defined in \eqref{eqn:E}. In particular, we will prove that $\cE$ is continuous with respect to mollification. While doing so, we also give a version of Paserval's identity based on some assumptions of $W$.

\begin{proposition}\label{prop:energy-Fourier}
	Assume $W$ satisfies {\bf (H1)}-{\bf (H5)} and $U$ has the form \eqref{eqn:U}. Let $\rho_i$ for $i = 1,2$ and $\rho$ be probability measures on $\bT$. If $\cE[\rho]<\infty$, then
	\begin{equation}\label{lem_Fou1_0}
		\lim_{\alpha\rightarrow 0+} \cE[\rho*\psi_a] = \cE[\rho].
	\end{equation}	
	If $\cE[\rho_i]<\infty$, then
	\begin{equation}\label{lem_Fou1_1}
		\int_{\bT} (W*\rho_1)\cdot \rho_2 \rd{x} = \sum_{k\in\mathbb{Z}}\hat{ W}(k)\hat{\rho_1}(k)\bar{\hat{\rho_2}}(k).
	\end{equation}
\end{proposition}
\begin{proof}
	We first prove \eqref{lem_Fou1_0} in the case $U=0$. By assumption, $\inf W:=-C_0$ with $C_0>0$. To prove \eqref{lem_Fou1_0}, it suffices to prove
	\begin{equation}\label{Ealpha}
		\lim_{a\rightarrow0+}\int( W_1*\psi_a*\rho)(x)(\psi_a*\rho)(x)\rd{x} = \iint  W_1(x-y)\rho(y)\rd{y}\rho(x)\rd{x},
	\end{equation}
	where $ W_1= W+2C_0$ is bounded from below by $C_0$. Since $\psi_a$ is even, we obtain
	\begin{equation}\label{planc1}
		\int( W_1*\psi_a*\rho)(x)(\psi_a*\rho)(x)\rd{x} = \int (W_1*\Psi_a*\rho)(x)\rho(x)\rd{x} 
	\end{equation}
	where $\Psi_a=\psi_a*\psi_a$. The assumption $\cE[\rho]<\infty$ implies $\iint W_1(x-y)\rho(y)\rd{y}\rho(x)\rd{x} <\infty$. Therefore the measure of the line $x=y$ is $0$ since $ W_1(0)=\infty$. By {\bf (H1)} $W_1$ is continuous on $\mathbb{T}\backslash\{0\}$, then $\lim_{a\rightarrow0+}( W_1*\Psi_a)(x-y)= W_1(x-y)$ for any $x\ne y$, and therefore this convergence holds almost everywhere with respect to the measure $\rho(x)\rho(y)$ on $\mathbb{T}^2$. 
	
	Notice that $\psi_a(x)$ is a positive linear combination of the characteristic functions $\chi_{[-r,r]}$ with $r>0$, and same for $\Psi_a$. Combining with the assumption {\bf (H5)} (which is equivalent to $(W_1*\frac{1}{2r}\chi_{[-r,r]})(x) \le C_1W_1(x)$), we see that
	\begin{equation}\label{claim_wpsi}
		( W_1*\Psi_a)(x) \le  C_1 W_1(x),
	\end{equation} 
	for any $0<a<1/2$ and $x\in \mathbb{T}$. Therefore, the RHS integral of \eqref{planc1} is dominated by $C_1\iint  W_1(x-y)\rho(y)\rd{y}\rho(x)\rd{x}<\infty$. Combined with the convergence $\lim_{a\rightarrow0+}( W_1*\Psi_a)(x-y)= W_1(x-y)$ for  almost everywhere with respect to the measure $\rho(x)\rho(y)$, the dominated convergence theorem shows that 
	\begin{equation}
		\lim_{a\rightarrow0+}\iint ( W_1*\Psi_a)(x-y)\rho(y)\rd{y}\rho(x)\rd{x} = \iint  W_1(x-y)\rho(y)\rd{y}\rho(x)\rd{x}.
	\end{equation}
	which finishes the proof of \eqref{Ealpha}.
	
     Next we prove \eqref{lem_Fou1_0} for general $U$ and $\cE_U$. Since $U$ is bounded from below, the finiteness of $\cE[\rho]$ implies $\iint  W(x-y)\rho(y)\rd{y}\rho(x)\rd{x}<\infty$. We have shown the convergence of $W$-term in $\cE_{U}$ in $U=0$ case. Therefore, it suffices to show that
	\begin{equation}
		\lim_{a\rightarrow0+}\int U(x)(\rho*\psi_a)(x)\rd{x} = \int U(x)\rho(x)\rd{x},
	\end{equation}
	for $U= W*\rho_{U,+}$, $U= W*\rho_{U,-}$ or $U= W*\delta= W$, using the expression of $U$ in \eqref{eqn:U}. For the first two cases, we first observe
	\begin{equation}
		\int U(x)(\rho*\psi_a)(x)\rd{x} = \int (U*\psi_a)(x)\rho(x)\rd{x},
	\end{equation}
and then continuity of $U$ implies the uniform convergence of $U*\psi_a$ to $U$ on $\bT$ as $a\rightarrow0+$. For the case $U=W$, we again have 
	\begin{equation}
		\int  W(x)(\rho*\psi_a)(x)\rd{x} = \int ( W*\psi_a)(x)\rho(x)\rd{x}.
	\end{equation}
	Notice that $\cE[\rho]<\infty$ implies that $\int W \rho \rd{x} = \int U\rho\rd{x}<\infty$ since $W$ is bounded from below. Using \eqref{claim_wpsi} with $\Psi_a$ replaced by $\psi_a$, we see that $\lim_{a\rightarrow0+}\int ( W*\psi_a)(x)\rho(x)\rd{x} = \int W \rho \rd{x}$  by dominated convergence. 
	
	Finally we prove \eqref{lem_Fou1_1}. We first treat the case $\rho_1=\rho_2 = \rho$. We first notice that for any $0<a<1/2$,
	\begin{equation}\label{planc}
		\int( W_1*\psi_a*\rho)(x)(\psi_a*\rho)(x)\rd{x} = \sum_{k\in\mathbb{Z}}\hat{ W_1}(k)|\hat{\psi}(a k)|^2|\hat{\rho}(k)|^2,
	\end{equation}
	by Plancherel formula, since $W_1*\psi_a*\rho$ and $\psi_a*\rho$ are both continuous. 
	By \eqref{lem_Fou1_0}, the LHS of \eqref{planc} converges to $\iint  W_1(x-y)\rho(y)\rd{y}\rho(x)\rd{x}$. 
	
	To analyze the RHS of \eqref{planc}, we show that $\sum_{k\in\mathbb{Z}}\hat{ W_1}(k)|\hat{\rho}(k)|^2$ is finite. First notice that since $\hat{\psi}(0)=1$ and $\hat{\psi}(\xi)$ is a smooth even function, we have
	\begin{equation}\label{uniK}
		\lim_{a\rightarrow0+}|\hat{\psi}(a k)|^2 = 1,\quad \text{uniformly for $k\in [-K,K]$},
	\end{equation}
	for any $K\in\mathbb{Z}_+$. Therefore
	\begin{equation}\begin{split}
			\sum_{k\in\mathbb{Z},\,|k|\le K}\hat{ W_1}(k)|\hat{\rho}(k)|^2 = & \lim_{a\rightarrow0+} \sum_{k\in\mathbb{Z},\,|k|\le K}\hat{ W_1}(k)|\hat{\psi}(a k)|^2|\hat{\rho}(k)|^2 \\
			\le &  \lim_{a\rightarrow0+} \sum_{k\in\mathbb{Z}}\hat{ W_1}(k)|\hat{\psi}(a k)|^2|\hat{\rho}(k)|^2 \\
			 =  & \iint  W_1(x-y)\rho(y)\rd{y}\rho(x)\rd{x}<\infty.
	\end{split}\end{equation}
The first inequality uses the positivity of $\hat{ W}_1$ and $|\hat{\psi}|^2$, and the last equality uses 
\eqref{Ealpha} and \eqref{planc}. Therefore the RHS of \eqref{planc} is dominated by $\sum_{k\in\mathbb{Z}}\hat{ W_1}(k)|\hat{\rho}(k)|^2$ since $|\hat{\psi}(a k)|\le 1$ for any $a$ and $k$. Then we see that the RHS of \eqref{planc} converges to $\sum_{k\in\mathbb{Z}}\hat{ W_1}(k)|\hat{\rho}(k)|^2$.
	
	For the general case with possibly $\rho_1\ne \rho_2$, we use the bilinear property
	\begin{equation}\begin{split}
			2\iint &  W(x-y)\rho_1(y)\rd{y}\rho_2(x)\rd{x} = \iint  W(x-y)(\rho_1(y)+\rho_2(y))\rd{y}(\rho_1(x)+\rho_2(x))\rd{x} \\
			& - \iint  W(x-y)\rho_1(y)\rd{y}\rho_1(x)\rd{x} - \iint  W(x-y)\rho_2(y)\rd{y}\rho_2(x)\rd{x},
	\end{split}\end{equation}
	and the RHS of \eqref{lem_Fou1_1} can be written similarly. By assuming $\cE[\rho_i]<\infty$, we claim that $\cE[\rho_1+\rho_2]<\infty$. Suppose not, then $\int (W_1 * \rho)\rho \rd{x} = \infty$. We can define $W_1^A(x):= \min\{ W_1(x), A \}$ for $A>0$, then
	\begin{equation}
	\lim_{A\to \infty}	\int (W_1^A *\rho) \rho \rd{x} = \infty, 
	\end{equation}
     Notice that since $W_1^A$ is continuous, we have $\lim_{a\to 0+} \int (W_1^A*\psi_a*\rho)(\psi_a*\rho) \rd{x} = 	\int (W_1^A *\rho) \rho \rd{x}$, therefore we show that
   	\begin{equation}
   	\lim_{a\to 0+} \int (W_1*\psi_a*\rho)(\psi_a*\rho) \rd{x} = \infty.
   \end{equation}
	However, $\cE[\rho_i]<\infty$ implies the RHS of \eqref{planc} for $\rho = \rho_1+\rho_2$ is uniformly bounded as $a\to 0+$. Therefore we find a contradiction. Then the conclusion follows from the previous case applied to $\rho_1,\rho_2,\rho_1+\rho_2$. 
\end{proof}

\bibliographystyle{alpha}
\bibliography{set.bib}
\vspace{ 1cm}

\Addresses
\end{document}